\documentclass[1 [leqno,11pt]{amsart}
\usepackage{amssymb, amsmath}
\allowdisplaybreaks[4]
\usepackage{color}

\let\pa=\partial
\let\p=\partial
\let\ve=\varepsilon
\let\e=\varepsilon
\let\al=\alpha
\let\f=\frac
\let\D=\Delta
\let\la=\lambda
\let\wt=\widetilde

\def\no{\noindent}
\def\na{\nabla}
\def\th{\theta}
\def\dive{\mathop{\rm div}\nolimits}

\def\cA{{\mathcal A}}
\def\cB{{\mathcal B}}
\def\cC{{\mathcal C}}

\def\cE{{\mathcal E}}
\def\cF{{\mathcal F}}

\def\cO{{\mathcal O}}
\def\cL{{\mathcal L}}
\def\cM{{\mathcal M}}

\def\PP{\Bbb{P}}

\def\eqdef{\buildrel\hbox{\footnotesize def}\over =}
\def\eqdefa{\buildrel\hbox{\footnotesize def}\over =}
\def\N{\mathop{\mathbb N\kern 0pt}\nolimits}
\def\Q{\mathop{\mathbb Q\kern 0pt}\nolimits}
\def\R{\mathop{\mathbb R\kern 0pt}\nolimits}
\def\Z{\mathop{\mathbb Z\kern 0pt}\nolimits}

\newcommand{\beq}{\begin{equation}}
\newcommand{\eeq}{\end{equation}}
\newcommand{\ben}{\begin{eqnarray}}
\newcommand{\een}{\end{eqnarray}}
\newcommand{\beno}{\begin{eqnarray*}}
\newcommand{\eeno}{\end{eqnarray*}}
\newcommand{\andf}{\quad\hbox{and}\quad}
\newcommand{\with}{\quad\hbox{with}\quad}

\newtheorem{thm}{Theorem}[section]
\newtheorem{lem}{Lemma}[section]
\newtheorem{rmk}{Remark}[section]
\newtheorem{cor}{Corollary}[section]
\newtheorem{prop}{Proposition}[section]

\newcommand{\vv}[1]{\boldsymbol{#1}}

\def\ur{u^r}
\def\ut{u^\theta}
\def\uz{u^z}
\def\vr{v^r}
\def\vt{v^\theta}
\def\vz{v^z}

\numberwithin{equation}{section}

\begin{document}
\title[Global stability of  3-D Navier-Stokes equations]
{On the global stability of large Fourier mode for 3-D Navier-Stokes equations}

\author[Y. Liu]{Yanlin Liu}
\address[Y. Liu]{School of Mathematical Sciences,
Laboratory of Mathematics and Complex Systems,
MOE, Beijing Normal University, 100875 Beijing, China.}
\email{liuyanlin@bnu.edu.cn}

\author[P. Zhang]{Ping Zhang}
\address[P. Zhang]{Academy of Mathematics $\&$ Systems Science
	and Hua Loo-Keng Center for Mathematical Sciences, Chinese Academy of
	Sciences, Beijing 100190, CHINA, and School of Mathematical Sciences,
	University of Chinese Academy of Sciences, Beijing 100049, China.} \email{zp@amss.ac.cn}

\date{\today}

\maketitle

\begin{abstract} 
In this paper, we first prove the global existence of strong solutions to 3-D incompressible
Navier-Stokes equations with solenoidal initial data, which writes in the cylindrical coordinates
is of the form: ${\rm A}(r,z)\cos N\th +{\rm B}(r,z)\sin N\th,$ provided that $N$ is large enough. In particular, we prove that
the  corresponding solution has almost the same frequency $N$ for any positive time.   The main idea
of the proof is first to write the solution in trigonometrical series in  $\th$ variable and estimate the coefficients
separately in some scale-invariant spaces, then we handle a sort of weighted sum of these norms of the coefficients
in order to
close the
{\it a priori} estimate of the solution.
Furthermore, we shall extend the above well-posedness result for initial data which is a linear combination of axisymmetric data without swirl
 and infinitely many large mode trigonometric series in the angular variable.	
\end{abstract}




\setcounter{equation}{0}
\section{Introduction}\label{Secintro}

In this paper, we investigate the global well-posedness of  the following three-dimensional
 incompressible Navier-Stokes equations  with large Fourier mode initial data:
\beq\label{NS}
\left\{
\begin{aligned}
&\p_t u+ u\cdot\nabla u-\Delta u+\nabla P= 0,
\qquad(t,x)\in\R^+\times\R^3,\\
&\dive u=0,\\
& u|_{t=0} =u_{\rm in},
\end{aligned}
\right.
\eeq
where $u$ stands for the velocity field and $P$ for
the scalar pressure function of the fluid, which guarantees the
divergence-free condition of the velocity field.
This system describes the motion of viscous incompressible fluid.

Leray proved in the seminal paper \cite{lerayns} that given any finite energy initial data, \eqref{NS} has a global-in-time weak solution $u$ which verifies the energy inequality for any positive time $t$
$$\frac  12 \|u(t)\|_{L^2}^2 +\int_0^t \|\nabla u(t')\|_{L^2}^2 dt'
\leq \frac 12 \|u_{\rm in}\|_{L^2}^2\,.$$
Similar existence result  in bounded domain was established by Hopf in \cite{Hopf} for \eqref{NS} with Dirichlet boundary condition.
However, the regularity and uniqueness of such weak solutions still remains open except for some special cases (see \cite{ABC,BV19,JS14,JS15,La67} for instance).

In fact, Leray \cite{lerayns} also proved the local existence of a unique strong solution
$u\in L^\infty([0,T]; H^1)$ with $\nabla u\in L^2([0,T],H^1)$ for some positive time $T.$ Whether or not this local smooth solution blows up in finite time is a fundamental open problem in the field of mathematical fluid mechanics (\cite{Fe06,Tao}).

\smallskip

Yet for initial data with some special structures, there are
 affirmative answers to the above question. The most widely studied case is perhaps the axisymmetric solution of \eqref{NS} which  writes
$$u(t,x)=u^r(t,r,z)\vv e_r+u^\theta(t,r,z)\vv e_\theta+u^z(t,r,z)\vv e_z,$$
where $(r,\theta,z)$ denotes the cylindrical coordinates in $\R^3$
so that $x=(r\cos\theta,r\sin\theta,z)$, and
$$\vv e_r\eqdefa (\cos\theta,\sin\theta,0),\ \vv e_\theta\eqdefa (-\sin\theta,\cos\theta,0),\
\vv e_z\eqdefa (0,0,1),\ r\eqdefa \sqrt{x_1^2+x_2^2}.$$
And if in addition $\ut=0$,  we  call such a solution $u$ to be
axisymmetric without swirl.

For the axisymmetric data without swirl, Ladyzhenskaya \cite{La} and independently Ukhovskii
and Yudovich \cite{UY} proved the global existence of strong solution to \eqref{NS}, which keeps this symmetry.
Later Leonardi et al. \cite{LMNP} gave a simpler proof to this result.  Recently Gallay and  $\check{S}$ver\'ak \cite{GS19}
proved that the system \eqref{NS} admits a unique axisymmetric solution without swirl if the initial vorticity is a circular vortex filament with arbitrarily large circulation Reynolds number.

 While the global existence of strong solutions to \eqref{NS} with general axisymmetric data is still open except for the case
  when $u^\th_{\rm in}$ is sufficiently small, with the size of which  depends on the other components of $u_{\rm in}.$
  The strategy of the proof  is mostly based on a
 perturbation argument for the no swirl case.
There are numerous works concerning this problem,
here we only list the references \cite{CL02, Yau1, Yau2, KNSS,   Liu, LZ, Wei}.

\smallskip

In this paper, we shall not limit ourselves in axisymmetric solutions of \eqref{NS},
but this geometric symmetry will play a crucial role in the proof of our main result below.
Let us write any smooth solution $u$ of \eqref{NS} in cylindrical coordinates as
\begin{equation}\label{ucylin}
u(t,x)=u^r(t,r,\th,z)\vv e_r+u^\theta(t,r,\th,z)\vv e_\theta+u^z(t,r,\th,z)\vv e_z.
\end{equation}
Noticing that in cylindrical coordinates, one has
\beq\label{coordinate}
\na_x=\vv e_r\p_r+\vv e_\theta\f{\pa_\th}{r}+\vv e_z\p_z,
\quad \D_x=\p_r^2+\f{\p_r}{r}+\f{\p_\theta^2}{r^2}+\p_z^2,
\eeq
so that we may equivalently reformulate the system \eqref{NS} in cylindrical coordinates as
\begin{equation}\label{eqtu}
\left\{
\begin{aligned}
& D_t u^r-\bigl(\pa_r^2+\pa_z^2+\f{\pa_r}{r}
+\f{\pa_\th^2}{r^2}-\f1{r^2}\bigr)u^r
-\f{\left(u^\theta\right)^2}{r}
+2\f{\pa_\th u^\theta}{r^2}+\pa_r P=0,\\
& D_t u^\theta-\bigl(\pa_r^2+\pa_z^2+\f{\pa_r}{r}
+\f{\pa_\th^2}{r^2}-\f1{r^2}\bigr)u^\theta
+\f{u^r u^\theta}{r}-2\f{\pa_\th u^r}{r^2}
+\f{\pa_\th P}{r}=0,\\
& D_t u^z-\bigl(\pa_r^2+\pa_z^2+\f{\pa_r}{r}
+\f{\pa_\th^2}{r^2}\bigr) u^z+\pa_z P=0,\\
& \pa_r u^r+\frac{u^r}{r}+\pa_z u^z
+\f{\pa_\th u^\theta}{r}=0,\\
& (u^r,u^\theta,u^z)|_{t=0} =(u_{\rm in}^r,u^\theta_{\rm in},u^z_{\rm in}),
\end{aligned}
\right.
\end{equation}
where $D_t\eqdef\pa_t+ u\cdot\nabla_x=\pa_t+\bigl(u^r\pa_r+\f{u^\theta}{r}\pa_\th
+u^z\pa_z\bigr)$ denotes the material derivative.

\smallskip

Before proceeding, let us recall another kind of interesting initial data which varies slowly in one direction, that is,
 data of the form:
\begin{equation}\label{slowvari}
u_{\rm in}^\varepsilon(x)=(v^1_{\rm in} +\varepsilon w^1_{\rm in},v^2_{\rm in} +\varepsilon w^2_{\rm in},w^3_{\rm in})(x_1,x_2,\varepsilon x_3),\with
\dive v_{\rm in}=\dive w_{\rm in}=0,
\end{equation}
Chemin and Gallagher \cite{CG10} proved the global well-posedness of the system \eqref{NS} for  $\ve$ being sufficiently small (see also \cite{CGZ, CZ15}). In general, we proved in \cite{LZ4} that if $\pa_3u_{\rm in}$ is small enough in some critical space,
then \eqref{NS} has a unique global solution. Whereas if we replace $\pa_3$ by some curved derivative, the corresponding problem becomes  extremely difficult. Indeed even if
$\pa_\th \ur_{\rm in}=\pa_\th \ut_{\rm in}=\pa_\th \uz_{\rm in}=0,$
 which corresponds to the axisymmetric initial data,
the global existence of strong solution to \eqref{NS} remains open.
The first author and Xu \cite{LX} proved  the global well-posedness of \eqref{eqtu} provided that both $u^\th_{\rm in}$
and $(\pa_\th \ur_{\rm in},\pa_\th \ut_{\rm in},\pa_\th \uz_{\rm in})$ are sufficiently small.
The main idea of \cite{LX} is to expand the initial data $u_{\rm in}$ into Fourier series in $\th$ variable:
\begin{equation}\label{initialFourier}
	\left\{
	\begin{split}
& \ur_{\rm in}(x)=\bar u^r_{\rm in}(r,z)
+\sum_{k=1}^\infty\Bigl(a^r_k(r,z)\cos k\th
+b^r_k(r,z)\sin k\th\Bigr),\\
& \ut_{\rm in}(x)=\bar u^\th_{\rm in}(r,z)
+\sum_{k=1}^\infty k^{-1}\Bigl(a^\th_k(r,z)\cos k\th
+b^\th_k(r,z)\sin k\th\Bigr),\\
& \uz_{\rm in}(x)=\bar u^z_{\rm in}(r,z)
+\sum_{k=1}^\infty\Bigl(a^z_k(r,z)\cos k\th
+b^z_k(r,z)\sin k\th\Bigr),
	\end{split}
	\right.
\end{equation}
where the Fourier coefficients are regular enough and decay rapidly as $k\rightarrow\infty$,
and satisfy
\begin{equation}\label{1.16}
	\pa_r\bar u^r_{{\rm in}}+\f{\bar u^r_{{\rm in}}}{r}+\pa_z\bar u^z_{{\rm in}}=0,
\quad\pa_r a^r_k+\f{a^r_k}{r}+\pa_z a^z_k
	+\f{b^\th_k}{r}=0,\quad\pa_r b^r_k
	+\f{b^r_k}{r}+\pa_z b^z_k-\f{a^\th_k}{r}=0.
\end{equation}
It is easy to verify that the constraint \eqref{1.16}
guarantees that $\dive u_{\rm in}=0$. We remark that the smallness condition in \cite{LX} corresponds to
the smallness assumption on the Fourier coefficients $a_k$ and $b_k$ in \eqref{initialFourier}. We shall prove in this paper that there is unnecessary for
smallness condition on  $a_k$ and $b_k$ if the initial data concentrate on high enough frequency.

Our first result states as follows:

\begin{thm}\label{thm1}
{\sl Let \begin{equation}\label{defcM}
\cM\eqdef\Bigl\{f(r,z):\ \|f\|_{\cM}\eqdefa
\bigl\|r^{\f12}f\bigr\|_{L^6}
+\|f\|_{L^2}+\bigl\|(\pa_rf,\pa_zf)\bigr\|_{L^2}
+\bigl\|r^{-1}f\bigr\|_{L^2}<\infty\Bigr\}.
\end{equation} Here all the $L^p$ norm denotes that of $L^p(\R^3)$.
Then for any $(a^r,a^\th,a^z,b^r,b^\th,b^z)\in\cM$ satisfying
$$\pa_r a^r+\f{a^r}{r}+\pa_z a^z
	+\f{b^\th}{r}=0,\quad\pa_r b^r
	+\f{b^r}{r}+\pa_z b^z-\f{a^\th}{r}=0,$$
there exists an integer $N_0>0$ depending only on $\|(a^r,a^\th,a^z,b^r,b^\th,b^z)\|_{\cM}$ such that for any integer $N>N_0$,
the system \eqref{eqtu} with
initial data
\begin{equation}\label{initialN}
	\left\{
	\begin{split}
& \ur_{\rm in}(x)=a^r(r,z)\cos N\th+b^r(r,z)\sin N\th,\\
& \ut_{\rm in}(x)=N^{-1}\bigl(a^\th(r,z)\cos N\th
+b^\th(r,z)\sin N\th\bigr),\\
& \uz_{\rm in}(x)=a^z(r,z)\cos N\th+b^z(r,z)\sin N\th.
	\end{split}
	\right.
\end{equation}
has a unique global strong solution, which can be expanded into the following Fourier series for any positive time $t:$
\begin{equation}\label{solgeneral}
\left\{
\begin{split}
\ur(t,x)=&\ur_{0}(t,r,z)
+\sum_{k=1}^\infty\Bigl(\ur_{k}(t,r,z)\cos kN\th
+\vr_{k}(t,r,z)\sin kN\th\Bigr),\\
\ut(t,x)=&\ut_{0}(t,r,z)+\sum_{k=1}^\infty
\Bigl(\ut_{k}(t,r,z)\cos kN\th
+\vt_{k}(t,r,z)\sin kN\th\Bigr),\\
\uz(t,x)=&\uz_{0}(t,r,z)
+\sum_{k=1}^\infty\Bigl(\uz_{k}(t,r,z)\cos kN\th
+\vz_{k}(t,r,z)\sin kN\th\Bigr).
\end{split}
\right.
\end{equation}
Furthermore,  for any $p\in ]5,6[$ and $\al_p,\beta_p$ satisfying
\begin{equation}\label{condialbeta}
 1<\beta_p<\f{p-1}4,\andf
0<\al_p<\f{p-3-2\beta_p}4,
\end{equation}
$\varpi_k\eqdef\bigl(\ur_k,\vr_k,\uz_k,\vz_k, \sqrt{kN}\ut_k,\sqrt{kN}\vt_k\bigr)$ satisfies
\beq \label{S1eqw}\|u_0\|_{L^\infty(\R^+;L^3)}\leq C_{\rm in}N^{-\f15},
\andf\sum_{k=2}^\infty (kN)^{2\beta_p}
\bigl\|r^{1-\f3p}\varpi_{k}\bigr\|_{L^\infty(\R^+;L^p)}^p
\leq C_{\rm in} N^{-2\al_p},\eeq
where $C_{\rm in}$ designates some positive constant depending only
on $\|(a^r,a^\th,a^z,b^r,b^\th,b^z)\|_{\cM}$.
}\end{thm}

\begin{rmk}
\begin{enumerate}
\item[(1)] Leray  emphasized in  \cite{lerayns} that all the reasonable estimates of the solutions to \eqref{NS} should be
scale-invariant under the following scaling transformation:
\beq \label{S1eq2w} u_\lambda(t,x)=\lambda u(\lambda^2 t,\lambda x) \andf u_{0,\lambda}(x)=\lambda u_0(\lambda x).\eeq
 We notice that for any solution $u$ of \eqref{NS} on $[0,T],$
 $u_\la$ is also a solution of \eqref{NS} on $[0,T/\la^2].$ It is easy to observe that the norms in \eqref{S1eqw}: $\|\cdot\|_{L^\infty(\R^+;L^3)}$ and $\bigl\|r^{1-\f3p}\ \cdot\bigr\|_{L^\infty(\R^+;L^p)}$ are scaling invariant with respect to the
 scaling transformation \eqref{S1eq2w}.

\item[(2)] In Theorem \ref{thm1}, we not only prove the global existence of strong solutions to \eqref{NS}, but also show that
this solution almost concentrates on the frequency  $N$ for any positive time (see \eqref{S1eqw}). In section \ref{secpre} and the proof of Lemma \ref{S2lem2},
 we shall also present the estimate to the Fourier coefficients of the pressure function $P,$ which will be one of the crucial step to prove  Theorem \ref{thm1}. We believe that the estimate of the pressure function $P$ will be helpful to study the global well-posedness of \eqref{eqtu}
 with general  data \eqref{initialFourier}.

 \end{enumerate}
\end{rmk}

\begin{rmk}
\begin{enumerate}
\item[(1)] We remark that axisymmetric data without swirl  corresponds to the case $N=0$ in \eqref{initialN}. While Theorem \ref{thm1} claims the global well-posedness of \eqref{eqtu} with large Fourier mode data.
However, for the intermediate case when $N\in \left\{ 1, \cdots, N_0 \right\},$  the global well-posedness of the system \eqref{eqtu}
with initial data \eqref{initialN}
leaves open.

\item[(2)]
In the recent very interesting paper \cite{ABC}, Albritton, Bru\'e and Colombo
 proved the non-uniqueness of Leray weak solutions of a forced Navier-Stokes equations.
  Elgindi, Chen and Hou \cite{CH21,E21,EJ19} successfully constructed
finite time blow-up solutions for 3-D Euler equations and some relative models.
We  mention that
the axisymmetric  structure of Navier-Stokes (or Euler) equations plays a fundamental role in all these works.

\item[(3)]
It seems that our method of proving Theorem \ref{thm1}  can  be employed to investigate the global well-posedness
 of the system \eqref{NS} in  such axisymmetric domain as the infinite channel
$\{x\in\R^3:\, r\leq1,z\in\R\},$ which will be the aim of our forthcoming paper.

\item[(4)]
We believe that a comprehensive study on the initial data \eqref{initialFourier} with some other special structures can shed lights into the global well-posedness of Navier-Stokes equations with general data.
\end{enumerate}
\end{rmk}

 By combining Theorem \ref{thm1} with the classical  result of  axisymmetric solution without swirl
for  Navier-Stokes equations, we have the following more general result:

\begin{thm}\label{thm2}
{\sl {Let $\cM$ be defined by \eqref{defcM}. Let $\bar u^r_{{\rm in}}(r,z)\vv e_r
+\bar u^z_{{\rm in}}(r,z)\vv e_z\in H^2(\R^3)$ and
 $\Upsilon_k
\eqdef \left(a^r_k,a^\th_k,a^z_k,b^r_k,b^\th_k,b^z_k\right)\in \cM,$ which satisfy \eqref{1.16}
 and
\begin{equation}\label{condsumk} 
\sum_{k=1}^\infty
\|\Upsilon_k\|_{L^3}^{\f32}<\infty.
\end{equation} Then
there exists an integer $N_0>0$
such that for any integers: $N_1,~N_2,\cdots>N_0$,
the system \eqref{eqtu} has a unique global strong solution
with initial data
\begin{equation}\label{initialsum}
	\left\{
	\begin{split}
& \ur_{\rm in}(x)=\bar u^r_{\rm in}(r,z)
+\sum_{k=1}^\infty\Bigl(a^r_k(r,z)\cos N_k\th
+b^r_k(r,z)\sin N_k\th\Bigr),\\
& \ut_{\rm in}(x)=\sum_{k=1}^\infty N_k^{-1}\Bigl(a^\th_k(r,z)\cos N_k\th
+b^\th_k(r,z)\sin N_k\th\Bigr),\\
& \uz_{\rm in}(x)=\bar u^z_{\rm in}(r,z)
+\sum_{k=1}^\infty\Bigl(a^z_k(r,z)\cos N_k\th
+b^z_k(r,z)\sin N_k\th\Bigr).
	\end{split}
	\right.
\end{equation}
}}\end{thm}

\begin{rmk}
\begin{enumerate}
\item[(1)]  The summable condition \eqref{condsumk} holds automatically if $\{\Upsilon_k\}$ is a finite sequence. And one may check the proof of
Theorem \ref{thm2} in Section \ref{secthm2} for the determination of  the integer $N_0.$

\item[(2)]  Together with Theorem 2.1 of \cite{GIP03}, our proof of Theorem \ref{thm2} implies that
the system \eqref{NS} still have a unique global solution even if
we replace $\left(\bar u^r_{\rm in}(r,z),0,\bar u^z_{\rm in}(r,z)\right)$
in \eqref{initialsum} by  any  initial data which generates a unique global strong solution of \eqref{NS}.
\end{enumerate}
\end{rmk}

Let us end this section with some notations that we shall use throughout this paper.

\noindent{\bf Notations:}
We always denote $C$ to be a uniform constant
which may vary from line to line.
$a\lesssim b$ means that $a\leq Cb$,  and $a\thicksim b$
means that both $a\leq Cb$ and $b\leq Ca$ hold,  while $a\lesssim_\e b$ means that $a\leq C_\e b$ for some constant $C_\e$ depending on $\e.$
For a Banach space $B$, we shall use the shorthand $L^p_T(B)$ for $\bigl\|\|\cdot\|_B\bigr\|_{L^p(0,T)}$. Finally
all the $L^p$ norm designates the $L^p(\R^3)$ norm so that
$\|f\|_{L^p}=\Bigl(\int_0^\infty\int_0^{2\pi}\int_{-\infty}^\infty |f(r,\theta,z)|^p \,rdr\,d\theta\,dz\Bigr)^{\f1p}.$

\setcounter{equation}{0}
\section{Ideas and structure of the proof}\label{secidea}

The purpose of this section is to sketch the main idea used in the proof of Theorem \ref{thm1}.
Before proceeding, we first present some symmetric properties of
Navier-Stokes equations in cylindrical coordinates besides the well-known axisymmetric one.
The persistence of these symmetries deeply reflects some nonlinear structures
of the system, and is of independent interest.

\begin{thm}\label{thm3}
{\sl {\rm(i)} Let
 $\ur_{\rm in}$ and $\uz_{\rm in}$   be even and  $\ut_{\rm in}$ be odd  in $\th$ variable.
Then this odevity will persist as long as the
strong solution $u$ of \eqref{eqtu} exists.

\noindent{\rm(ii)} Let $(\ur_{\rm in}, \ut_{\rm in}, \uz_{\rm in})$
be $\f{2\pi}{m}$-periodic in  $\th$ variable for some $m\in\N^+.$
Then the strong solution $u$ of \eqref{eqtu}   remains
$\f{2\pi}{m}$-periodic in $\th$ variable as long as it exists.
}\end{thm}

\begin{rmk}
{\rm Theorem \ref{thm3}.(ii)} can be regarded as a generalization of
the classical result that the strong solution of \eqref{NS} with
axisymmetric initial data remains to be axisymmetric.
Indeed, if $u_{\rm in}$ is axisymmetric, which means $u_{\rm in}$  is
$\f{2\pi}{m}$-periodic in $\th$ variable for every $m\in\N^+$,
then {\rm Theorem \ref{thm3}.(ii)} ensures that the corresponding strong solution $u$ of \eqref{NS} is also $\f{2\pi}{m}$-periodic in $\th$
variable for every $m\in\N^+$,
which implies that $u$ is axisymmetric.
\end{rmk}

\begin{proof}
[{\bf Proof of Theorem \ref{thm3}}]
(i)\ Let $u=(\ur, \ut, \uz)$ be the  unique local strong solution of \eqref{eqtu} on $[0,T].$ We introduce $\bar{u}=(\bar{u}^r, \bar{u}^\th, \bar{u}^z)$ via
\beq \label{urtz}\begin{split} \bar{u}^r(t,r,\th,z)\eqdefa &\ur(t,r,-\th,z),\quad \bar{u}^\th(t,r,\th,z)\eqdefa-\ut(t,r,-\th,z),
\andf\\
 &\bar{u}^z(t,r,\th,z)\eqdefa\uz(t,r,-\th,z).\end{split}
 \eeq
Notice that under the odevity assumption of $(\ur_{\rm in}, \ut_{\rm in}, \uz_{\rm in})$ in {\rm Theorem \ref{thm3}.(i)}, one has
$$\ur_{\rm in}(r,-\th,z)=\ur_{\rm in}(r,\th,z),
\quad u^\th_{\rm in}(r,-\th,z)=-\ut_{\rm in}(r,\th,z),
\quad u^z_{\rm in}(r,-\th,z)=\uz_{\rm in}(r,-\th,z).$$
Hence $\bar{u}=(\bar{u}^r, \bar{u}^\th, \bar{u}^z) $ defined by \eqref{urtz}
verifies the system \eqref{eqtu}  on $[0,T]$.
Then we deduce from the uniqueness of strong solution to  Navier-Stokes equations
that $u=\bar{u}$ on $[0,T],$ which means that $\ur$ and $\uz$ are even in  $\th$ variable,
and $\ut$ is odd in $\th$ variable.

\no (ii)\ Let $u=(\ur, \ut, \uz)$ be the  unique local strong solution of \eqref{eqtu} on $[0,T].$
We introduce $w$ to be
$$w(t,r,\th,z)=(w^r,w^\th,w^z)(t,r,\th,z)\eqdefa(\ur,\ut,\uz)(t,r,\th+\f{2\pi}{m},z).$$
Since $(\ur_{\rm in}, \ut_{\rm in}, \uz_{\rm in})$
is $\f{2\pi}{m}$-periodic in  $\th$ variable,
$u_{\rm in}$  satisfies
$$(\ur_{\rm in},u^\th_{\rm in},u^z_{\rm in})(r,\th+\f{2\pi}{m},z)
=(\ur_{\rm in},u^\th_{\rm in},u^z_{\rm in})(r,\th,z).$$
Hence $w$ also satisfies the system \eqref{eqtu} with initial data $u_{\rm in}$.
Then it follows from the uniqueness of strong solution to \eqref{eqtu}
that $u=w$ on $[0,T],$ which implies that $\ur,~\ut$ and $\uz$ keep to be
$\f{2\pi}{m}$-periodic in $\th$ variable. This completes the proof of Theorem \ref{thm3}.
\end{proof}

Now let us turn to the proof of Theorem \ref{thm1}.
In order to simplify the notations and the calculations,
so that we can highlight the main new ideas of this paper,
we choose first to consider initial data of the
following special form of \eqref{initialN}:
\begin{equation}\label{initialNodd}
u^{(1)}_{\rm in}(x)=a^r(r,z)\cos (N\th) \vv e_r
+N^{-1}b^\th(r,z)\sin (N\th)\vv e_\th
+a^z(r,z)\cos (N\th)\vv e_z.
\end{equation}
For the general case, we shall outline its proof at the end of this section.

For initial data given by \eqref{initialNodd}, the system \eqref{eqtu} has  a unique  strong solution   $(\ur,\ut,\uz)$
on $[0,T^\ast[$ for some $T^\ast\in ]0,\infty].$
For any $t<T^\ast,$ we deduce from Theorem \ref{thm3}
that $(\ur,\ut,\uz)$ can be expanded into the following Fourier series:
\begin{equation}\label{solexpan}
\left\{
\begin{split}
\ur(t,x)=&\ur_{0}(t,r,z)
+\sum_{k=1}^\infty\ur_{k}(t,r,z)\cos kN\th,\\
\ut(t,x)=&\sum_{k=1}^\infty\vt_{k}(t,r,z)\sin kN\th,\\
\uz(t,x)=&\uz_{0}(t,r,z)
+\sum_{k=1}^\infty\uz_{k}(t,r,z)\cos kN\th.
\end{split}
\right.
\end{equation}

Correspondingly, in view of the system \eqref{eqtu}
and  \eqref{solexpan}, we observe
 that the pressure function
$P$ must be even in  $\th$ variable. So that we write
\begin{equation}\label{Pexpan}
P(t,x)=P_{0}(t,r,z)
+\sum_{k=1}^\infty P_{k}(t,r,z)\cos kN\th.
\end{equation}
One should keep in mind that we always use the subscript $k$ to denote the $kN$-th Fourier coefficient for notational simplification.
And in what follows, we always denote
\beq \label{S2notion}
\begin{split}
&\wt u\eqdef\ur\vv e_r+\uz\vv e_z,
\quad D_{0,t}\eqdefa \p_t+u_0^r\pa_r+u_0^z\pa_z,\\
u_0\eqdefa \ur_0\vv e_r&+\uz_0\vv e_z,
\quad\wt u_k\eqdef\ur_k\vv e_r+\uz_k\vv e_z,\andf
\wt\nabla\eqdef\vv e_r\pa_r+\vv e_z\pa_z.
\end{split}
\eeq

By substituting  \eqref{solexpan} and \eqref{Pexpan}
into the system \eqref{eqtu} and using the identities
\begin{align*}
&2\cos k_1\th\cdot\cos k_2\th=\cos (k_1+k_2)\th+\cos (k_1-k_2)\th,\\
&2\sin k_1\th\cdot\sin k_2\th=\cos (k_1-k_2)\th-\cos (k_1+k_2)\th,\\
&2\sin k_1\th\cdot\cos k_2\th=\sin (k_1+k_2)\th+\sin (k_1-k_2)\th,
\end{align*}
and finally comparing the Fourier coefficients of the resulting equations, we find that
\begin{itemize}
\item
$(\ur_0,\uz_0)$ verifies
\end{itemize}
\begin{equation}\label{eqtu0}
\left\{
\begin{split}
&D_{0,t}\ur_{0}
-\bigl(\D-\frac{1}{r^2}\bigr)\ur_{0}+\pa_r P_{0}
=-\f12\sum_{k=1}^\infty\Bigl(\wt u_k\cdot\wt\nabla\ur_{k}
-\f{|\vt_{k}|^2}r-\f{kN}r\vt_k\ur_k\Bigr),\\
&D_{0,t}\uz_{0}
-\D\uz_{0}+\pa_z P_{0}
=-\f12\sum_{k=1}^\infty\Bigl(\wt u_k\cdot\wt\nabla\uz_{k}
-\f{kN}r\vt_k\uz_k\Bigr),\\
& \pa_r\ur_{0}+\frac{\ur_{0}}r+\pa_z\uz_{0}=0,\\
& \ur_{0}|_{t=0}=\uz_{0}|_{t=0}=0.
\end{split}
\right.
\end{equation}

\begin{itemize}
\item
 $\left(\ur_{k},~\vt_{k},~\uz_{k}\right)$ for $k\in\N^+$ satisfies
 \end{itemize}
\begin{equation}\label{eqtuk}
\left\{
\begin{split}
&D_{0,t}\ur_{k}
+\wt u_k\cdot\wt\nabla\ur_0
-\bigl(\D-\frac{1+k^2N^2}{r^2}\bigr)\ur_{k}
+2\f{kN}{r^2}\vt_{k}+\pa_r P_{k}=F^r_{k},\\
&D_{0,t}\vt_{k}
+\f{\ur_0\vt_{k}}r-\bigl(\D-\frac{1+k^2N^2}{r^2}\bigr)\vt_{k}
+2\f{kN}{r^2}\ur_{k}-\f{kN}r P_{k}=F^\th_{k},\\
&D_{0,t}\uz_{k}
+\wt u_k\cdot\wt\nabla\uz_0-\bigl(\D-\frac{k^2N^2}{r^2}\bigr)
\uz_{k}+\pa_z P_{k}=F^z_{k},\\
& \pa_r\ur_{k}+\frac{\ur_k}r+\pa_z\uz_{k}
+kN\f{\vt_{k}}r=0,\\
& (\ur_{1},\vt_{1},\uz_{1})|_{t=0}=(a^r,N^{-1}b^\th,a^z)
\andf(\ur_{k},\vt_{k},\uz_{k})|_{t=0}=0
~\text{ for }~k\geq2,
\end{split}
\right.
\end{equation}
where the external force terms, $F^r_{k},~F^\th_{k},~F^z_{k},$ are given respectively by
\beq\label{eqtukq}
\begin{split}
F^r_k\eqdefa&-\f12\sum_{\substack{
|k_1-k_2|=k\\
\text{ or }k_1+k_2=k}}
\wt u_{k_1}\cdot\wt\nabla\ur_{k_2}
+\f{1}{2r}\Bigl(\sum_{|k_1-k_2|=k}
-\sum_{k_1+k_2=k}\Bigr)
\bigl(\vt_{k_1}\vt_{k_2}+k_2 N \vt_{k_1}\ur_{k_2}\bigr),\\
F^\th_k\eqdefa &\f12\Bigl(\sum_{k_2-k_1=k}
-\sum_{\substack{k_1+k_2=k\\
\text{ or }k_1-k_2=k}}\Bigr)
\Bigl(\wt u_{k_2}\cdot\wt\nabla\vt_{k_1}
+\f1{r}\bigl(\vt_{k_1}\ur_{k_2}+k_2N\vt_{k_1}\vt_{k_2}\bigr)\Bigr),\\
F^z_k\eqdefa&-\f12\sum_{\substack{|k_1-k_2|=k\\
\text{ or }k_1+k_2=k}}
\wt u_{k_1}\cdot\wt\nabla\uz_{k_2}+\f{N}{2r}\Bigl(\sum_{|k_1-k_2|=k}
-\sum_{k_1+k_2=k}\Bigr)
{k_2\vt_{k_1}\uz_{k_2}},
\end{split} \eeq
and $k_1,k_2\in\N^+$ in the above summations.

In Section \ref{secEp}, we shall prove the following {\it a priori} estimate:

\begin{prop}\label{propEp}
{\sl  Let
$p\in ]5,6[,$ $\al_p$ and $\beta_p$ satisfy \eqref{condialbeta}.
Let $(\ur_0,\uz_0)$ and $\left(\ur_{k},~\vt_{k},~\uz_{k}\right)$ be smooth enough solutions of
 \eqref{eqtu0} and \eqref{eqtuk} respectively on $[0,T]$. We denote $E_{p,\al_p,\beta_p}(t)\eqdefa E_{p,\al_p,\beta_p}^{r,z}(t)
+E_{p,\al_p,\beta_p}^\th(t)$ with
\begin{equation}\begin{split}\label{defEpr}
E_{p,\al_p,\beta_p}^{r,z}(t)\eqdef
&N^{-2\al_p}\Bigl(\bigl\|r^{\f{p-3}2}|\wt u_1|^{\f p2}\bigr\|_{L^\infty_t(L^2)}^2
+\bigl\|r^{\f{p-3}2}|{\wt\nabla}\wt u_1|
|\wt u_1|^{\f p2-1}\bigr\|_{L^2_t(L^2)}^2\\
&+N^{2}\bigl\|r^{\f{p-5}2}|\wt u_1|^{\f p2}\bigr\|_{L^2_t(L^2)}^2\Bigr)
+\sum_{k=2}^\infty(kN)^{2\beta_p}\Bigl(
\bigl\|r^{\f{p-3}2}|\wt u_{k}|^{\f p2}\bigr\|_{L^\infty_t(L^2)}^2\\
&+\bigl\|r^{\f{p-3}2}|{\wt\nabla}\wt u_k|
|\wt u_k|^{\f p2-1}\bigr\|_{L^2_t(L^2)}^2
+(kN)^{2}\bigl\|r^{\f{p-5}2}|\wt u_{k}|^{\f p2}\bigr\|_{L^2_t(L^2)}^2\Bigr),
\end{split}\end{equation}
and
 \begin{equation}\begin{split}\label{defEpth}
E_{p,\al_p,\beta_p}^\th(t)\eqdef&N^{\f p2-2\al_p}
\Bigl(\bigl\|r^{\f{p-3}2}|\vt_1|^{\f p2}\bigr\|
_{L^\infty_t(L^2)}^2
+\bigl\|r^{\f{p-3}2}|{\wt\nabla}\vt_1|
|\vt_1|^{\f p2-1}\bigr\|_{L^2_t(L^2)}^2\\
&+N^2\bigl\|r^{\f{p-5}2}|\vt_1|^{\f p2}\bigr\|_{L^2_t(L^2)}^2\Bigr)+\sum_{k=2}^\infty(kN)^{\f p2+2\beta_p}\Bigl(
\bigl\|r^{\f{p-3}2}|\vt_{k}|^{\f p2}\bigr\|
_{L^\infty_t(L^2)}^2\\
&+\bigl\|r^{\f{p-3}2}|{\wt\nabla}\vt_k|
|\vt_k|^{\f p2-1}\bigr\|_{L^2_t(L^2)}^2
+(kN)^2\bigl\|r^{\f{p-5}2}|\vt_{k}|^{\f p2}\bigr\|_{L^2_t(L^2)}^2\Bigr).
\end{split}\end{equation}
 Then there exists large enough integer $N_0$ so that  for $N>N_0$ and $t\leq T,$ there holds
\begin{equation}\begin{split}\label{aprioriEp}
\bigl(\f12-C\|u_0\|_{L^\infty_t(L^3)}\bigr)E_{p,\al_p,\beta_p}(t)
\leq &E_{p,\al_p,\beta_p}(0)
+CE_{p,\al_p,\beta_p}^{1+\f1p}(t)
+CN^{-2}D^{\f{p+1}2}(t)\\
&+CN^{-p}\Bigl(\|u_0\|_{L^\infty_t(L^3)}
+N^{-\f{p}{2(p-5)}}\Bigr)D^{\f{p}2}(t),
\end{split}\end{equation}
where $D(t)\eqdef\sum_{k=1}^\infty(kN)^2
\bigl\|r^{-\f32}\wt u_k\bigr\|_{L^2_t(L^2)}^2$.
}\end{prop}

\begin{rmk} In \cite{lerayns}, Leray emphasized that the  energy method and the scale-invariant
property are very important
in the study of Navier-Stokes equations. Here we shall use scaled $L^p$ type energy estimate to prove Proposition \ref{propEp}.
Moreover,
we observe that all the terms in the energy $E_{p,\al_p,\beta_p}(t)$
given by \eqref{defEpr} and \eqref{defEpth}
are scale-invariant with respect to the scaling transformation \eqref{S1eq2w}.
\end{rmk}

In order to close the estimate \eqref{aprioriEp}, we  need to deal with
the  $L^\infty_t(L^3)$ estimate  of $u_0$. However,
due to the operator $\cL_0^{-1}$ behaves much worse than $\cL_{kN}^{-1}$
(see Remark \ref{rmkLm}), where
\begin{equation}\label{defcLm}
\cL_m\eqdef -\bigl(\pa_r^2+\f{\pa_r}r
+\pa_z^2\bigr)+\f{m^2}{r^2},\quad m=0,1,2,\cdots,
\end{equation}
it seems very difficult to estimate $P_0$ through $E_{p,\al_p,\beta_p}$. In Section \ref{secpre}, we
shall investigate the inverse operator of $\cL_m,$ namely, Propositions \ref{lemLm}-\ref{lemLmweightedRiesz}, which will be
 one of the most crucial step to prove Proposition \ref{propEp}.

To overcome the above difficulty,  we write the system \eqref{eqtu0} for $u_0$ in  the Euclidean coordinates:
\begin{equation}\label{eqtuoEuclidean}
\left\{
\begin{split}
&\pa_t u_0+u_0\cdot\nabla u_{0}
-\D u_{0}+\nabla P_{0}
=-\f{\cos\th}2\sum_{k=1}^\infty
\Bigl(\wt u_k\cdot\wt\nabla\ur_{k}-\f{|\vt_{k}|^2}r
-\f{kN}r\vt_k\ur_k\Bigr)\vv e_1\\
&\quad-\f{\sin\th}2\sum_{k=1}^\infty
\Bigl(\wt u_k\cdot\wt\nabla\ur_{k}-\f{|\vt_{k}|^2}r
-\f{kN}r\vt_k\ur_k\Bigr)\vv e_2
-\f12\sum_{k=1}^\infty\Bigl(\wt u_k\cdot\wt\nabla\uz_{k}
-\f{kN}r\vt_k\uz_k\Bigr)\vv e_3,\\
& \dive u_{0}=0,\\
& u_{0}|_{t=0}=0.
\end{split}
\right.
\end{equation}
Then we can handle  the estimate of the pressure function $P_0$ as that of the classical Navier-Stokes equations in
order to
derive the $L^\infty_t(L^3)$ estimate of $u_0.$

\begin{rmk}\label{rmku0}
In the system for $kN$-th
Fourier coefficients \eqref{eqtuk}, the equation for $\vv e_\th$
component contains a pressure term $-\f{kN}r P_{k}$.
Hence if we rewrite the system \eqref{eqtuk} in Euclidean coordinates,
the pressure term is not purely $\nabla P_k$.

Fortunately, the term $-\f{kN}r P_{k}$ vanishes when $k=0$.
This is a subtle  fact that there is no such pressure term
appearing in the $\vt_0$ equation, and thus the pressure term in \eqref{eqtuoEuclidean}
is the standard $\nabla P_0$ in Euclidean coordinates
even for the general initial data \eqref{initialN}.
\end{rmk}

With Proposition \ref{propEp} at hand, we shall handle the {\it a priori} estimate of the solution $u$ of \eqref{eqtu}
through  the estimate of  its Fourier coefficients
by using Plancherel's identity and Hausdorff-Young's inequality.
Finally we close the global {\it a priori} estimates by a standard
 continuity argument.
Precisely, we shall prove the following proposition in Section \ref{secH1}:

\begin{prop}\label{propH1}
{\sl There exists large enough integer $N_0$ so that  for $N>N_0,$
the system \eqref{eqtu} with initial data
\eqref{initialNodd} admits a unique global strong solution
$u\in L^\infty(\R^+,H^1)$ with $\nabla u\in L^2(\R^+,H^1),$
which satisfies
\begin{equation}\label{estipropH1}
\|\nabla u\|_{L^\infty(\R^+;L^2)}^2
+\|\D u\|_{L^2(\R^+;L^2)}^2\leq 2\|\nabla u_{\rm in}\|_{L^2}^2,
\andf\|u\|_{L^5(\R^+;L^5)}\leq C_{\rm in}N^{-\f15},
\end{equation}
where $C_{\rm in}$ is some positive constant depending only
on $\|(a^r,b^\th,a^z)\|_{\cM}$.
}\end{prop}

Next we sketch the proof of Theorem \ref{thm1}.

\begin{proof}[{\bf Proof of Theorem \ref{thm1}}]
It is obvious that Proposition \ref{propH1} proves the existence part  of Theorem \ref{thm1}
when the initial data are given by \eqref{initialNodd}. The first inequality of \eqref{S1eqw} in this case
follows from Lemma  \ref{propuL2}. While it follows from \eqref{4.3} and the proof of Proposition \ref{propH1}
that
\begin{equation*}
E_{p,\al_p,\beta_p}(t)
\leq C_{\rm in}N^{-2\al_p},\quad\forall\ t>0,
\end{equation*} for any $\al_p$ satisfying \eqref{condialbeta},
which together with \eqref{defEpr} and  \eqref{defEpth} ensures the second inequality of \eqref{S1eqw} in this special case.

Now let us consider the system \eqref{eqtu} with the initial data given by \eqref{initialN}. We observe  that although there is no symmetry of odevity
in this case, the solution $u$ still keeps $\f{2\pi}{N}$-periodic in $\th$ variable
thanks to Theorem \ref{thm3}{\rm.(ii)}.
Namely, we can write the unique local strong solution as follows:
\begin{equation}\label{solgeneral}
\left\{
\begin{split}
\ur(t,x)=&\ur_{0}(t,r,z)
+\sum_{k=1}^\infty\Bigl(\ur_{k}(t,r,z)\cos kN\th
+\vr_{k}(t,r,z)\sin kN\th\Bigr),\\
\ut(t,x)=&\ut_{0}(t,r,z)+\sum_{k=1}^\infty
\Bigl(\ut_{k}(t,r,z)\cos kN\th
+\vt_{k}(t,r,z)\sin kN\th\Bigr),\\
\uz(t,x)=&\uz_{0}(t,r,z)
+\sum_{k=1}^\infty\Bigl(\uz_{k}(t,r,z)\cos kN\th
+\vz_{k}(t,r,z)\sin kN\th\Bigr).
\end{split}
\right.
\end{equation}

In view of \eqref{solgeneral},
we need to modify the energies in \eqref{defEpr} and \eqref{defEpth},
which is designed for the solution of \eqref{eqtu} with data \eqref{initialNodd},
as follows:
\begin{align*}
\cE_{p,\al_p,\beta_p}^{r,z}(t)\eqdef&
N^{-2\al_p}\Bigl(\bigl\|r^{\f{p-3}2}|(\wt u_1,\wt v_1)|^{\f p2}\bigr\|_{L^\infty_t(L^2)}^2
+\bigl\|r^{\f{p-3}2}|{\wt\nabla}\wt u_1|
|\wt u_1|^{\f p2-1}\bigr\|_{L^2_t(L^2)}^2\\
&+\bigl\|r^{\f{p-3}2}|{\wt\nabla}\wt v_1|
|\wt v_1|^{\f p2-1}\bigr\|_{L^2_t(L^2)}^2
+N^{2}\bigl\|r^{\f{p-5}2}|(\wt u_1,\wt v_1)|^{\f p2}\bigr\|_{L^2_t(L^2)}^2
\Bigr)\\
&+\sum_{k=2}^\infty(kN)^{2\beta_p}\Bigl(
\bigl\|r^{\f{p-3}2}|(\wt u_{k},\wt v_{k})|^{\f p2}\bigr\|_{L^\infty_t(L^2)}^2
+\bigl\|r^{\f{p-3}2}|{\wt\nabla}\wt u_k|
|\wt u_k|^{\f p2-1}\bigr\|_{L^2_t(L^2)}^2\\
&+\bigl\|r^{\f{p-3}2}|{\wt\nabla}\wt v_k|
|\wt v_k|^{\f p2-1}\bigr\|_{L^2_t(L^2)}^2
+(kN)^{2}\bigl\|r^{\f{p-5}2}|(\wt u_{k},\wt v_{k})|^{\f p2}\bigr\|_{L^2_t(L^2)}^2
\Bigr),
\end{align*}
and
\begin{align*}
\cE_{p,\al_p,\beta_p}^\th(t)\eqdef & N^{\f p2-2\al_p}
\Bigl(\bigl\|r^{\f{p-3}2}|(\ut_1,\vt_1)|^{\f p2}\bigr\|
_{L^\infty_t(L^2)}^2
+\bigl\|r^{\f{p-3}2}|{\wt\nabla}\ut_1|
|\ut_1|^{\f p2-1}\bigr\|_{L^2_t(L^2)}^2\\
&+\bigl\|r^{\f{p-3}2}|{\wt\nabla}\vt_1|
|\vt_1|^{\f p2-1}\bigr\|_{L^2_t(L^2)}^2
+N^2\bigl\|r^{\f{p-5}2}|(\ut_1,\vt_1)|^{\f p2}\bigr\|_{L^2_t(L^2)}^2
\Bigr)\\
&+\sum_{k=2}^\infty(kN)^{\f p2+2\beta_p}
\Bigl(\bigl\|r^{\f{p-3}2}|(\ut_{k},\vt_{k})|^{\f p2}\bigr\|_{L^\infty_t(L^2)}^2
+\bigl\|r^{\f{p-3}2}|{\wt\nabla}\ut_k|
|\ut_k|^{\f p2-1}\bigr\|_{L^2_t(L^2)}^2\\
&+\bigl\|r^{\f{p-3}2}|{\wt\nabla}\vt_k|
|\vt_k|^{\f p2-1}\bigr\|_{L^2_t(L^2)}^2
+(kN)^2\bigl\|r^{\f{p-5}2}|(\ut_{k},\vt_{k})|^{\f p2}\bigr\|_{L^2_t(L^2)}^2
\Bigr),
\end{align*} where $\wt v_k\eqdef\vr_k\vv e_r+\vz_k\vv e_z.$
Then along the same line to the proof of Propositions \ref{propEp} and \ref{propH1},
we can prove that the same types of estimates as \eqref{aprioriEp} and \eqref{estipropH1}
still hold for the solution \eqref{solgeneral}, provided that $N$ is sufficiently large.
This indicates that the system
\eqref{NS} with initial data \eqref{initialN} still
has a unique global strong solution. 
We omit the details here.
This completes the proof of Theorem \ref{thm1}.
\end{proof}

Theorem \ref{thm2} follows  from Theorem \ref{thm1}
together with the global well-posedness of Navier-Stokes equations with
axisymmetric data without swirl through a perturbation argument.
The detailed proof of Theorem \ref{thm2} will be presented
at Section \ref{secthm2}.

\section{The estimates of the operator $\bigl(-(\pa_r^2+\f{\pa_r}r
+\pa_z^2)+\f{m^2}{r^2}\bigr)^{-1}$}\label{secpre}

This section is devoted to the study of the following elliptic equation
\begin{equation}\label{a1}
\cL_m P(r,z)\eqdef\Bigl(-\bigl(\pa_r^2+\f{\pa_r}r
+\pa_z^2\bigr)+\f{m^2}{r^2}\Bigr)P(r,z)=a(r,z),\quad m=0,1,2,\cdots,
\end{equation}
 in the half-plane $\{(r,z):r > 0, z\in\R\},$ which will be used in the estimate of
 the pressure functions in the system \eqref{eqtuk}.

  Recalling that in the cylindrical coordinates, $ \D=\p_r^2+\f{\p_r}{r}+\f{\p_\theta^2}{r^2}+\p_z^2,$
 $P(r,z)\cos m\th$ verifies
$$-\D\bigl(P(r,z)\cos m\th\bigr)=a(r,z)\cos m\th.$$
Noticing that for $x=(r\cos\th, r\sin\th, z)$ and $\bar{x}=(\bar{r}\cos\bar{\th}, \bar{r}\sin\bar{\th}, \bar{z}),$ one has
$$|x-\bar{x}|^2=\left(r-\bar r\right)^2+\left(z-\bar z\right)^2+2r\bar r\bigl(1-\cos(\th-\bar\th)\bigr),$$
then we  get, by using the fundamental solution of $3$-D
Laplace's equation, that
$$P(r,z)\cos m\th=\f1{4\pi}\int_{\R^+\times\R}\int_0^{2\pi}
\f{a(\bar r,\bar z)\cos m\bar\th}
{\sqrt{\left(r-\bar r\right)^2+\left(z-\bar z\right)^2+2r\bar r\bigl(1-\cos(\th-\bar\th)\bigr)}}
\bar r\,d\bar\th\, d\bar r\,d\bar z.$$
Taking $\th=0$ in the above equality yields
\begin{equation}\label{a2}
\begin{split}
&P(r,z)=\f{1}{2\pi}\int_{\R^+\times\R}F_m(\xi^2)
 a(\bar r,\bar z)\sqrt{\f{\bar r}{r}}\,d\bar rd\bar z \with\\
F_m(s)\eqdef&\int_0^\pi\f{\cos m\th}{\sqrt{s+2(1-\cos\th)}}\,d\th
\andf \xi^2\eqdefa\f{\left(r-\bar r\right)^2+\left(z-\bar z\right)^2}{r\bar r}.
\end{split}
\end{equation}

It is very important to investigate the asymptotic
 behaviors of $F_m(s)$ as $s\to 0$ and $s\to\infty.$

\begin{lem}\label{propa1}
{\sl For any $m\geq0$, there exists a uniform constant $C$ so that
for any $0<\al\leq m+\f12,~1\leq\beta\leq m+\f32,
~2\leq\gamma\leq m+\f52$, and any $s\in\R^+$,
there hold
\begin{equation}\label{a3}
s^\al |F_m(s)|\leq C \al^{-1}\ln(2+m)+4^m\pi \andf
s^\beta |F_m'(s)|+s^\gamma |F_m''(s)|\leq C 4^m.
\end{equation}
}\end{lem}

It is worth mentioning that the estimates \eqref{a3} are accurate in
describing the asymptotic behaviors of $F_m(s),$  $F_m'(s)$ and  $F_m''(s)$ as $s$ tends to $0$ or $+\infty$.
Yet the disadvantage of the estimates \eqref{a3} lies in the fact that the right-hand sides of the estimates
grow too rapidly as $m\rightarrow\infty$, which will be a disaster
in our proof of Theorem \ref{thm1}. To overcome this difficulty,
we also need to investigate the asymptotic behaviors of $F_m(s)$
as $m\rightarrow\infty$.

\begin{lem}\label{propa2}
{\sl Let $m\geq3$. There exists a uniform constant $C$ so that there hold

\noindent{\rm (i)} for any $1\leq\al\leq \f72,~2\leq\beta\leq \f92,~3\leq\gamma\leq \f{11}2$,
and any $s\in\R^+$,
\begin{equation}\label{a8}
s^\al |F_m(s)|\leq \f{C}m,\andf
s^\beta |F_m'(s)|+s^\gamma |F_m''(s)|\leq \f{C}m,
\end{equation}
{\rm (ii)} for any $0<\bar\al\leq \f52,~1\leq\bar\beta\leq \f72,
~2\leq\bar\gamma\leq \f92$ and any $s\in\R^+,$
\begin{equation}\label{a9}
s^{\bar\al} |F_m(s)|\leq C\bar\al^{-1}\ln m,\andf
s^{\bar\beta} |F_m'(s)|+s^{\bar\gamma} |F_m''(s)|\leq C\ln m.
\end{equation}
}\end{lem}

We shall  postpone
the proof of Lemmas \ref{propa1} and \ref{propa2} in the Appendix \ref{sectA}.

\begin{cor}\label{corFm}
{\sl For any $m\geq3$ and any $0<\delta'<\delta<1$,
there exists some positive constant $C$ depending only on
the choices of $\delta$ and $\delta'$ such that
\begin{equation}\label{a14}
s^{\delta} |F_m(s)|+
s^{1+\delta} |F_m'(s)|+s^{2+\delta}|F_m''(s)|\leq Cm^{-\delta'}.
\end{equation}
}\end{cor}

\begin{proof} Indeed it follows from \eqref{a8} and \eqref{a9} that
\beno
|F_m(s)|\leq C(ms)^{-1} \andf |F_m(s)|\leq C\bar{\al}^{-1}s^{-\bar{\al}}\ln m \quad \forall \ \bar{\al}\in (0,1),
\eeno
so that for any $\tau\in (0,1),$ one has
\begin{align*}
|F_m(s)|\leq C\bar{\al}^{-1+\tau} m^{-\tau}\left(\ln m\right)^{1-\tau} s^{-\left(\tau+\bar{\al}(1-\tau)\right)}.
\end{align*}
Taking $\bar{\al}\in (0,\delta)$ and $\tau=\f{\delta-\bar{\al}}{1-\bar{\al}}$ in the above inequality leads to the estimate of $F_m(s)$ in \eqref{a14}. The estimates of $F_m'(s)$ and $F_m''(s)$ can be derived along the same line.
\end{proof}

Let us now turn to the
estimates of the inverse operator of   $\cL_m$ given by \eqref{a1}.

\begin{prop}\label{lemLm}
{\sl Let $m\geq3$, $1<q<p<\infty$, and  $\al,~\beta$ satisfy
\begin{equation}\label{a15}
\al+\beta>0,\quad
\al+\f1p\geq-3,\quad\beta-\f1q\geq-2,\andf\f1q=\f1p+\f{\al+\beta}3.
\end{equation}
Then for any axisymmetric function $f(r,z)$ and any sufficiently small $\e>0,$ one has
\begin{equation}\label{a16}
\bigl\|r^{\al}\cL_m^{-1}\bigl(r^{\beta-2} f\bigr)\bigr\|_{L^p}
\lesssim_\e m^{-\left(1+\f1p-\f1q-\e\right)}\|f\|_{L^q}.
\end{equation}
Furthermore, if $p$ and $q$ satisfy in addition that
$\f12+\f1p-\f1q>0,$
 there holds
\begin{equation}\label{a17}
\bigl\|r^{\al}\cL_m^{-1}\bigl(\f1r\wt\nabla(r^\beta f)\bigr)\bigr\|_{L^p}
+\bigl\|r^{\al}\cL_m^{-1}\wt\nabla\bigl(r^{\beta-1} f\bigr)\bigr\|_{L^p}
\lesssim_\e m^{-\left(\f12+\f1p-\f1q-\e\right)}\|f\|_{L^q}.
\end{equation}
}\end{prop}

\begin{proof}
Since
$$\pa_z\bigl(r^{\beta-1} f\bigr)=\f1r\pa_z(r^\beta f),\andf
\pa_r\bigl(r^{\beta-1} f\bigr)=\f1r\pa_r(r^\beta f)
-r^{\beta-2} f,$$
it suffices to prove \eqref{a16} and
\begin{equation}\label{a18}
\bigl\|r^{\al}\cL_m^{-1}\bigl(\f1r\wt\nabla(r^\beta f)\bigr)\bigr\|_{L^p}
\lesssim_\e m^{-\left(\f12+\f1p-\f1q-\e\right)}\|f\|_{L^q}\quad\mbox{if}\ \ \f12+\f1p-\f1q>0.
\end{equation}

In fact, it follows from  \eqref{a2} that
\begin{align*}
r^{\al+\f1p}\cL_m^{-1}\bigl(r^{\beta-2} f\bigr)
=\f{r^{\al+\f1p-\f12}}{2\pi}\int_{\R^+\times\R} F_m(\xi^2)
 {\bar r}^{\beta-\f32} f(\bar r,\bar z)\,d\bar rd\bar z,
\end{align*}
and we get, by using integration by parts, that
\footnote{For any fixed $r>0,~z\in\R$, there holds $\xi^2\rightarrow+\infty$
whenever $\bar r\rightarrow0^+$ or $\bar r\rightarrow+\infty$ or $\bar z\rightarrow\infty$.
On the other hand, it follows from Lemma \ref{propa1}  that $F(\xi^2)$
decays at least as $\xi^{-(2m+1)}$ at infinity, so that
 no boundary term appears in the process of  integration by parts.}
\begin{align*}
r^{\al+\f1p}\cL_m^{-1}\Bigl(\f1r(\pa_r,\pa_z)(r^\beta f)\Bigr)
=-\f{r^{\al+\f1p-\f12}}{2\pi}\int_{\R^+\times\R}(\pa_{\bar r},\pa_{\bar z})
\left( F_m(\xi^2)\bar{r}^{-\f12}\right)
\cdot {\bar r}^{\beta} f(\bar r,\bar z)\,d\bar rd\bar z,
\end{align*}
where $\xi^2=\f{(r-\bar r)^2+(z-\bar z)^2}{r\bar r}$.
Accordingly, we write
\begin{equation}\begin{split}\label{a19}
&r^{\al+\f1p}\cL_m^{-1}\bigl(r^{\beta-2} f\bigr)(r,z)
=\int_{\R^+\times\R}\wt H_m(r,\bar r,z,\bar z)
\cdot {\bar r}^{\f1q}f(\bar r,\bar z)\,d\bar rd\bar z,\\
&r^{\al+\f1p}\cL_m^{-1}\bigl(\f1r\wt\nabla(r^\beta f)\bigr)(r,z)
=\int_{\R^+\times\R}H_m(r,\bar r,z,\bar z)
\cdot {\bar r}^{\f1q}f(\bar r,\bar z)\,d\bar rd\bar z,
\end{split}\end{equation}
where the integral kernels $H_m(r,\bar r,z,\bar z)$ and $\wt H_m(r,\bar r,z,\bar z)$ satisfy
\begin{equation}\begin{split}\label{a20}
&\quad\qquad\qquad\qquad|\wt H_m(r,\bar r,z,\bar z)|
\lesssim {r^{\al+\f1p-\f12}\bar r^{\beta-\f1q-\f32}}
|F_m(\xi^2)|,\\
&|H_m(r,\bar r,z,\bar z)|
\lesssim {r^{\al+\f1p-\f12}\bar r^{\beta-\f1q-\f32}}
\Bigl(|F_m(\xi^2)|+|F_m'(\xi^2)|\cdot
\bigl(\f{|r-\bar r|+|z-\bar z|}{r}+\xi^2\bigr)\Bigr).
\end{split}\end{equation}

Next we claim that, under the assumption \eqref{a15},
one has \begin{equation}\label{a21}
|\wt H_m(r,\bar r,z,\bar z)|\lesssim_\e
m^{-\left(1+\f1p-\f1q-\e\right)}
\left((r-\bar r)^{2}+(z-\bar z)^2\right)^{-\left(1+\f1p-\f1q\right)},
\end{equation}
and if in addition $\f12+\f1p-\f1q>0$, there holds
\begin{equation}\label{a22}
|H_m(r,\bar r,z,\bar z)|\lesssim_\e
m^{-\left(\f12+\f1p-\f1q-\e\right)}\left((r-\bar r)^{2}+(z-\bar z)^2\right)^{-\left(1+\f1p-\f1q\right)}.
\end{equation}

Let us first present the proof of \eqref{a22}.
The strategy is to use
\eqref{a20} and to decompose the integral domain in \eqref{a19} into the following three parts:
$$\Omega_1\eqdefa \bigl\{(\bar r,\bar z):\ \bar r<\f r2\bigr\},\quad
\Omega_2\eqdefa \bigl\{(\bar r,\bar z):\ \bar r>2r\bigr\}\andf
\Omega_3\eqdefa \bigl\{(\bar r,\bar z):\ \f r2\leq\bar r\leq2r\bigr\}.$$

\no {\bf Case 1:}
When $(\bar r,\bar z)\in\Omega_1$, there holds $r\thicksim|r-\bar r|$
and $\xi^2\geq\f12$.
If $-2\leq\beta-\f1q\leq\f12$, we deduce from \eqref{a8} and \eqref{a20}
that
\begin{align*}
|H_m(r,\bar r,z,\bar z)|
\lesssim & m^{-1}{r^{\al+\f1p-\f12}\bar r^{\beta-\f1q-\f32}}
\Bigl(\bigl(\f{r\bar r}{(r-\bar r)^2+(z-\bar z)^2}\bigr)^{\f32-\beta+\f1q}\\
&+\bigl(\f{r\bar r}{(r-\bar r)^2+(z-\bar z)^2}\bigr)^{\f52-\beta+\f1q}
\bigl(\f{(r-\bar r)^2+(z-\bar z)^2}{r\bar r}\bigr)\Bigr)\\
\lesssim &\f{r^{1+\al-\beta+\f1p+\f1q}}
{m\bigl((r-\bar r)^2+(z-\bar z)^2\bigr)^{\f32-\beta+\f1q}}
\lesssim \f{|r-\bar r|^{1+\al-\beta+\f1p+\f1q}}{m\bigl((r-\bar r)^2+(z-\bar z)^2\bigr)^{\f32-\beta+\f1q}}.
\end{align*}
Due to $\f1q=\f1p+\f{\al+\beta}3,$
 we have
\begin{align*}
\f32-\beta+\f1q-\f12\left(1+\al-\beta+\f1p+\f1q\right)=1-\left(\f1q-\f1p\right),
\end{align*}
which together with $\beta-\f1q\leq\f12$ and $q<p$ implies
that $1+\al-\beta+\f1p+\f1q>0,$ so that
\begin{align*}
|H_m(r,\bar r,z,\bar z)|
\lesssim &\f{\bigl((r-\bar r)^2+(z-\bar z)^2\bigr)^{\f12\left(1+\al-\beta+\f1p+\f1q\right)}}
{m\bigl((r-\bar r)^2+(z-\bar z)^2\bigr)^{\f32-\beta+\f1q}}\\
\lesssim & m^{-1}
\bigl((r-\bar r)^2+(z-\bar z)^2\bigr)^{-(1+\f1p-\f1q)}.
\end{align*}

While when $\beta-\f1q>\f12$, we deduce from \eqref{a8}
and the facts that $\bar r<\f r2$ and $\xi^2\geq\f12,$ that
\begin{align*}
|H_m(r,\bar r,z,\bar z)|
\lesssim & m^{-1} r^{-1+\al+\f1p+\beta-\f1q}\bar r^{-1}
\bigl(\xi^{-2}+\xi^{-4}\xi^2\bigr)\\
\lesssim & \f{|r-\bar r|^{\al+\beta+\f1p-\f1q}}{m\bigl((r-\bar r)^2+(z-\bar z)^2\bigr)}.
\end{align*}
Noticing that
\beq\label{a22a} \al+\beta+\f1p-\f1q=2\left(\f1q-\f1p\right)>0,\eeq
 we infer
\begin{align*}
|H_m(r,\bar r,z,\bar z)|
\lesssim&m^{-1}
\bigl((r-\bar r)^2+(z-\bar z)^2\bigr)^{-\left(1+\f1p-\f1q\right)}.
\end{align*}

By summarizing the above two cases, we achieve \eqref{a22} for  $(\bar r,\bar z)\in \Omega_1$.

\no {\bf Case 2:}
When  $(\bar r,\bar z)\in \Omega_2$, there holds $\bar r\thicksim|r-\bar r|$.
In case $-3\leq\al+\f1p\leq-\f12$, we deduce once again from  \eqref{a8} and \eqref{a20}
that
\begin{align*}
|H_m(r,\bar r,z,\bar z)|
\lesssim & m^{-1}{r^{\al+\f1p-\f12}\bar r^{\beta-\f1q-\f32}}
\Bigl(\bigl(\f{r\bar r}{(r-\bar r)^2+(z-\bar z)^2}\bigr)^{\f12-\al-\f1p}\\
&+\bigl(\f{r\bar r}{(r-\bar r)^2+(z-\bar z)^2}\bigr)^{\f32-\al-\f1p}
\bigl(\f{r\bar r}{(r-\bar r)^2+(z-\bar z)^2}\bigr)^2\Bigr)\\
\lesssim & \f{\bar r^{-1-\al+\beta-\f1p-\f1q}}
{m\bigl((r-\bar r)^2+(z-\bar z)^2\bigr)^{\f12-\al-\f1p}}
\lesssim  \f{|r-\bar r|^{-1-\al+\beta-\f1p-\f1q}}
{m\bigl((r-\bar r)^2+(z-\bar z)^2\bigr)^{\f12-\al-\f1p}}.
\end{align*}
It is easy to observe that
$$
\f12-\al-\f1p+\f12\left(1+\al-\beta+\f1p+\f1q\right)=1-\left(\f1q-\f1p\right).
$$
Due to $\al+\f1p\leq-\f12$ and $q<p,$ we find
$1+\al-\beta+\f1p+\f1q<0,$ so that we  obtain
\begin{align*}
|H_m(r,\bar r,z,\bar z)|
\lesssim & \f{\bigl((r-\bar r)^2+(z-\bar z)^2\bigr)^{-\f12\left(1+\al-\beta+\f1p+\f1q\right)}}
{m\bigl((r-\bar r)^2+(z-\bar z)^2\bigr)^{\f12-\al-\f1p}}\\
\lesssim & m^{-1}
\bigl((r-\bar r)^2+(z-\bar z)^2\bigr)^{-(1+\f1p-\f1q)}.
\end{align*}

While when $\al+\f1p>-\f12$, we get, by using \eqref{a8}
and the facts: $r<\f{\bar r}2$ and $\xi^2\geq\f12,$ that
\begin{align*}
|H_m(r,\bar r,z,\bar z)|
\lesssim & m^{-1} r^{-1}{\bar r}^{-1+\al+\beta+\f1p-\f1q}
\bigl(\xi^{-2}+\xi^{-4}\xi^2\bigr)\\
\lesssim & \f{\bar r^{\al+\beta+\f1p-\f1q}}{m\bigl((r-\bar r)^2+(z-\bar z)^2\bigr)}
\lesssim  \f{|r-\bar r|^{\al+\beta+\f1p-\f1q}}{m\bigl((r-\bar r)^2+(z-\bar z)^2\bigr)},
\end{align*}
which together with \eqref{a22a} ensures that
\begin{align*}
|H_m(r,\bar r,z,\bar z)|
\lesssim& m^{-1}
\bigl(|r-\bar r|^2+|z-\bar z|^2\bigr)^{-(1+\f1p-\f1q)}.
\end{align*}

By combining the above two cases, we achieve \eqref{a22} for $(\bar r,\bar z)\in\Omega_2$.

\no {\bf Case 3:}
When $(\bar r,\bar z)\in\Omega_3$, one has $r\thicksim\bar r$, so that there holds
$${r^{\al+\f1p-\f12}\bar r^{\beta-\f1q-\f32}}
\thicksim r^{-2+\al+\beta+\f1p-\f1q}=r^{-2\left(1+\f1p-\f1q\right)},\andf
\xi\thicksim\f{|r-\bar r|+|z-\bar z|}r.$$
When $\xi\geq1$, we get, by using \eqref{a20}, that
$$|H_m(r,\bar r,z,\bar z)|
\lesssim r^{-2\left(1+\f1p-\f1q\right)}\bigl(|F_m(\xi^2)|
+|F_m'(\xi^2)|\xi^2\bigr),$$
from which and \eqref{a14}, we deduce that for any
$\e\in\left]0,1+\f1p-\f1q\right[$,
\begin{align*}
|H_m(r,\bar r,z,\bar z)|
\lesssim_\e & m^{-\left(1+\f1p-\f1q-\e\right)}\cdot r^{-2\left(1+\f1p-\f1q\right)}
\Bigl(\xi^{-2\left(1+\f1p-\f1q\right)}+\xi^{-2-2\left(1+\f1p-\f1q\right)}\xi^2\Bigr)\\
\lesssim_\e & m^{-\left(1+\f1p-\f1q-\e\right)}
\bigl((r-\bar r)^2+(z-\bar z)^2\bigr)^{-\left(1+\f1p-\f1q\right)}.
\end{align*}
While when $\xi\in [0,1[$, it follows from \eqref{a20} that
$$|H_m(r,\bar r,z,\bar z)|
\lesssim r^{-2\left(1+\f1p-\f1q\right)}\bigl(|F_m(\xi^2)|
+|F_m'(\xi^2)|\xi\bigr),$$
from which and
 \eqref{a14}, we deduce that for any small
$\e\in\left]0,\f12+\f1p-\f1q\right[$,
\begin{align*}
|H_m(r,\bar r,z,\bar z)|
\lesssim_\e& r^{-2\left(1+\f1p-\f1q\right)}
\Bigl( m^{-\left(1+\f1p-\f1q-\e\right)}\xi^{-2\left(1+\f1p-\f1q\right)}+
 m^{-\left(\f12+\f1p-\f1q-\e\right)}\xi^{-2-2\left(\f12+\f1p-\f1q\right)}\xi\Bigr)\\
\lesssim_\e&m^{-\left(\f12+\f1p-\f1q-\e\right)}
\bigl((r-\bar r)^2+(z-\bar z)^2\bigr)^{-\left(1+\f1p-\f1q\right)}.
\end{align*}

By combining the above two cases, we achieve \eqref{a22} for $(\bar r,\bar z)\in\Omega_3$.

So far we have completed the proof of \eqref{a22}. Moreover, thanks to \eqref{a20}, the proof of \eqref{a22} also ensures
\eqref{a21}. Thanks to \eqref{a19} and \eqref{a22}, we find
\begin{align*}
r^{\al+\f1p}\bigl|\cL_m^{-1}\bigl(\f1r\wt\nabla(r^\beta f)\bigr)(r,z)\bigr|
\lesssim_\e &
m^{-\left(\f12+\f1p-\f1q-\e\right)}\\
&\times \iint_{\R^+\times\R}\left((r-\bar r)^{2}+(z-\bar z)^2\right)^{-\left(1+\f1p-\f1q\right)}
{\bar r}^{\f1q}f(\bar r,\bar z)\,d\bar rd\bar z.
\end{align*}
Then we get, by using Hardy-Littlewood-Sobolev inequality, that
\begin{align*}
\bigl\|r^{\al}\cL_m^{-1}\bigl(\f1r\wt\nabla(r^\beta f)\bigr)\bigr\|_{L^p}
=& \bigl\|r^{\al+\f1p}\cL_m^{-1}\bigl(\f1r\wt\nabla(r^\beta f)\bigr)\bigr\|_{L^p(\,dr\,dz)}\\
\lesssim_\e &
m^{-\left(\f12+\f1p-\f1q-\e\right)}\bigl\|{\bar r}^{\f1q}f(\bar r,\bar z)\bigr\|_{L^q(\,d\bar r\,d\bar z)}
\lesssim_\e
m^{-\left(\f12+\f1p-\f1q-\e\right)}\|f\bigr\|_{L^q},
\end{align*}
which is exactly the desired estimate \eqref{a18}. And \eqref{a16} follows along the same line
by using \eqref{a21} instead of \eqref{a22}.
This completes the proof of Proposition \ref{lemLm}.
\end{proof}

\begin{rmk}\label{rmkLm}
\begin{itemize}
\item[(1)] The assumption $\al+\beta>0$ in \eqref{a15} is essential, as our proof is based on
Hardy-Littlewood-Sobolev inequality.

\item[(2)] The constants $-3$ and $-2$
in $\al+\f1p\geq-3$ and $\beta-\f1q\geq-2,$ which appears in \eqref{a15}, actually depend on how large
$m$ is. For instance, if $\al$ is so small  that $\al+\f1p\geq-3-n$, then similar estimates as \eqref{a16} and \eqref{a17} still hold for $m\geq3+n.$ All what we need to modify the proof is to replace the upper bounds for the corresponding $|F_m(s)|$.
This  will not cause any difficulty when we  deal with
the pressure term in \eqref{eqtuk}, as the operator we
shall encounter there is $\cL_{kN}^{-1}$ with $kN$ being sufficiently large.

\item[(3)] The difficulty arises in the estimate of the inverse operator of $\cL_0$.
Indeed, it follows from \eqref{a3}  that
for any $0<\al\leq \f12,~1\leq\beta\leq \f32,~2\leq\gamma\leq \f52$,
there holds
$$s^{\al} |F_0(s)|<C\al^{-1},
\andf
s^{\beta} |F_0'(s)|+s^{\gamma} |F_0''(s)|<C.$$
Then we can verify that
similar estimates as \eqref{a16} and \eqref{a17} hold for $\cL_0^{-1}$ under the assumption that
$$\al+\f1p\geq-1\andf\beta-\f1q\geq 0,$$
which will not allow us to handle the pressure term $P_0$
in \eqref{eqtu0} in the same way as the estimate of $P_k$ in Section \ref{secEp}.
This motivates us to estimate $u_0$ alternatively.
\end{itemize}
\end{rmk}

Observing that in view of \eqref{a2}, we have
\begin{align*}
&r^{\al+\f1p}(\pa_r,\pa_z)\cL_m^{-1}\bigl(r^{\beta-1} f\bigr)=\f{r^{\al+\f1p-\f32}}{2\pi}\\&\times \int_{\R^+\times\R}
\Bigl(-\f12 F_m(\xi^2)+F_m'(\xi^2)
\bigl(\f{2(r-\bar r)}{\bar r}-\xi^2\bigr), F_m'(\xi^2)\f{2(z-\bar z)}{\bar r}\Bigr)
\cdot {\bar r}^{\beta-\f12} f(\bar r,\bar z)\,d\bar rd\bar z.
\end{align*}
Then along the same line to the proof of Proposition \ref{lemLm},
we can prove that

\begin{prop}\label{lemLmp}
{\sl Let $m\geq3$, $1<q<p<\infty$, and  $\al,~\beta$ satisfy
$$\al+\beta>0,\quad
\al+\f1p\geq-2,\quad\beta-\f1q\geq-3,\quad\f1q=\f1p+\f{\al+\beta}3
\andf\f12+\f1p-\f1q>0.$$
Then for any axisymmetric function $f(r,z)$ and any sufficiently small $\e>0,$
one has
$$\bigl\|r^{\al}\wt\nabla\cL_m^{-1}\bigl(r^{\beta-1} f\bigr)\bigr\|_{L^p}
\lesssim_\e m^{-\left(\f12+\f1p-\f1q-\e\right)}\|f\|_{L^q}.$$
}\end{prop}
\smallskip

To deal with the pressure term in \eqref{eqtuk}, we shall encounter
the operators of the form
\begin{equation}\label{a23}
\wt\nabla\cL_m^{-1}\circ\f1r\quad\text{or}\quad\cL_m^{-1}\lozenge\lozenge,
\quad\text{where}\quad\lozenge\eqdefa\bigl(\pa_r,\pa_z,\f1r\bigr).
\end{equation}
With Propositions \ref{lemLm} and \ref{lemLmp}, it remains to deal with the operator
 $\cL_m^{-1}\wt\nabla\wt\nabla$,
 which will be very different from the previous cases.
Precisely,  we have

\begin{prop}\label{lemLmRiesz}
{\sl Let $m\geq3$ and $1<p<\infty.$ Then
 for any axisymmetric  function $f(r,z),$  one has
$$\bigl\|\cL_m^{-1}\wt\nabla\wt\nabla f\bigr\|_{L^p}
\lesssim\|f\|_{L^p}+\bigl\|\cL_m^{-1}
\bigl(\f{1}r\wt\nabla f\bigr)\bigr\|_{L^p}.$$
}\end{prop}

\begin{proof}
In what follows, we shall frequently use the relations:
$$\pa_r=\cos\th\pa_1+\sin\th \pa_2,\quad
\pa_\th=-r\sin\th\pa_1+r\cos\th\pa_2,$$
and
$$\pa_1=\cos\th\pa_r-\f{\sin\th}r\pa_\th,\quad
\pa_2=\sin\th\pa_r+\f{\cos\th}r\pa_\th.$$

Let us first handle the estimate of $\cL_m^{-1}\pa_r\pa_z f$.
We denote $Q_1\eqdefa \cL_m^{-1}\pa_r\pa_z f$, i.e. $Q_1$ solves
$$\Bigl(-\bigl(\pa_r^2+\f{\pa_r}r
+\pa_z^2\bigr)+\f{m^2}{r^2}\Bigr)Q_1=\pa_r\pa_z f,$$
so that there holds
\begin{align*}
\cL_m Q_1=\pa_r\pa_z f
=\pa_3\bigl(\cos\th\pa_1+\sin\th\pa_2\bigr)f
=\pa_3\Bigl(\pa_1(f\cos\th)+\pa_2(f\sin\th)
-\f fr\Bigr).
\end{align*}
Accordingly, we  decompose $Q_1$ into $Q_{11}+Q_{12}$ with
$$Q_{11}\eqdefa -\cL_m^{-1}\pa_z\bigl(\f{f}r\bigr),\andf
Q_{12}\eqdefa\cL_m^{-1}\pa_3\left(\pa_1(f\cos\th)+\pa_2(f\sin\th)\right).$$
Since $f(r,z)$ is an axisymmetric function,  so are $Q_1$ and
$Q_{11}.$
Then $Q_{12}=Q_1-Q_{11}$ is also an axisymmetric function, i.e. $\pa_\th Q_{12}=0$.
As a result, it comes out
$$-\D\bigl(Q_{12}(r,z)\cos m\th\bigr)=\bigl(\cL_m Q_{12}(r,z)\bigr)\cos m\th
=\pa_3\left(\pa_1(f\cos\th)+\pa_2(f\sin\th)\right)\cos m\th.$$
It is interesting to observe that
\begin{align*}
\left(\pa_1(f\cos\th)+\pa_2(f\sin\th)\right)\cos m\th
=\pa_1(f\cos\th\cos m\th)+\pa_2(f\sin\th\cos m\th),
\end{align*}
which implies
$$Q_{12}(r,z)\cos m\th
=(-\D)^{-1}\pa_3\bigl(\pa_1(f\cos\th\cos m\th)
+\pa_2(f\sin\th\cos m\th)\bigr).$$
Then we get, by applying the $L^p~(1<p<\infty)$ boundedness
of the double Riesz transform $(-\D)^{-1}\pa_i\pa_j$ (see for instance \cite{stein}), that
\begin{equation}\label{a24}
\|Q_{12}(r,z)\cos m\th\|_{L^p}
\lesssim\|f\cos\th\cos m\th\|_{L^p}
+\|f\sin\th\cos m\th\|_{L^p}\lesssim\|f\|_{L^p}.
\end{equation}

On the other hand, observing that  $Q_{12}(r,z)$ is an axisymmetric function, one has
\begin{align*}
\|Q_{12}\cos m\th\|_{L^p}^p=\int_0^{2\pi}|\cos m\th|^p\,d\th\|Q_{12}\|_{L^p}^p, \andf\\
\int_0^{2\pi}|\cos m\th|^p\,d\th=2m\int_{-\f\pi{2m}}^{\f\pi{2m}}\cos^p m\th\,d\th=2\int_{-\f\pi2}^{\f\pi{2}}\cos^p \th\,d\th,
\end{align*}
which together with \eqref{a24} ensures that
$$\|Q_{12}\|_{L^p}\lesssim\|f\|_{L^p}.$$
Consequently, we achieve
\begin{equation}\label{a25}
\|\cL_m^{-1}\pa_r\pa_z f\|_{L^p}
\leq\|Q_{11}\|_{L^p}+\|Q_{12}\|_{L^p}
\lesssim\|f\|_{L^p}+\bigl\|\cL_m^{-1}\bigl(\f{1}r\pa_z f\bigr)\bigr\|_{L^p}.
\end{equation}

The estimate of $Q_2(r,z)\eqdefa\cL_m^{-1}\pa_z^2 f$ is much easier. Indeed
noticing that
$$-\D\bigl(Q_2(r,z)\cos m\th\bigr)=\bigl(\cL_m Q_2(r,z)\bigr)\cos m\th
=\pa_3^2(f\cos m\th),$$ from which,
 the fact that $Q_2$ is axisymmetric and the $L^p~(1<p<\infty)$ boundedness
of the double Riesz transform $(-\D)^{-1}\pa_3^2,$  we deduce that
\begin{equation}\label{a26}
\|Q_2\|_{L^p}\thicksim\|Q_2(r,z)\cos m\th\|_{L^p}
\lesssim\|f\cos m\th\|_{L^p}\lesssim\|f\|_{L^p}.
\end{equation}

Finally let us turn to the estimate of the term: $Q_3(r,z)\eqdefa\cL_m^{-1}\pa_r^2 f$.
For axisymmetric function $f,$
we first get, by a direct calculation, that
\begin{align*}
\pa_r^2 f=&\pa_r\Bigl(\pa_1(f\cos\th)+\pa_2(f\sin\th)
-\f fr\Bigr)\\
=&\pa_r\pa_1\left(f\cos\th\right)+\pa_r\pa_2\left(f\sin\th\right)
-\f{\pa_r f}r+\f{f}{r^2}\\
=&-\f{\pa_r f}r+\f{f}{r^2}+\pa_1\left[\cos\th\bigl(\pa_1(f\cos\th)
+\pa_2(f\sin\th)\bigr)\right]\\
&+\pa_2\left[\sin\th\bigl(\pa_1(f\cos\th)
+\pa_2(f\sin\th)\bigr)\right]
-\f1r\left(\pa_1(f\cos\th)
+\pa_2(f\sin\th)\right)\\
=&-2\f{\pa_r f}r+\pa_1^2\bigl(f\cos^2\th\bigr)
+\pa_2^2\bigl(f\sin^2\th\bigr)
+\pa_1\pa_2\bigl(f\sin2\th\bigr).
\end{align*}
So that we  decompose $Q_3(r,z)$ into $Q_{31}(r,z)+Q_{32}(r,z)$, where
$$Q_{31}\eqdefa-2\cL_m^{-1}\bigl(\f{\pa_r f}r \bigr),\quad
Q_{32}\eqdefa \cL_m^{-1}\left(\pa_1^2\bigl(f\cos^2\th\bigr)
+\pa_2^2\bigl(f\sin^2\th\bigr)
+\pa_1\pa_2\bigl(f\sin2\th\bigr)\right).$$
Since $Q_{32}$ does not depend on the $\th$ variable, one has
\begin{align*}
-\D\bigl(Q_{32}(r,z)\cos m\th\bigr)&=\bigl(\cL_m Q_{32}(r,z)\bigr)\cos m\th\\
&=\left(\pa_1^2\bigl(f\cos^2\th\bigr)
+\pa_2^2\bigl(f\sin^2\th\bigr)+\pa_1\pa_2\bigl(f\sin2\th\bigr)\right)\cos m\th.
\end{align*}
Yet observing that
\begin{align*}
&\bigl(\pa_1^2\bigl(f\cos^2\th\bigr)
+\pa_2^2\bigl(f\sin^2\th\bigr)\bigr)\cos m\th\\
&=\pa_1^2\bigl(f\cos^2\th\cos m\th\bigr)
+\pa_2^2\bigl(f\sin^2\th\cos m\th\bigr)-f\cos^2\th\pa_1^2\cos m\th\\
&\quad-f\sin^2\th\pa_2^2\cos m\th
-2\bigl(\pa_1\bigl(f\cos^2\th\bigr)\pa_1\cos m\th+\pa_2\bigl(f\sin^2\th\bigr)\pa_2\cos m\th\bigr)\\
&=\pa_1^2\bigl(f\cos^2\th\cos m\th\bigr)
+\pa_2^2\bigl(f\sin^2\th\cos m\th\bigr)\\
&\quad+\f{3m}{2r^2}f\sin 4\th\sin m\th+\f{m^2}{2r^2}f\sin^22\th\cos m\th -\f{m}{2r}\p_rf\sin 4\th\cos m\th,
\end{align*}
and
\begin{align*}
\pa_1\pa_2\bigl(f\sin2\th\bigr)\cos m\th
&=\pa_1\pa_2\bigl(f\sin2\th\cos m\th\bigr)-f\sin2\th\pa_1\pa_2\cos m\th\\
&\quad-\pa_1\bigl(f\sin2\th\bigr)\pa_2\cos m\th
-\pa_2\bigl(f\sin2\th\bigr)\pa_1\cos m\th\\
&=\pa_1\pa_2\bigl(f\sin2\th\cos m\th\bigr)-\f{3m}{2r^2}f\sin 4\th\sin m\th\\
&\quad-\f{m^2}{2r^2}f\sin^22\th\cos m\th +\f{m}{2r}\p_rf\sin 4\th\cos m\th.
\end{align*}
As a result, it comes out
\begin{align*}
-\D\bigl(Q_{32}(r,z)\cos m\th\bigr)
&=\pa_1^2\bigl(f\cos^2\th\cos m\th\bigr)
+\pa_2^2\bigl(f\sin^2\th\cos m\th\bigr)
+\pa_1\pa_2\bigl(f\sin2\th\cos m\th\bigr).
\end{align*}
Then we get, by a similar derivation of \eqref{a26}, that
$$\|Q_{32}\|_{L^p}\thicksim\|Q_{32}\cos m\th\|_{L^p}
\lesssim\|f\|_{L^p}.$$
Hence we obtain
\begin{equation}\label{a27}
\|\cL_m^{-1}\pa_r^2 f\|_{L^p}
\leq\|Q_{31}\|_{L^p}+\|Q_{32}\|_{L^p}
\lesssim\|f\|_{L^p}+\bigl\|\cL_m^{-1}\bigl(\f{\pa_r f}r \bigr)\bigr\|_{L^p}.
\end{equation}

By summarizing  the estimates \eqref{a25}-\eqref{a27},  we
complete the proof of  Proposition \ref{lemLmRiesz}.
\end{proof}

Proposition \ref{lemLmRiesz} handles the estimate of $\cL_m^{-1}\wt\nabla\wt\nabla f$ in the $L^p$ framework.
Next we are going to reduce its $r$-weighted estimate to Propositions \ref{lemLm} and \ref{lemLmRiesz} via the following proposition:

\begin{prop}\label{lemLmweightedRiesz}
{\sl Let $m\geq3,~1<q<p<\infty$ and $0<\gamma<\f32$ which
satisfy $\f1q=\f1p+\f\gamma3.$ Then
for any axisymmetric function$f(r,z)$ and any small enough  $\e>0,$ one has
\begin{align*} \bigl\|r^\gamma\cL_m^{-1}\wt\nabla\wt\nabla f\bigr\|_{L^p}
\lesssim_\e & m^{-\f12+\f\gamma3+\e}
\bigl\|\cL_m^{-1}\wt\nabla\wt\nabla f\bigr\|_{L^q}
+\bigl\|r^\gamma f\bigr\|_{L^p}\\
&+\bigl\|\cL_m^{-1}
\bigl(\f{1}r\wt\nabla(r^\gamma f)\bigr)\bigr\|_{L^p}
+\bigl\|\cL_m^{-1} \bigl(r^{\gamma-2} f\bigr)\bigr\|_{L^p}.\end{align*}
}\end{prop}

\begin{proof}
Here we only present the estimate of the term $r^\gamma\cL_m^{-1}\pa_r\pa_z f$.
The estimates of the remaining terms: $r^\gamma\cL_m^{-1}\pa_r^2 f$
and $r^\gamma\cL_m^{-1}\pa_z^2 f,$ can be handled along the same line.
As in the proof of Proposition \ref{lemLmRiesz}, we denote  $Q_1\eqdefa\cL_m^{-1}\pa_r\pa_z f$. Then one has
\begin{align*}
\cL_m\bigl(r^\gamma Q_1\bigr)&=r^\gamma\cL_m Q_1
+\gamma^2 r^{\gamma-2}Q_1-2\gamma\f{\pa_r}r\bigl(r^\gamma Q_1\bigr)\\
&=\pa_r\pa_z\bigl(r^\gamma f\bigr)-\gamma\pa_z\bigl(r^{\gamma-1} f\bigr)
+\gamma^2r^{\gamma-2}Q_1-2\f{\gamma}r\pa_r\bigl(r^\gamma Q_1\bigr).
\end{align*}
Then for sufficiently small $\e>0$, we deuce from Propositions \ref{lemLm} and  \ref{lemLmRiesz} that
\begin{align*}
\bigl\|r^\gamma Q_1\bigr\|_{L^p}
&\lesssim\bigl\|\cL_m^{-1}\pa_r\pa_z\bigl(r^\gamma f\bigr)\bigr\|_{L^p}
+\bigl\|\cL_m^{-1}{\pa_z}\bigl(r^{\gamma-1} f\bigr)\bigr\|_{L^p}\\
&\qquad+\bigl\|\cL_m^{-1}\bigl(r^{\gamma-2}Q_1\bigr)\bigr\|_{L^p}
+\bigl\|\cL_m^{-1}\bigl(\f{\pa_r}r(r^\gamma Q_1)\bigr)\bigr\|_{L^p}\\
&\lesssim_\e\bigl\|r^\gamma f\bigr\|_{L^p}+\bigl\|\cL_m^{-1}
\bigl(\f{1}r\wt\nabla(r^\gamma f)\bigr)\bigr\|_{L^p}
+m^{-\f12+\f\gamma3+\e}\|Q_1\|_{L^q},
\end{align*}
where $p,q$ satisfy $\f1q=\f1p+\f\gamma3$. This completes the proof of Proposition \ref{lemLmweightedRiesz}.
\end{proof}

\begin{rmk}
\begin{itemize}
\item[(1)]
We deduce  from the proof of Proposition \ref{lemLm} that,
if $\lozenge\lozenge$ in \eqref{a23} contains $j$ derivatives,
for $j\in\{0,1,2\}$, then to ensure the boundedness of the operator $\cL_m^{-1}\lozenge\lozenge,$ $p,q$ and $j$ have to satisfy
\begin{equation}\label{conditionj}
1-\f j2+\f1p-\f1q>0.
\end{equation}
In particular, when $j=0$,  \eqref{conditionj}  holds automatically,
that is reason why there is no additional assumption  for the estimate \eqref{a16}. And \eqref{a17}
corresponds to the case when $j=1$. However, when $j=2$,
\eqref{conditionj} becomes $\f1p-\f1q>0$, which contradicts with the assumption
$1<q<p<\infty$.

\item[(2)]
Concerning resolvent estimate of
 the
Schr\"{o}dinger operator: $-\D+V(x)$,  with potential $V(x)\geq0,$
the boundedness of $\nabla^2(-\D+V)^{-1},~\nabla(-\D+V)^{-1}\nabla$
etc. in $L^p$ or Hardy spaces usually requires $V$
to satisfy a sort of reverse H\"older's inequality.
There are numerous works concerning this topic,
here we only list the references  \cite{shen, zhong}.

Observing from \eqref{a1} that $\cL_m$ can be regarded as a Schr\"{o}dinger operator with potential $V(x)=\f{m^2}{r^2},$
which is even not locally integrable.
Yet it follows from  Propositions \ref{lemLm} and \ref{lemLmRiesz}
 that for any axisymmetric function $f(r,z)$ 
$$\bigl\|\cL_m^{-1}\wt\nabla\wt\nabla f\bigr\|_{L^p}
\lesssim\|f\|_{L^p}+\bigl\|\cL_m^{-1}
\bigl(\f{1}r\wt\nabla f\bigr)\bigr\|_{L^p}
\lesssim\|f\|_{L^p}+\|r^{-\gamma} f\|_{L^q}$$
for any $1<q<p<\infty,~0<\gamma<\f32$
satisfying $\f1q=\f1p+\f\gamma3$.
In this sense,  Proposition \ref{lemLmRiesz} is very interesting.
Of course, here the proof highly relies on the  axisymmetric structures.
\end{itemize}
\end{rmk}

\section{The proof of Proposition \ref{propEp}}\label{secEp}

This section is devoted to the proof of Proposition \ref{propEp}, which is the core to the proof of
Theorem \ref{thm1}.
Without ambiguity,
we shall write $E_{p,\al_p,\beta_p}(t)$
briefly as $E_p(t)$ in this section.

\begin{proof}[{\bf Proof of Proposition \ref{propEp}}] For $p\in ]5,6[,$
by taking $L^2$ inner product of the $\ur_k$ and $\uz_k$ equations   in
\eqref{eqtuk} with $r^{p-3}|\wt u_k|^{p-2}\ur_k$ and
$r^{p-3}|\wt u_k|^{p-2}\uz_k$ respectively,
we obtain
\beq\label{S2eq1}
\begin{split}
\int_{\R^3}\Bigl(\f12\pa_t(\ur_k)^2&-\D\ur_k\cdot\ur_k\Bigr)
r^{p-3}|\wt u_k|^{p-2}\,dx
+(1+k^2N^2)\int_{\R^3} r^{p-5}|\wt u_k|^{p-2}(\ur_k)^2\,dx\\
&=-\int_{\R^3}\Bigl(u_0\cdot\widetilde{\nabla}\ur_{k}
+\wt u_k\cdot\wt\nabla\ur_0+\f{2kN}{r^2}\vt_k
-F^r_k+\pa_r P_k\Bigr) r^{p-3}|\wt u_k|^{p-2}\ur_k\,dx,
\end{split} \eeq
and
\beq\label{S2eq2}
\begin{split}
\int_{\R^3}\Bigl(\f12\pa_t|\uz_k|^2&-\D\uz_k\cdot\uz_k\Bigr)
r^{p-3}|\wt u_k|^{p-2}\uz_k\,dx
+k^2N^2\int_{\R^3} r^{p-5}|\wt u_k|^{p-2}|\uz_k|^2\,dx\\
&=-\int_{\R^3}\left(u_0\cdot\widetilde{\nabla}\uz_{k}
+\wt u_k\cdot\wt\nabla\uz_0
-F^z_k+\pa_z P_k\right) r^{p-3}|\wt u_k|^{p-2}\uz_k\,dx.
\end{split} \eeq

In view of \eqref{S2notion}, one has $
|\wt u_k|^2=(\ur_k)^2+(\uz_k)^2,$
and thus
\begin{align*}
\f12\int_{\R^3} r^{p-3}|\wt u_k|^{p-2}\pa_t(\ur_k)^2\,dx
+\f12\int_{\R^3} r^{p-3}|\wt u_k|^{p-2}\pa_t|\uz_k|^2\,dx
=\f1p\f{d}{dt}\int_{\R^3} r^{p-3}|\wt u_k|^p\,dx.
\end{align*}
By using integration by parts, we have
\begin{align*}
-&\int_{\R^3}\D\ur_k\cdot\ur_k r^{p-3}|\wt u_k|^{p-2}\,dx
=-2\pi\int_0^\infty\int_{\R}\bigl(\pa_r^2+\f{\pa_r}r+\pa_z^2\bigr)\ur_k\cdot\ur_k
|\wt u_k|^{p-2} r^{p-2}\,drdz\\
&=2\pi\int_0^\infty\int_{\R}\Bigl(\bigl|\wt\nabla\ur_k\bigr|^2
+\f{p-2}4\bigl(\wt\nabla(\ur_k)^2\cdot|\wt u_k|^{-2}\wt\nabla|\wt u_k|^2\bigr)
+\f{p-3}{2r}\pa_r(\ur_k)^2\Bigr)|\wt u_k|^{p-2}r^{p-2}\,drdz,
\end{align*}
so that there holds
\begin{align*}
-&\int_{\R^3}\D\ur_k\cdot\ur_k r^{p-3}|\wt u_k|^{p-2}\,dx
-\int_{\R^3}\D\uz_k\cdot\uz_k r^{p-3}|\wt u_k|^{p-2}\,dx\\
=&\int_{\R^3} \bigl|\wt\nabla\wt u_k\bigr|^2
|\wt u_k|^{p-2}r^{p-3}\,dx
+(p-2)\int_{\R^3} \bigl|\wt\nabla|\wt u_k|\bigr|^2
|\wt u_k|^{p-2}r^{p-3}\,dx\\
&+\pi(p-3)\int_0^\infty\int_{\R}\pa_r|\wt u_k|^2\cdot|\wt u_k|^{p-2}r^{p-3}\,drdz\\
\geq&\int_{\R^3}|\nabla\wt u_k|^2|\wt u_k|^{p-2}r^{p-3}\,dx
-2\int_{\R^3}|\wt u_k|^p r^{p-5}\,dx,
\end{align*}
where in the last step, we  used integration by part
and  $p\in ]5,6[$ so that
\begin{align*}
\pi(p-3)&\int_0^\infty\int_{\R}\pa_r|\wt u_k|^2\cdot|\wt u_k|^{p-2}r^{p-3}\,drdz\\
&=
-\f{2\pi(p-3)^2}{p}\int_0^\infty\int_{\R} r^{p-4}|\wt u_k|^p\,drdz
+\f{2\pi(p-3)}{p}\int_{\R}\lim_{r\to\infty} \left(r^{p-3}|\wt u_k|^p\right)\,dz\\
&\geq -2\int_{\R^3}|\wt u_k|^p r^{p-5}\,dx.
\end{align*}
Then we get, by summing up \eqref{S2eq1} and \eqref{S2eq2} and then substituting the above estimates into the resulting inequality, that
\begin{equation}\label{3.1}\begin{split}
&\f1p\f{d}{dt}\bigl\|r^{\f{p-3}2}|\wt u_k|^{\f p2}\bigr\|_{L^2}^2
+\bigl\|r^{\f{p-3}2}|\wt\nabla\wt u_k||\wt u_k|^{\f p2-1}\bigr\|_{L^2}^2
+\bigl((kN)^{2}-2\bigr)\bigl\|r^{\f{p-5}2}|\wt u_k|^{\f p2}\bigr\|_{L^2}^2
\leq \sum_{j=1}^5 |I_k^{(j)}|,\\
&I_k^{(1)}\eqdefa
\int_{\R^3}(u_0\cdot\wt\nabla\wt u_k)\cdot |\wt u_k|^{p-2}\wt u_k  r^{p-3}\,dx,
\quad I_k^{(2)}\eqdefa \int_{\R^3}(\wt u_k\cdot\wt\nabla u_0)
\cdot|\wt u_k|^{p-2}\wt u_k r^{p-3}\,dx,\\
&I_k^{(3)}\eqdefa 2kN\int_{\R^3}\vt_k |\wt u_k|^{p-2}\ur_k r^{p-5}\,dx,
\qquad\ \ I_k^{(4)}
\eqdefa\int_{\R^3} \bigl(F^r_k\ur_k+F^z_k\uz_k\bigr) |\wt u_k|^{p-2}r^{p-3}\,dx,\\
&I_k^{(5)}
\eqdefa \int_{\R^3} \bigl(\wt u_k\cdot\wt\nabla P_k\bigr) |\wt u_k|^{p-2}r^{p-3}\,dx.
\end{split}\end{equation}

In what follows, we shall handle the estimate of $|I_k^{(j)}|$ term by term.
Firstly, since $r^{\f{p-3}2}|\wt u_k|^{\f p2}$
is axisymmetric, it follows from Sobolev embedding theorem that
\beq \label{S2eq3}\bigl\|r^{\f{p-3}2}|\wt u_k|^{\f p2}\bigr\|_{L^6}
\lesssim\bigl\|\nabla\bigl(r^{\f{p-3}2}|\wt u_k|^{\f p2}\bigr)\bigr\|_{L^2}
\lesssim\bigl\|r^{\f{p-3}2}|\wt\nabla\wt u_k||\wt u_k|^{\f p2-1}\bigr\|_{L^2}
+\bigl\|r^{\f{p-5}2}|\wt u_k|^{\f p2}\bigr\|_{L^2},\eeq
from which,
we infer
\begin{equation}\begin{split}\label{3.2}
|I_k^{(1)}|&\lesssim\|u_0\|_{L^3}
\bigl\|r^{\f{p-3}2}|\wt\nabla\wt u_k||\wt u_k|^{\f p2-1}\bigr\|_{L^2}
\bigl\|r^{\f{p-3}2}|\wt u_k|^{\f p2}\bigr\|_{L^6}\\
&\lesssim\|u_0\|_{L^3}\Bigl(
\bigl\|r^{\f{p-3}2}|\wt\nabla\wt u_k||\wt u_k|^{\f p2-1}\bigr\|_{L^2}^2
+\bigl\|r^{\f{p-5}2}|\wt u_k|^{\f p2}\bigr\|_{L^2}^2\Bigr).
\end{split}\end{equation}

For $I_k^{(2)}$, we get, by using integration by parts and the 4-th equation of \eqref{eqtuk}, that
\begin{align*}
I_k^{(2)}=&2\pi\int_0^\infty\int_{\R}(\wt{u}_k\cdot\wt\nabla u_0)
\cdot|\wt u_k|^{p-2}\wt u_k r^{p-2}\,dr\,dz\\
=&-\int_{\R^3}\Bigl( u_0\cdot\bigl(\wt u_k\cdot\wt\nabla\wt u_k\bigr)
+(u_0\cdot\wt u_k)\bigl((p-3)\f{\ur_k}r
-kN\f{\vt_k}r\bigr) \Bigr)|\wt u_k|^{p-2}r^{p-3}\,dx\\
&-(p-2)\int_{\R^3}u_0\cdot\wt u_k\bigl(\wt u_k\cdot\wt\nabla|\wt u_k|\bigr)
|\wt u_k|^{p-3}r^{p-3}\,dx.
\end{align*}
By applying H\"older inequality and Young inequality, we find
\begin{equation}\begin{split}\label{3.3}
|I_k^{(2)}|\lesssim &\|u_0\|_{L^3}
\bigl\|r^{\f{p-3}2}|\wt u_k|^{\f p2}\bigr\|_{L^6}\Bigl(
\bigl\|r^{\f{p-3}2}|\wt\nabla\wt u_k||\wt u_k|^{\f p2-1}\bigr\|_{L^2}
+\bigl\|r^{\f{p-5}2}|\wt u_k|^{\f p2}\bigr\|_{L^2}\\
&+kN
\bigl\|r^{\f{p-5}2}|\wt u_k|^{\f p2}\bigr\|_{L^2}^{1-\f2p}
\bigl\|r^{\f{p-5}2}|\vt_k|^{\f p2}\bigr\|_{L^2}^{\f2p}\Bigr)\\
\lesssim &\|u_0\|_{L^3}
\Bigl(\bigl\|r^{\f{p-3}2}|\wt\nabla\wt u_k||\wt u_k|^{\f p2-1}\bigr\|_{L^2}^2
+(kN)^{2}\bigl\|r^{\f{p-5}2}|\wt u_k|^{\f p2}\bigr\|_{L^2}^2\\
&+(kN)^{\f p2+2}\bigl\|r^{\f{p-5}2}|\vt_k|^{\f p2}\bigr\|_{L^2}^2\Bigr).
\end{split}\end{equation}

And $I_k^{(3)}$ can be handled as follows
\begin{equation}\begin{split}\label{3.4}
|I_k^{(3)}|
&\lesssim(kN)^{\f12}
\Bigl(\bigl\|r^{\f{p-5}2}|\wt u_k|^{\f p2}\bigr\|_{L^2}^2\Bigr)^{\f{p-1}p}
\Bigl((kN)^{\f p2}\bigl\|r^{\f{p-5}2}|\vt_k|^{\f p2}\bigr\|_{L^2}^2\Bigr)^{\f1p}\\
&\lesssim(kN)^{-\f32}
\Bigl((kN)^2\bigl\|r^{\f{p-5}2}|\wt u_k|^{\f p2}\bigr\|_{L^2}^2
+(kN)^{\f p2+2}\bigl\|r^{\f{p-5}2}|\vt_k|^{\f p2}\bigr\|_{L^2}^2\Bigr).
\end{split}\end{equation}

By substituting the estimates \eqref{3.2}-\eqref{3.4}
into  \eqref{3.1} and multiplying the inequalities
by $N^{-2\al_p}$ for $k=1$
and by $(kN)^{2\beta_p}$ for $k\geq2$,  and then integrating them with respect to time over $[0,t],$ and finally
 summing up the resulting inequalities for $k\in\N^+,$ we achieve
\begin{equation}\label{3.5}
\begin{split}
E_p^{r,z}(t)\leq E_{p}^{r,z}(0)
+ &C\bigl(\|u_0\|_{L^\infty_t(L^3)}+N^{-\f32}\bigr)E_p(t)\\
&+C\sum_{j=4}^5\int_0^t\Bigl(N^{-2\al_p}|I^{(j)}_1|
+\sum_{k\geq2}(kN)^{2\beta_p} |I_k^{(j)}|\Bigr)\,dt'.
\end{split}
\end{equation}

The estimate of the last term in \eqref{3.5} relies on the following lemmas:

\begin{lem}\label{S2lem1}
{\sl For $I_k^{(4)}$ in \eqref{3.1}, one has
\begin{equation}\label{3.19}
\int_0^t\Bigl(N^{-2\al_p}|I^{(4)}_1|
+\sum_{k\geq2}(kN)^{2\beta_p} |I_k^{(4)}|\Bigr)\,dt'
\leq C E_p^{1+\f1p}.
\end{equation}}
\end{lem}

\begin{lem}\label{S2lem2}
{\sl For $I_k^{(5)}$ in \eqref{3.1}, one has
\begin{equation}\begin{split}\label{3.46}
\int_0^t\Bigl(N^{-2\al_p}|I^{(5)}_1|
+\sum_{k\geq2}(kN)^{2\beta_p}& |I_k^{(5)}|\Bigr)\,dt'
\leq CE_p^{1+\f1p}+\bigl(\f14+C\|u_0\|_{L^\infty_t(L^3)}\bigr)E_p\\
+&CN^{-2}D^{\f{p+1}2}+CN^{-p}\bigl(\|u_0\|_{L^\infty_t(L^3)}
+N^{-\f{p}{2(p-5)}}\bigr)D^{\f{p}2},
\end{split}\end{equation}
for $D$ being given by \eqref{aprioriEp}.
}
\end{lem}

We admit the above lemmas for the time being and continue our proof of Proposition \ref{propEp}.
By substituting the estimates \eqref{3.19} and \eqref{3.46}
into \eqref{3.5}, we find
\begin{equation}\begin{split}\label{ineqEpr}
E_p^{r,z}\leq E_{p}^{r,z}(0)
+\bigl(\f14+&C\|u_0\|_{L^\infty_t(L^3)}+CN^{-\f32}\bigr)E_p
+CE_p^{1+\f1p}\\
&+CN^{-2}D^{\f{p+1}2}+CN^{-p}\bigl(\|u_0\|_{L^\infty_t(L^3)}
+N^{-\f{p}{2(p-5)}}\bigr)D^{\f{p}2}.
\end{split}\end{equation}
Exactly along the same line, we can get the estimate for $E_p^\th$:
\begin{equation}\begin{split}\label{ineqEpth}
E_p^\th\leq E_{p}^\th(0)
+\bigl(\f14+&C\|u_0\|_{L^\infty_t(L^3)}+CN^{-\f12}\bigr)E_p
+CE_p^{1+\f1p}\\
&+CN^{-2}D^{\f{p+1}2}+CN^{-p}\bigl(\|u_0\|_{L^\infty_t(L^3)}
+N^{-\f{p}{2(p-5)}}\bigr)D^{\f{p}2}.
\end{split}\end{equation}
By summing up the estimates \eqref{ineqEpr} and \eqref{ineqEpth}
and taking $N$ to be so large that $CN^{-\f12}\leq \f14$, we achieve \eqref{aprioriEp}.
This completes the proof of Proposition \ref{propEp}.
\end{proof}

It remains to prove Lemmas \ref{S2lem1} and \ref{S2lem2}, which we present below.

\begin{proof}[{\bf Proof of Lemma \ref{S2lem1}}]
Here we shall only  present the estimate of
$F^r_k\ur_k$ part in $I_k^{(4)}$. The
$F^z_k\uz_k$ part can be handled exactly along the same line.
We first get, by using integration by parts and the fourth equation of \eqref{eqtuk}, that
\begin{align*}
&\sum_{(k_1,k_2)\in\Omega_k}\Bigl|\int_{\R^3}\wt u_{k_1}\cdot\wt\nabla\ur_{k_2}
\cdot |\wt u_k|^{p-2}\ur_kr^{p-3}\,dx\Bigr|\\
&=\sum_{(k_1,k_2)\in\Omega_k}\Bigl|\int_{\R^3}
\ur_{k_2}\Bigl(\wt u_{k_1}\cdot\wt\nabla\ur_k
+(p-2)\ur_k|\wt{u}_k|^{-1}\wt u_{k_1}
\cdot\wt\nabla|\wt u_k|\\
&\qquad\qquad\qquad\qquad\qquad
+r^{-1}\bigl((p-3)\ur_{k_1}-k_1Nv^\th_{k_1}\bigr)
\ur_k\Bigr)|\wt{u}_k|^{p-2}r^{p-3}\,dx\Bigr|\\
&\lesssim\sum_{(k_1,k_2)\in\Omega_k}\int_{\R^3}
\Bigl(|\wt u_{k_1}||\wt u_{k_2}|\bigl(|\wt\nabla\wt u_k|
+|r^{-1}\wt u_k|\bigr)
+k_1N r^{-1}|v^\th_{k_1}||\wt u_{k_2}|
|\wt u_k|\Bigr)|\wt u_k|^{p-2}r^{p-3}\,dx.
\end{align*}
Here and in all that follows, we always denote
\begin{equation}\label{defOmk}
\Omega_k\eqdef\bigl\{(k_1,k_2)\in\N^+\times\N^+
:|k_1-k_2|=k\text{ or }k_1+k_2=k\bigr\}.
\end{equation}
So that in view of \eqref{eqtukq}, we  decompose the estimate of the $F^r_k\ur_k$ part in $I_k^{(4)}$ as follows
\begin{equation}\label{3.6}\begin{split}
\Bigl|\int_{\R^3}& F^r_k\ur_k
|\wt u_k|^{p-2}r^{p-3}\,dx\Bigr|
\lesssim I_k^{(4,1)}+I_k^{(4,2)}
+I_k^{(4,3)} \with\\
I_k^{(4,1)}
&=\sum_{(k_1,k_2)\in\Omega_k}\int_{\R^3}
|\wt u_{k_1}||\wt u_{k_2}|\bigl(|\wt\nabla\wt u_k|
+|r^{-1}\wt u_k|\bigr)|\wt u_k|^{p-2}r^{p-3}\,dx,\\
 I_k^{(4,2)}&\eqdef
\sum_{(k_1,k_2)\in\Omega_k}\int_{\R^3}
|\vt_{k_1}||\vt_{k_2}|
|\wt u_k|^{p-1} r^{p-4}\,dx,\\
 I_k^{(4,3)}
&\eqdef \sum_{(k_1,k_2)\in\Omega_k}(k_1+k_2)N\int_{\R^3}
{|\vt_{k_1}||\wt u_{k_2}}||\wt u_k|^{p-1}
 r^{p-4}\,dx.
\end{split}\end{equation}

Below we shall handle the estimate term by term above.

\no$\bullet$~{\bf The estimate of $I_k^{(4,1)}$ when $k=1$.}

There is no $k_1,k_2\in\N^+$ satisfying
$k_1+k_2=1$. So that in this case, we have
\begin{align*}
I_1^{(4,1)}=2\sum_{j=1}^\infty\int_{\R^3}
|\wt u_{j}||\wt u_{j+1}|\bigl(|\wt\nabla\wt u_j|
+|r^{-1}\wt u_j|\bigr)|\wt u_j|^{p-2}r^{p-3}\,dx.
\end{align*}
Then we get, by using H\"older's inequality, that
\beq \label{S2eq4}
\begin{split}
&N^{-2\al_p}I_1^{(4,1)}
\lesssim\Bigl(N^{1-\al_p}\bigl\|r^{\f{p-5}2}|\wt u_1|^{\f p2}\bigr\|_{L^2}\Bigr)^{\f12-\f3{2p}}
\Bigl(N^{-\al_p}\bigl\|r^{\f{p-3}2}|\wt u_1|^{\f p2}\bigr\|_{L^6}\Bigr)^{\f7{2p}-\f12}\\
&\quad\times\Bigl(N^{1+\beta_p}\bigl\|r^{\f{p-5}2}|\wt u_2|^{\f p2}\bigr\|_{L^2}\Bigr)^{\f12-\f3{2p}}
\Bigl(N^{\beta_p}\bigl\|r^{\f{p-3}2}|\wt u_2|^{\f p2}\bigr\|_{L^6}\Bigr)^{\f7{2p}-\f12}\\
&\quad\times\Bigl(N^{-\al_p}\bigl\|r^{\f{p-3}2}|\wt\nabla\wt u_1|
|\wt u_1|^{\f p2-1}\bigr\|_{L^2}
+N^{1-\al_p}\bigl\|r^{\f{p-5}2}|\wt u_1|^{\f p2}\bigr\|_{L^2}\Bigr)\\
&\quad\times\Bigl(N^{-\al_p}\bigl\|r^{\f{p-3}2}|\wt u_1|^{\f p2}\bigr\|_{L^2}\Bigr)^{\f2p}
\Bigl(N^{-\al_p}\bigl\|r^{\f{p-3}2}|\wt u_1|^{\f p2}\bigr\|_{L^6}\Bigr)^{1-\f4p}
N^{-1+\f3p-\f2p\beta_p}\\
&\ +\sum_{j=2}^\infty
\Bigl((jN)^{1+\beta_p}\bigl\|r^{\f{p-5}2}|\wt u_j|^{\f p2}\bigr\|_{L^2}\Bigr)^{\f12-\f3{2p}}
\Bigl((jN)^{\beta_p}\bigl\|r^{\f{p-3}2}|\wt u_j|^{\f p2}\bigr\|_{L^6}\Bigr)^{\f7{2p}-\f12}\\
&\quad\times\Bigl(((j+1)N)^{1+\beta_p}\bigl\|r^{\f{p-5}2}|\wt u_{j+1}|^{\f p2}\bigr\|_{L^2}\Bigr)^{\f12-\f3{2p}}
\Bigl(((j+1)N)^{\beta_p}\bigl\|r^{\f{p-3}2}|\wt u_{j+1}|^{\f p2}\bigr\|_{L^6}\Bigr)^{\f7{2p}-\f12}\\
&\quad\times\Bigl(N^{-\al_p}\bigl\|r^{\f{p-3}2}|\wt\nabla\wt u_1|
|\wt u_1|^{\f p2-1}\bigr\|_{L^2}
+N^{1-\al_p}\bigl\|r^{\f{p-5}2}|\wt u_1|^{\f p2}\bigr\|_{L^2}\Bigr)\\
&\quad\times\Bigl(N^{-\al_p}\bigl\|r^{\f{p-3}2}|\wt u_1|^{\f p2}\bigr\|_{L^2}\Bigr)^{\f2p}
\Bigl(N^{-\al_p}\bigl\|r^{\f{p-3}2}|\wt u_1|^{\f p2}\bigr\|_{L^6}\Bigr)^{1-\f4p}
 N^{-\f2p\al_p}(jN)^{-1+\f3p-\f4p\beta_p},
\end{split}\eeq
where we have used the fact that $\f{j+1}j\in]1,\f32]$ for every $j\geq2$.

Observing that for $5<p<6$ and $\beta_p>1,$ one has
$$\left(-1+\f3p-\f4p\beta_p\right)\times \f{p}{p-2}<-\f{p+1}{p-2}<-\f74\Rightarrow
\sum_{j=2}^\infty j^{\left(-1+\f3p-\f4p\beta_p\right)\f{p}{p-2}}<\infty.$$
Then we get, by integrating \eqref{S2eq4} in time over $[0,t]$ and using \eqref{S2eq3},  that
\begin{equation}\begin{split}\label{3.7}
N^{-2\al_p}\int_0^t I^{(4,1)}_1\,dt'
\lesssim& N^{-1+\f3p-\f2p\beta_p}E_p^{1+\f1p}\\
&+N^{-1+\f3p-\f4p\beta_p-\f2p\al_p}E_p^{1+\f1p}\times
\Bigl(\sum_{j=2}^\infty j^{\left(-1+\f3p-\f4p\beta_p\right)\f{p}{p-2}}\Bigr)^{1-\f2p} \\
\lesssim& N^{-1+\f3p-\f2p\beta_p}E_p^{1+\f1p}.
\end{split}\end{equation}

\no$\bullet$~{\bf The estimate of $I_k^{(4,1)}$ when $k=2$.}

The estimate of $I^{(4,1)}_2$ is similar to that of $I^{(4,1)}_1$, except now the $k_1+k_2=2$ part in the summation is nontrivial.
Indeed it allows $k_1=k_2=1$, which corresponds to the term
\begin{align*}
&(2N)^{2\beta_p}\int_0^t\int_{\R^3}
|\wt u_1||\wt u_1|\bigl(|\wt\nabla\wt u_2|
+|r^{-1}\wt u_2|\bigr)|\wt u_2|^{p-2}r^{p-3}\,dxdt'\\
&\lesssim N^{-1+\f3p+\f2p\beta_p+\f4p\al_p}
\int_0^t\Bigl(N^{1-\al_p}\bigl\|r^{\f{p-5}2}|\wt u_1|^{\f p2}\bigr\|_{L^2}\Bigr)^{1-\f3{p}}
\Bigl(N^{-\al_p}\bigl\|r^{\f{p-3}2}|\wt u_1|^{\f p2}\bigr\|_{L^6}\Bigr)^{\f7{p}-1}\\
&\quad\times\Bigl(N^{\beta_p}\bigl\|r^{\f{p-3}2}|\wt\nabla\wt u_2|
|\wt u_2|^{\f p2-1}\bigr\|_{L^2}
+N^{1+\beta_p}\bigl\|r^{\f{p-5}2}
|\wt u_2|^{\f p2}\bigr\|_{L^2}\Bigr)\\
&\quad\times\Bigl(N^{\beta_p}\bigl\|r^{\f{p-3}2}|\wt u_2|^{\f p2}\bigr\|_{L^2}\Bigr)^{\f2p}
\Bigl(N^{\beta_p}\bigl\|r^{\f{p-3}2}|\wt u_2|^{\f p2}\bigr\|_{L^6}\Bigr)^{1-\f4p}\,dt'\\
&\lesssim N^{-1+\f3p+\f2p\beta_p+\f4p\al_p}E_p^{1+\f1p}.
\end{align*}

As for the estimate of the other part in $I^{(4,1)}_2$, namely
\begin{equation}\begin{split}\label{3.7.1}
&(2N)^{2\beta_p}\sum_{|k_1-k_2|=2}
\int_0^t\int_{\R^3}
|\wt u_{k_1}||\wt u_{k_2}|\bigl(|\wt\nabla\wt u_2|
+|r^{-1}\wt u_2|\bigr)|\wt u_2|^{p-2}r^{p-3}\,dx\,dt'\\
&=2(2N)^{2\beta_p}\sum_{j=1}^\infty \int_0^t
\int_{\R^3}|\wt{u}_j||\wt u_{j+2}|
\bigl(|\wt\nabla\wt u_2|
+|r^{-1}\wt u_2|\bigr)
|\wt{u}_2|^{p-2}r^{p-3}\,dx\,dt',
\end{split}\end{equation}
we observe that \eqref{3.7.1} has the same structure as that of $N^{-2\al_p}\int_0^t |I^{(4,1)}_1|\,dt'$.
Then it follows from the same derivation of \eqref{3.7} that \eqref{3.7.1} can also be bounded by $N^{-1+\f3p-\f2p\beta_p}E_p^{1+\f1p}.$

By summarizing the above estimates  and using the fact that
$-1+\f3p+\f2p\beta_p+\f4p\al_p<0,$ which is guaranteed
by the assumption $\al_p<\f{p-3-2\beta_p}4$, we achieve
\begin{equation}\begin{split}\label{3.8}
(2N)^{2\beta_p}\int_0^t I^{(4,1)}_2\,dt'
&\lesssim \bigl(N^{-1+\f3p+\f2p\beta_p+\f4p\al_p} +N^{-1+\f3p-\f2p\beta_p}\bigr)E_p^{1+\f1p}\\
&\lesssim E_p^{1+\f1p}.
\end{split}\end{equation}

\no$\bullet$~{\bf The estimate of $I_k^{(4,1)}$ when $k\geq3$.}

The case for $k\geq3$ is more complicated.
And we shall decompose $I_k^{(4,1)}$ into three parts.
{\bf Case 1:} The terms with $k_1$ or $k_2$ equals $1$, which we denote as
$$I_k^{(4,1,1)}\eqdef\int_{\R^3}|\wt u_{1}||\wt u_{k\pm1}|\bigl(|\wt\nabla\wt u_k|
+|r^{-1}\wt u_k|\bigr)|\wt u_k|^{p-2}r^{p-3}\,dx.$$
By using H\"older's inequality, we find
\begin{align*}
(k&N)^{2\beta_p}I_k^{(4,1,1)}
\lesssim
\Bigl(N^{1-\al_p}\bigl\|r^{\f{p-5}2}|\wt u_1|^{\f p2}\bigr\|_{L^2}\Bigr)^{\f12-\f3{2p}}
\Bigl(N^{-\al_p}\bigl\|r^{\f{p-3}2}|\wt u_1|^{\f p2}\bigr\|_{L^6}\Bigr)^{\f7{2p}-\f12}\\
&\times\Bigl(\bigl((k\pm1)N\bigr)^{1+\beta_p}
\bigl\|r^{\f{p-5}2}|\wt u_{k\pm1}|^{\f p2}\bigr\|_{L^2}\Bigr)^{\f12-\f3{2p}}
\Bigl(\bigl((k\pm1)N\bigr)^{\beta_p}
\bigl\|r^{\f{p-3}2}|\wt u_{k\pm1}|^{\f p2}\bigr\|_{L^6}\Bigr)^{\f7{2p}-\f12}\\
&\times\Bigl((kN)^{\beta_p}\bigl\|r^{\f{p-3}2}|\wt\nabla\wt u_{k}|
|\wt u_{k}|^{\f p2-1}\bigr\|_{L^2}
+(kN)^{1+\beta_p}\bigl\|r^{\f{p-5}2}
|\wt u_{k}|^{\f p2}\bigr\|_{L^2}\Bigr)\\
&\times\Bigl((kN)^{\beta_p}\bigl\|r^{\f{p-3}2}|\wt u_{k}|^{\f p2}\bigr\|_{L^2}\Bigr)^{\f2p}
\Bigl((kN)^{\beta_p}\bigl\|r^{\f{p-3}2}|\wt u_{k}|^{\f p2}\bigr\|_{L^6}\Bigr)^{1-\f4p}\\
&\times(kN)^{-\f12+\f3{2p}}
 N^{-\f12+\f3{2p}+\f2p\al_p},
\end{align*}
here we  used the fact that $\f{k\pm1}k\in[\f23,\f43]$
for $k\geq3$.
By integrating the above inequality over $[0,t]$  and then summing up the resulting inequalities for $k\geq3,$ we obtain
\begin{align*}
\sum_{k\geq3}&(kN)^{2\beta_p}\int_0^t I_k^{(4,1,1)}\,dt'
\lesssim\Bigl(N^{2-2\al_p}\bigl\|r^{\f{p-5}2}|\wt u_1|^{\f p2}\bigr\|_{L^2_t(L^2)}^2\Bigr)^{\f14-\f3{4p}}\\
&\times\Bigl(N^{-2\al_p}\bigl\|r^{\f{p-3}2}|\wt u_1|^{\f p2}\bigr\|_{L^2_t(L^6)}^2\Bigr)^{\f7{4p}-\f14}\Bigl(\sum_{k\geq3}\bigl((k\pm1)N\bigr)^{2+2\beta_p}
\bigl\|r^{\f{p-5}2}|\wt u_{(k\pm1)}|^{\f p2}\bigr\|_{L^2_t(L^2)}^2\Bigr)^{\f14-\f3{4p}}\\
&\times\Bigl(\sum_{k\geq3}\bigl((k\pm1)N\bigr)^{2\beta_p}
\bigl\|r^{\f{p-3}2}|\wt u_{(k\pm1)}|^{\f p2}\bigr\|_{L^2_t(L^6)}^2\Bigr)^{\f7{4p}-\f14}\\
&\times\Bigl(\sum_{k\geq3}(kN)^{2\beta_p}\bigl\|r^{\f{p-3}2}|\wt\nabla\wt u_{k}|
|\wt u_{k}|^{\f p2-1}\bigr\|_{L^2_t(L^2)}^2
+(kN)^{2+2\beta_p}\bigl\|r^{\f{p-5}2}
|\wt u_{k}|^{\f p2}\bigr\|_{L^2_t(L^2)}^2\Bigr)^{\f12}\\
&\times\Bigl(\sum_{k\geq3}(kN)^{2\beta_p}\bigl\|r^{\f{p-3}2}|\wt u_{k}|^{\f p2}\bigr\|_{L^\infty_t(L^2)}^2\Bigr)^{\f1p}
\Bigl(\sum_{k\geq3}(kN)^{2\beta_p}\bigl\|r^{\f{p-3}2}|\wt u_{k}|^{\f p2}\bigr\|_{L^2_t(L^6)}^2\Bigr)^{\f12-\f2p}\\
&\times\sup_{k\geq3}k^{-\f12+\f3{2p}}
 N^{-1+\f3{p}+\f2p\al_p}.
\end{align*}
Due to $-\f12+\f3{2p}<0,$ we have
$\sup_{k\geq3}k^{-\f12+\f3{2p}}
=3^{-\f12+\f3{2p}}.$ As a result, it comes out
\begin{equation}\begin{split}\label{3.9}
\sum_{k\geq3} (kN)^{2\beta_p}\int_0^t I_k^{(4,1,1)}\,dt'
\lesssim N^{-1+\f3p+\f2p\al_p}E_p^{1+\f1p}.
\end{split}\end{equation}

\no{\bf Case 2:} The terms with $k_1$ or $k_2$ lying  in $[2,k-2],~k\geq4,$ and we denote this part as $$I_k^{(4,1,2)}\eqdefa\sum_{j=2}^{k-2}
\int_{\R^3}|\wt u_{j}||\wt u_{k\pm j}|
\bigl(|\wt\nabla\wt u_k|
+|r^{-1}\wt u_k|\bigr)|\wt u_k|^{p-2}r^{p-3}\,dx.$$
By applying H\"older's inequality, we obtain
\begin{equation}\begin{split}\label{3.10}
&(kN)^{2\beta_p}I_k^{(4,1,2)}\lesssim
\sum_{j=2}^{k-2}\Bigl((jN)^{1+\beta_p}\bigl\|r^{\f{p-5}2}|\wt u_j|^{\f p2}\bigr\|_{L^2}\Bigr)^{\f12-\f3{2p}}
\Bigl((jN)^{\beta_p}\bigl\|r^{\f{p-3}2}|\wt u_j|^{\f p2}\bigr\|_{L^6}\Bigr)^{\f7{2p}-\f12}\\
&\quad \times\Bigl(\bigl((k\pm j)N\bigr)^{1+\beta_p}
\bigl\|r^{\f{p-5}2}|\wt u_{k\pm j}|^{\f p2}\bigr\|_{L^2}\Bigr)^{\f12-\f3{2p}}
\Bigl(\bigl((k\pm j)N\bigr)^{\beta_p}
\bigl\|r^{\f{p-3}2}|\wt u_{k\pm j}|^{\f p2}\bigr\|_{L^6}\Bigr)^{\f7{2p}-\f12}\\
&\quad \times\Bigl((kN)^{\beta_p}\bigl\|r^{\f{p-3}2}|\wt\nabla\wt u_{k}|
|\wt u_{k}|^{\f p2-1}\bigr\|_{L^2}
+(kN)^{1+\beta_p}\bigl\|r^{\f{p-5}2}
|\wt u_{k}|^{\f p2}\bigr\|_{L^2}\Bigr)\\
&\quad \times\Bigl((kN)^{\beta_p}\bigl\|r^{\f{p-3}2}|\wt u_{k}|^{\f p2}\bigr\|_{L^2}\Bigr)^{\f2p}
\Bigl((kN)^{\beta_p}\bigl\|r^{\f{p-3}2}|\wt u_{k}|^{\f p2}\bigr\|_{L^6}\Bigr)^{1-\f4p}\\
&\quad \times\bigl(j (k\pm j)N^2\bigr)^{-\f12+\f3{2p}-\f2p\beta_p}
 (kN)^{\f2p\beta_p}.
\end{split}\end{equation}
And the terms containing $j$ in \eqref{3.10} can be bounded by
\begin{align*}
&\Bigl(\sum_{j=2}^{k-2}(jN)^{2+2\beta_p}\bigl\|r^{\f{p-5}2}|\wt u_j|^{\f p2}\bigr\|_{L^2}^2\Bigr)^{\f14-\f3{4p}}
\Bigl(\sum_{j=2}^{k-2}(jN)^{2\beta_p}\bigl\|r^{\f{p-3}2}|\wt u_j|^{\f p2}\bigr\|_{L^6}^2\Bigr)^{\f7{4p}-\f14}\\
&\ \times\Bigl(\sum_{j=2}^{k-2}\bigl((k\pm j)N\bigr)^{2+2\beta_p}
\bigl\|r^{\f{p-5}2}|\wt u_{(k\pm j)}|^{\f p2}\bigr\|_{L^2}^2\Bigr)^{\f14-\f3{4p}}
\Bigl(\sum_{j=2}^{k-2}\bigl((k\pm j)N\bigr)^{2\beta_p}
\bigl\|r^{\f{p-3}2}|\wt u_{(k\pm j)}|^{\f p2}\bigr\|_{L^6}^2\Bigr)^{\f7{4p}-\f14}\\
&\ \times\Bigl(\sum_{j=2}^{k-2}\bigl(j(k\pm j)N^2\bigr)^{\left(-\f12+\f3{2p}-\f2p\beta_p\right)\times\f{p}{p-2}}\Bigr)^{\f{p-2}p}.
\end{align*}
Observing that $\beta_p<\f{p-1}4$ ensures that
$$s\eqdef\left(-\f12+\f3{2p}-\f2p\beta_p\right)\times\f{p}{p-2}>-1,$$
so that there holds
\begin{align*}
\lim_{k\rightarrow\infty}\sum_{j=2}^{k-2}\Bigl(\f{j(k\pm j)}{k^2}\Bigr)^s\times\f1k
=\lim_{k\rightarrow\infty}\sum_{j=2}^{k-2}\Bigl(\f{j}k\bigl(1\pm\f jk\bigr)\Bigr)^s\times\f1k
=\int_0^1x^s(1\pm x)^s\,dx<\infty,
\end{align*}
which in particular implies for any $k\geq4$ that
$$\sum_{j=2}^{k-2} j^s(k\pm j)^s\lesssim k^{1+2s},$$
As a result, it comes out
\begin{equation}\label{3.11}
\Bigl(\sum_{j=2}^{k-2}\bigl(j(k\pm j)N^2\bigr)^{\left(-\f12+\f3{2p}-\f2p\beta_p\right)\f{p}{p-2}}\Bigr)^{\f{p-2}p}\lesssim k^{\f1p-\f4p\beta_p}N^{-1+\f3p-\f4p\beta_p}.
\end{equation}

By substituting the above estimates into \eqref{3.10} and  then integrating the resulting inequality in time over $[0,t],$ we find
\begin{equation}\begin{split}\label{3.12}
(kN)^{2\beta_p}\int_0^t& I_k^{(4,1,2)}\,dt'
\lesssim  k^{\f1p-\f2p\beta_p} N^{-1+\f3p-\f2p\beta_p}E_p^{\f2p}\\
&\times\Bigl((kN)^{2\beta_p}\bigl\|r^{\f{p-3}2}|\wt\nabla\wt u_{k}|
|\wt u_{k}|^{\f p2-1}\bigr\|_{L^2_t(L^2)}^2
+(kN)^{2+2\beta_p}\bigl\|r^{\f{p-5}2}
|\wt u_{k}|^{\f p2}\bigr\|_{L^2_t(L^2)}^2\Bigr)^{\f12}\\
&\times\Bigl((kN)^{2\beta_p}\bigl\|r^{\f{p-3}2}|\wt u_{k}|^{\f p2}\bigr\|_{L^\infty_t(L^2)}^2\Bigr)^{\f1p}
\Bigl((kN)^{2\beta_p}\bigl\|r^{\f{p-3}2}|\wt u_{k}|^{\f p2}
\bigr\|_{L^2_t(L^6)}^2\Bigr)^{\f12-\f2p}.
\end{split}\end{equation}
Our choice of $\beta_p>1$ ensures that
$$\Bigl(\f1p-\f2p\beta_p\Bigr) p<-1,$$
so that we get, by summing up  \eqref{3.12} for $k\geq 4,$ that
\begin{equation}\begin{split}\label{3.13}
\sum_{k\geq4}&(kN)^{2\beta_p} \int_0^tI_k^{(4,1,2)}\,dt'
\lesssim  \Bigl(\sum_{k\geq4}k^{\left(\f1p-\f2p\beta_p\right) p}\Bigr)^{\f1p} N^{-1+\f3p-\f2p\beta_p}E_p^{\f2p}\\
&\times\Bigl(\sum_{k\geq4} (kN)^{2\beta_p}\Bigl(\bigl\|r^{\f{p-3}2}|\wt\nabla\wt u_{k}|
|\wt u_{k}|^{\f p2-1}\bigr\|_{L^2_t(L^2)}^2
+(kN)^{2+2\beta_p}\bigl\|r^{\f{p-5}2}
|\wt u_{k}|^{\f p2}\bigr\|_{L^2_t(L^2)}^2\Bigr)\Bigr)^{\f12}\\
&\times\Bigl(\sum_{k\geq4}(kN)^{2\beta_p}\bigl\|r^{\f{p-3}2}|\wt u_{k}|^{\f p2}\bigr\|_{L^\infty_t(L^2)}^2\Bigr)^{\f1p}
\Bigl(\sum_{k\geq4}(kN)^{2\beta_p}\bigl\|r^{\f{p-3}2}|\wt u_{k}|^{\f p2}\bigr\|_{L^2_t(L^6)}^2\Bigr)^{\f12-\f2p}\\
&\lesssim N^{-1+\f3p-\f2p\beta_p} E_p^{1+\f1p}.
\end{split}\end{equation}

\no{\bf Case 3:} The terms with both $k_1$ and $k_2$ being bigger than $k-1$, and we denote this part as
$$I_k^{(4,1,3)}\eqdefa\sum_{j=k-1}^{\infty}
\int_{\R^3}|\wt u_{j}||\wt u_{k+j}|
\bigl(|\wt\nabla\wt u_k|
+|r^{-1}\wt u_k|\bigr)|\wt u_k|^{p-2}r^{p-3}\,dx.$$
Then we get, by a
similar derivation of \eqref{3.10}, that
\begin{align*}
(kN&)^{2\beta_p}I_k^{(4,1,3)}
\lesssim\sum_{j=k-1}^{\infty}
\Bigl((jN)^{1+\beta_p}\bigl\|r^{\f{p-5}2}|\wt u_j|^{\f p2}\bigr\|_{L^2}\Bigr)^{\f12-\f3{2p}}
\Bigl((jN)^{\beta_p}\bigl\|r^{\f{p-3}2}|\wt u_j|^{\f p2}\bigr\|_{L^6}\Bigr)^{\f7{2p}-\f12}\\
&\times\Bigl(\bigl((k+j)N\bigr)^{1+\beta_p}
\bigl\|r^{\f{p-5}2}|\wt u_{(k+j)}|^{\f p2}\bigr\|_{L^2}\Bigr)^{\f12-\f3{2p}}
\Bigl(\bigl((k+j)N\bigr)^{\beta_p}
\bigl\|r^{\f{p-3}2}|\wt u_{(k+j)}|^{\f p2}\bigr\|_{L^6}\Bigr)^{\f7{2p}-\f12}\\
&\times\Bigl((kN)^{\beta_p}\bigl\|r^{\f{p-3}2}|\wt\nabla\wt u_{k}|
|\wt u_{k}|^{\f p2-1}\bigr\|_{L^2}
+(kN)^{1+\beta_p}\bigl\|r^{\f{p-5}2}
|\wt u_{k}|^{\f p2}\bigr\|_{L^2}\Bigr)\\
&\times\Bigl((kN)^{\beta_p}\bigl\|r^{\f{p-3}2}|\wt u_{k}|^{\f p2}\bigr\|_{L^2}\Bigr)^{\f2p}
\Bigl((kN)^{\beta_p}\bigl\|r^{\f{p-3}2}|\wt u_{k}|^{\f p2}\bigr\|_{L^6}\Bigr)^{1-\f4p}\\
&\times\bigl(j(k+j)N^2\bigr)^{-\f12+\f3{2p}-\f2p\beta_p}
 (kN)^{\f2p\beta_p}.
\end{align*}
Due to  $\f{k+j}j\in]1,3[$ for any $j\geq k-1,$ one has
\begin{align*}
\Bigl(\sum_{j=k-1}^{\infty}\bigl(j(k+j)\bigr)
^{\left(-\f12+\f3{2p}-\f2p\beta_p\right)\f{p}{p-2}}\Bigr)^{\f{p-2}p}
&\lesssim \Bigl(\sum_{j=k-1}^{\infty}
j^{\left(-1+\f3p-\f4p\beta_p\right)\f{p}{p-2}}\Bigr)^{\f{p-2}p}
\\
&\lesssim k^{\f1p-\f4p\beta_p},
\end{align*}
from which, we deduce
\begin{align*}
(kN)^{2\beta_p}\int_0^t&I_k^{(4,1,3)}\,dt'
\lesssim  k^{\f1p-\f2p\beta_p} N^{-1+\f3p-\f2p\beta_p} E_p^{\f2p}\\
&\times\Bigl((kN)^{2\beta_p}\bigl\|r^{\f{p-3}2}|\wt\nabla\wt u_{k}|
|\wt u_{k}|^{\f p2-1}\bigr\|_{L^2_t(L^2)}^2
+(kN)^{2+2\beta_p}\bigl\|r^{\f{p-5}2}
|\wt u_{k}|^{\f p2}\bigr\|_{L^2_t(L^2)}^2\Bigr)^{\f12}\\
&\times\Bigl((kN)^{2\beta_p}\bigl\|r^{\f{p-3}2}|\wt u_{k}|^{\f p2}\bigr\|_{L^\infty_t(L^2)}^2\Bigr)^{\f1p}
\Bigl((kN)^{2\beta_p}\bigl\|r^{\f{p-3}2}|\wt u_{k}|^{\f p2}
\bigr\|_{L^2_t(L^6)}^2\Bigr)^{\f12-\f2p}.
\end{align*}
By summing up the above inequality for $k\geq 3,$ we find
\begin{equation}\begin{split}\label{3.14}
\sum_{k\geq3}(kN)^{2\beta_p}\int_0^t I_k^{(4,1,3)}\,dt'
\lesssim N^{-1+\f3p-\f2p\beta_p} E_p^{1+\f1p}.
\end{split}\end{equation}

By summarizing the estimates \eqref{3.7},
~\eqref{3.8},~\eqref{3.9},~\eqref{3.13} and \eqref{3.14}, we achieve
\begin{equation}\label{3.15}
\int_0^t\Bigl(N^{-2\al_p}I^{(4,1)}_1
+\sum_{k\geq2}(kN)^{2\beta_p} I_k^{(4,1)}\Bigr)\,dt'
\lesssim E_p^{1+\f1p}.
\end{equation}

\no$\bullet$~{\bf The estimate of $I_k^{(4,2)}$.}

Comparing $I_k^{(4,2)}$ with $I_k^{(4,1)}$ given by \eqref{3.6}, the terms $\wt\nabla\wt u_k$
and $\wt u_{k_1}\wt u_{k_2}$ in the integrand of $I_k^{(4,1)}$ is replaced by $r^{-1}\wt u_k$ and $\vt_{k_1}\vt_{k_2}$ in that of $I_k^{(4,2)}$ respectively. While in view of \eqref{defEpr} and \eqref{defEpth},
 there is an additional
$(kN)^{2}$ in the front of $\bigl\|r^{\f{p-5}2}|\wt u_{k}|^{\f p2}\bigr\|_{L^2_t(L^2)}^2$
than $\bigl\|r^{\f{p-3}2}|\nabla\wt u_k|(\ur_k)^{\f p2-1}\bigr\|_{L^2_t(L^2)}^2$
in $E_p^{r,z}$, and there is an additional $(kN)^{\f p2}$ in the front of the terms
in $E_p^\th$ than the terms of the same type in $E_p^{r,z}$.
So that exactly along the same line to the derivation of
\eqref{3.15}, we  deduce that
\begin{equation}\label{3.16}
\int_0^t\Bigl(N^{-2\al_p}I^{(4,2)}_1
+\sum_{k\geq2}(kN)^{2\beta_p} I_k^{(4,2)}\Bigr)\,dt'
\lesssim N^{-2} E_p^{1+\f1p}.
\end{equation}

\no$\bullet$~{\bf The estimate of $I_k^{(4,3)}$.}

Finally, let us handle the remaining term $I_k^{(4,3)}$ in \eqref{3.6},
which we shall decompose into $I_k^{(4,3,1)}+I_k^{(4,3,2)}$ according to the relations between $k_1,~k_2$ and $k$, precisely,
let $\Omega_k$ be given by \eqref{defOmk}, we denote
\beq \label{3.16a}\begin{split}
I_k^{(4,3,1)}\eqdef&
\sum_{\substack{(k_1,k_2)\in\Omega_k,\\
k_2\leq2k}}(k_1+k_2)N\int_{\R^3}
{|\vt_{k_1}||\wt u_{k_2}}||\wt u_k|^{p-1}
 r^{p-4}\,dx,\\
I_k^{(4,3,2)}\eqdef &\sum_{j=2k+1}^\infty{(2j\pm k)N}\int_{\R^3}
|\vt_{j\pm k}||\wt u_j|
|\wt u_k|^{p-1} r^{p-4}\,dx.
 \end{split} \eeq

Let us first focus on $I_k^{(4,3,1)}$. There holds for any $(k_1,k_2)\in\Omega_k$ with ${k_2}\leq 2k$ that
$$k_1+k_2\leq k+2k_2\leq5k.$$
Then it is easy to verify that
$I_k^{(4,3,1)}$ shares a similar estimate as $I_k^{(4,1)}$
in \eqref{3.15}. In fact, one only needs to modify the estimate of $\wt u_{k_1}$
in $I_k^{(4,1)}$ to $\vt_{k_1}$ in $I_k^{(4,3,1)}$, which gives an additional $N^{-\f12}$
 due to the different weights in the definitions of  $E_p^\th$ and $E_p^{r,z}$ given respectively by \eqref{defEpr}
 and \eqref{defEpth}. Hence we have
\begin{equation}
\int_0^t\Bigl(N^{-2\al_p}I^{(4,3,1)}_1
+\sum_{k\geq2}(kN)^{2\beta_p} I_k^{(4,3,1)}\Bigr)\,dt'
\lesssim N^{-\f12} E_p^{1+\f1p}.
\end{equation}

For $I_k^{(4,3,2)},$
 we get, by using H\"older's inequality, that
\begin{align*}
N^{-2\al_p}I_1^{(4,3,2)}
\lesssim&\sum_{j=3}^\infty
\Bigl(\bigl((j\pm 1)N\bigr)^{1+\f p4+\beta_p}
\bigl\|r^{\f{p-5}2}|\vt_{j\pm 1}|^{\f p2}\bigr\|_{L^2}\Bigr)^{\f2p}
\Bigl((jN)^{1+\beta_p}\bigl\|r^{\f{p-5}2}|\wt u_j|^{\f p2}\bigr\|_{L^2}\Bigr)^{\f2p}\\
&\qquad\times\Bigl(N^{1-\al_p}
\bigl\|r^{\f{p-5}2}|\wt u_1|^{\f p2}\bigr\|_{L^2}\Bigr)^{2-\f7p}
\Bigl(N^{-\al_p}
\bigl\|r^{\f{p-3}2}|\wt u_1|^{\f p2}\bigr\|_{L^2}\Bigr)^{\f2p}\\
&\qquad\times\Bigl(N^{-\al_p}\bigl\|r^{\f{p-3}2}|\wt u_1|^{\f p2}\bigr\|_{L^6}\Bigr)^{\f3p}
 j^{\f12-\f4p(1+\beta_p)} N^{-\f32+\f3p-\f2p\al_p-\f4p\beta_p},
\end{align*}
and for $k\geq2$ that
\begin{align*}
(kN)^{2\beta_p}I_k^{(4,3,2)}
\lesssim&\sum_{j=2k+1}^\infty
\Bigl(\bigl((j\pm k)N\bigr)^{1+\f p4+\beta_p}
\bigl\|r^{\f{p-5}2}|\vt_{j\pm k}|^{\f p2}\bigr\|_{L^2}\Bigr)^{\f2p}
\Bigl((jN)^{1+\beta_p}\bigl\|r^{\f{p-5}2}|\wt u_j|^{\f p2}\bigr\|_{L^2}\Bigr)^{\f2p}\\
&\qquad\times\Bigl((kN)^{1+\beta_p}
\bigl\|r^{\f{p-5}2}|\wt u_{k}|^{\f p2}\bigr\|_{L^2}\Bigr)^{2-\f7p}
\Bigl((kN)^{\beta_p}
\bigl\|r^{\f{p-3}2}|\wt u_{k}|^{\f p2}\bigr\|_{L^2}\Bigr)^{\f2p}\\
&\qquad\times\Bigl((kN)^{\beta_p}\bigl\|r^{\f{p-3}2}|\wt u_{k}|^{\f p2}\bigr\|_{L^6}\Bigr)^{\f3p}
 j^{\f12-\f4p(1+\beta_p)} k^{-2+\f7p+\f2p\beta_p}
 N^{-\f32+\f3p-\f2p\beta_p},
\end{align*}
where we used the fact that $2j\pm k\thicksim j\pm k\thicksim j$ for any $j\geq2k+1$. By integrating the above inequalities over $[0,t]$ and summing up the resulting inequalities for $k\geq 1,$
we obtain
\begin{align*}
\int_0^t&\Bigl(N^{-2\al_p}I^{(4,3,2)}_1
+\sum_{k\geq2}(kN)^{2\beta_p} I_k^{(4,3,2)}\Bigr)\,dt'\\
\lesssim& E_p^{1+\f1p} N^{-\f32+\f3p-\f2p\al_p-\f4p\beta_p}
\Bigl(\sum_{j=3}^\infty j^{\bigl(\f12-\f4p(1+\beta_p)\bigr)
\f{p}{p-2}}\Bigr)^{1-\f2p}+E_p^{\f2p} N^{-\f32+\f3p-\f2p\beta_p}\\
&\times\sum_{k\geq2}\Bigl\{\Bigl((kN)^{2+2\beta_p}
\bigl\|r^{\f{p-5}2}|\wt u_{k}|^{\f p2}\bigr\|_{L^2_t(L^2)}^2\Bigr)^{1-\f7{2p}}
\Bigl((kN)^{2\beta_p}
\bigl\|r^{\f{p-3}2}|\wt u_{k}|^{\f p2}\bigr\|_{L^\infty_t(L^2)}^2\Bigr)^{\f1p}\\
&\times\Bigl((kN)^{2\beta_p}\bigl\|r^{\f{p-3}2}|\wt u_{k}|
^{\f p2}\bigr\|_{L^2_t(L^6)}^2\Bigr)^{\f3{2p}}
k^{-2+\f7p+\f2p\beta_p}
\Bigl(\sum_{j=2k+1}^\infty
j^{\bigl(\f12-\f4p(1+\beta_p)\bigr)\f{p}{p-2}}\Bigr)^{1-\f2p}\Bigr\}.
\end{align*}
It follows from  $p<6$ and $\beta_p>1$  that
$$\left(\f12-\f4p(1+\beta_p)\right)\f{p}{p-2}<-1,$$
so that one has
$$\Bigl(\sum_{j=3}^\infty j^{\bigl(\f12-\f4p(1+\beta_p)\bigr)
\f{p}{p-2}}\Bigr)^{1-\f2p}\lesssim1\andf
\Bigl(\sum_{j=2k+1}^\infty
j^{\bigl(\f12-\f4p(1+\beta_p)\bigr)\f{p}{p-2}}\Bigr)^{1-\f2p}
\lesssim k^{\f32-\f4p\bigl(\f32+\beta_p\bigr)}.$$
Therefore, we  obtain
\begin{equation}\begin{split}\label{3.18}
&\int_0^t\Bigl(N^{-2\al_p}I^{(4,3,2)}_1
+\sum_{k\geq2}(kN)^{2\beta_p} I_k^{(4,3,2)}\Bigr)\,dt'\\
&\lesssim N^{-\f32+\f3p-\f2p\al_p-\f4p\beta_p} E_p^{1+\f1p}
+ N^{-\f32+\f3p-\f2p\beta_p}E_p^{1+\f1p}
\Bigl(\sum_{k\geq2}k^{\left(-\f12+\f1p-\f2p\beta_p\right) p}\Bigr)^{\f1p}\\
&\lesssim N^{-\f32+\f3p-\f2p\beta_p} E_p^{1+\f1p},
\end{split}\end{equation}
where in the last step, we  used the fact that $\left(-\f12+\f1p-\f2p\beta_p\right) p<-1$
for $\beta_p>1$ and $p>5$.

By summarizing the estimates \eqref{3.15}-\eqref{3.18}, we obtain  the estimate of the part concerning
 $F^r_k\ur_k$  in $I_k^{(4)}$. The
$F^z_k\uz_k$ related part can be handled  exactly along the same line.
As a result, we  achieve \eqref{3.19}. This completes the proof of Lemma \ref{S2lem1}.
\end{proof}

\begin{proof}[{\bf Proof of Lemma \ref{S2lem2}}]
In view of \eqref{defcLm}, we deduce from \eqref{eqtuk} that
\begin{equation}\begin{split}\label{3.20}
\cL_{kN} P_k
=&\bigl(\pa_r+\f1r\bigr)\bigl(\wt u_0\cdot\wt\nabla\ur_k
+\wt u_k\cdot\wt\nabla\ur_0\bigr)
+\f{kN}r\Bigl(\wt u_0\cdot\wt\nabla\vt_k
+\f{\ur_0\vt_k}r\Bigr)\\
&+\pa_z\bigl(\wt u_0\cdot\wt\nabla\uz_k
+\wt u_k\cdot\wt\nabla\uz_0\bigr)
-\bigl(\pa_r+\f1r\bigr)\D\ur_k-\f{kN}r\D\vt_k-\pa_z\D\uz_k\\
&+\bigl(\pa_t+\f{1+k^2N^2}{r^2}\bigr)\Bigl(\pa_r\ur_{k}+\frac{\ur_k}r
+\pa_z\uz_{k}+kN\f{\vt_{k}}r\Bigr)\\
&-\f{2}{r^3}\ur_k{
+2kN\bigl(\pa_r+\f1r\bigr)\Bigl(\f{\vt_k}{r^2}\Bigr)}
-\f{1}{r^2}\pa_z\uz_k
-\Bigl(\bigl(\pa_r+\f1r\bigr)F^r_k
+\f{kN}r F^\th_k+\pa_z F^z_k\Bigr).
\end{split}\end{equation}
It is worth mentioning that the third line of \eqref{3.20} contains term of order $N^3$, which seems
impossible to be absorbed by the left-hand side of \eqref{3.5}. Fortunately,
this line of terms actually vanishes due to the  fourth equation of \eqref{eqtuk}. So that we can rewrite \eqref{3.20} as
\begin{equation}\label{3.21}
\begin{split}
&P_k=\cL_{kN}^{-1}\cA_k^{(1)}+\cL_{kN}^{-1}\cA_k^{(2)}+\cL_{kN}^{-1}\cA_k^{(3)}
+\cL_{kN}^{-1}\cA_k^{(4)}\with\\
&\cA_k^{(1)}\eqdefa\bigl(\pa_r+\f1r\bigr)
\Bigl(\f{\ur_0(2\ur_k+kN\vt_k)}{r}\Bigr)
+\p_z\f{\ur_0u_k^z+\ur_ku_0^z+kN\vt_ku_0^z}r\\
&\qquad\qquad+\f1r\wt\nabla\cdot\bigl(\wt u_0 \ur_k
+\wt u_k\ur_0+kN\wt u_0\vt_k\bigr)+2kN\f{\ur_0\vt_k}{r^2},\\
&\cA_k^{(2)}\eqdefa\pa_r\wt\nabla\cdot\bigl(\wt u_0\ur_k
+\wt u_k\ur_0\bigr)+\pa_z\wt\nabla\cdot\bigl(\wt u_0\uz_k
+\wt u_k\uz_0\bigr),\\
&\cA_k^{(3)}\eqdefa\bigl[\D,\pa_r+\f1r\bigr]\ur_k
+\bigl[\D,\f{kN}r\bigr]\vt_k-\f{2}{r^3}\ur_k+{2kN\bigl(\pa_r+\f1r\bigr)
\Bigl(\f{\vt_k}{r^2}\Bigr)}-\f{1}{r^2}\pa_z\uz_k,\\
&\cA_k^{(4)}\eqdefa-\Bigl(\bigl(\pa_r+\f1r\bigr)F^r_k
+\f{kN}r F^\th_k+\pa_z F^z_k\Bigr).
\end{split} \end{equation}
Correspondingly, we decompose the following summation as
\begin{equation}\label{3.22}
\begin{split}
\int_0^t\Bigl(N^{-2\al_p}&|I^{(5)}_1|
+\sum_{k\geq2}(kN)^{2\beta_p} |I_k^{(5)}|\Bigr)\,dt'
\leq\sum_{j=1}^4\Psi^{(j)}\with\\
\Psi^{(j)}\eqdef &N^{-2\al_p}
\Bigl|\int_0^t\int_{\R^3} \wt\nabla\cL_{N}^{-1}\cA_1^{(j)}
|\wt u_1|^{p-2}\wt u_1 r^{p-3}\,dx\,dt'\Bigr|\\
&+\sum_{k\geq2}(kN)^{2\beta_p}
\Bigl|\int_0^t\int_{\R^3} \wt\nabla\cL_{kN}^{-1}\cA_k^{(j)}
|\wt u_k|^{p-2}\wt u_k r^{p-3}\,dx\,dt'\Bigr|.
\end{split} \end{equation}

In what follows, we shall estimate term by term in \eqref{3.22}.

\no$\bullet$~{\bf The estimate of $\Psi^{(1)}+\Psi^{(2)}.$}

Let us first consider the bilinear term $\cA_k^{(1)}$.
By applying Proposition \ref{lemLm}, we get
\begin{align*}
\bigl\|r^{2-\f5p}\cL_{kN}^{-1}\cA_k^{(1)}
\bigr\|_{L^p_t(L^p)}
&\lesssim_\e (kN)^{1-\f12+\f13+\e}
\bigl\|r^{1-\f5p}u_0\otimes\upsilon_k
\bigr\|_{L^p_t(L^{\f{3p}{p+3}})}\\
&\lesssim (kN)^{1-\f2p}\|u_0\|_{L^\infty_t(L^3)}
\bigl(kN\bigl\|r^{\f{p-5}2}|\upsilon_k|^{\f p2}\bigr\|_{L^2_t(L^2)}\bigr)^{\f2p},
\end{align*}
where in the last step we take $\e=\f16$ for instance.
Here and in all that follows, we always denote
$\upsilon_k\eqdef(\ur_k,\vt_k,\uz_k)$. Then by using integration by parts, we get
\begin{equation}\begin{split}\label{3.23}
&\Bigl|\int_0^t\int_{\R^3} \wt\nabla\cL_{kN}^{-1}\cA_k^{(1)}
|\wt u_k|^{p-2}\wt u_k r^{p-3}\,dx\,dt'\Bigr|\\
&= 2\pi\Bigl|\int_0^t\int_{\R^+}\int_{\R}
\cL_{kN}^{-1}\cA_k^{(j)}\wt\nabla\cdot
\bigl(r^{p-2}|\wt u_k|^{p-2}\wt u_k\bigr)\,drdzdt'\Bigr|\\
&\lesssim\bigl\|r^{2-\f5p}\cL_{kN}^{-1}
\cA_k^{(j)}\bigr\|_{L^p_t(L^p)}
\bigl\|\bigl(r^{\f{p-3}2}|\wt\nabla\wt u_k|
|\wt u_k|^{\f p2-1},r^{\f{p-5}2}|\wt u_k|^{\f p2}\bigr)\bigr\|_{L^2_t(L^2)}
\bigl\|r^{1-\f5p}\wt u_k\bigr\|_{L^p_t(L^p)}^{\f p2-1}\\
&\lesssim\|u_0\|_{L^\infty_t(L^3)}
\bigl(kN\bigl\|r^{\f{p-5}2}|\upsilon_k|^{\f p2}\bigr\|_{L^2_t(L^2)}\bigr)^{\f2p}
\bigl\|\bigl(r^{\f{p-3}2}|\wt\nabla\wt u_k|
|\wt u_k|^{\f p2-1},r^{\f{p-5}2}|\wt u_k|^{\f p2}\bigr)\bigr\|_{L^2_t(L^2)}\\
&\quad\times\bigl(kN\bigl\|r^{\f{p-5}2}|\wt u_k|^{\f p2}\bigr\|_{L^2_t(L^2)}\bigr)^{1-\f2p}.
\end{split}\end{equation}

The estimate for $\cA_k^{(2)}$ is much more complicated.
By applying Proposition \ref{lemLmweightedRiesz}, we find
\begin{equation}\begin{split}\label{3.24}
\bigl\|r^{1-\f3p}&\cL_{kN}^{-1}\cA_k^{(2)}\bigr\|_{L^p_t(L^{\f{3p}{p+1}})}
\lesssim\bigl\|\cL_{kN}^{-1}\cA_k^{(2)}\bigr\|_{L^p_t(L^{\f{3p}{2p-2}})}
+\bigl\|r^{1-\f3p}u_0\otimes\wt u_k\bigr\|_{L^p_t(L^{\f{3p}{p+1}})}\\
&+\bigl\|\cL_{kN}^{-1}\bigl(\f1r \wt\nabla(r^{1-\f3p}u_0\otimes\wt u_k)\bigr)
\bigr\|_{L^p_t(L^{\f{3p}{p+1}})}
+\bigl\|\cL_{kN}^{-1}\bigl(\f1{r^2}
(r^{1-\f3p}u_0\otimes\wt u_k)\bigr)
\bigr\|_{L^p_t(L^{\f{3p}{p+1}})}.
\end{split}\end{equation}
Yet it follows from  Propositions \ref{lemLm} and  \ref{lemLmRiesz} that
\begin{align*}
\bigl\|\cL_{kN}^{-1}\cA_k^{(2)}\bigr\|_{L^p_t(L^{\f{3p}{2p-2}})}
\lesssim&\bigl\|u_0\otimes\wt u_k\bigr\|_{L^p_t(L^{\f{3p}{2p-2}})}
+\bigl\|\cL_{kN}^{-1}\bigl(\f1r \wt\nabla(u_0\otimes\wt u_k)\bigr)
\bigr\|_{L^p_t(L^{\f{3p}{2p-2}})}\\
\lesssim&\bigl\|u_0\otimes\wt u_k\bigr\|_{L^p_t(L^{\f{3p}{2p-2}})}
+\bigl\|r^{1-\f{13}{2p}+\f3{p^2}}u_0\otimes\wt u_k\bigr\|
_{L^p_t(L^{\f{6p^2}{2p^2+9p-6}})},
\end{align*}
and
\begin{align*}
\bigl\|\cL_{kN}^{-1}\bigl(\f1r \wt\nabla(r^{1-\f3p}u_0\otimes\wt u_k)\bigr)
\bigr\|_{L^p_t(L^{\f{3p}{p+1}})}
&+\bigl\|\cL_{kN}^{-1}\bigl(\f1{r^2}(r^{1-\f3p}u_0\otimes\wt u_k)\bigr)
\bigr\|_{L^p_t(L^{\f{3p}{p+1}})}\\
&\lesssim\bigl\|r^{1-\f{13}{2p}
+\f3{p^2}}u_0\otimes\wt u_k\bigr\|_{L^p_t(L^{\f{6p^2}{2p^2+9p-6}})}.
\end{align*}
By substituting the above estimates into \eqref{3.24}, and using
$$\bigl\|u_0\otimes\wt u_k\bigr\|_{L^p_t(L^{\f{3p}{2p-2}})}
\leq \bigl\|r^{1-\f3p}u_0\otimes\wt u_k\bigr\|_{L^p_t(L^{\f{3p}{p+1}})}^{\f{13p-6-2p^2}{7p-6}}
\bigl\|r^{1-\f{13}{2p}+\f3{p^2}}u_0\otimes\wt u_k\bigr\|_{L^p_t(L^{\f{6p^2}{2p^2+9p-6}})}^{\f{2p(p-3)}{7p-6}},$$
we obtain
\beq\label{3.24a}
\begin{split}
\bigl\|r^{1-\f3p}\cL_{kN}^{-1}\cA_k^{(2)}\bigr\|_{L^p_t(L^{\f{3p}{p+1}})}
\lesssim &\bigl\|r^{1-\f3p}u_0\otimes\wt u_k\bigr\|_{L^p_t(L^{\f{3p}{p+1}})}
+\bigl\|r^{1-\f{13}{2p}+\f3{p^2}}
u_0\otimes\wt u_k\bigr\|_{L^p_t(L^{\f{6p^2}{2p^2+9p-6}})}\\
\lesssim &\|u_0\|_{L^\infty_t(L^3)}\Bigl(
\bigl\|r^{1-\f3p}\wt u_k\bigr\|_{L^p_t(L^{3p})}\\
&\qquad+\bigl\|r^{1-\f3p}\wt u_k\bigr\|_{L^\infty_t(L^p)}^{\f12-\f1p}
\bigl\|r^{1-\f5p}\wt u_k\bigr\|_{L^p_t(L^p)}^{\f12}
\bigl\|r^{-\f32}\wt u_k\bigr\|_{L^2_t(L^2)}^{\f1p}\Bigr).
\end{split} \eeq
Then by using integration by parts similarly as \eqref{3.23}, we get
\begin{align*}
&\Bigl|\int_0^t\int_{\R^3} \wt\nabla\cL_{kN}^{-1}\cA_k^{(2)}
|\wt u_k|^{p-2}\wt u_k r^{p-3}\,dx\,dt'\Bigr|\\
&\lesssim\bigl\|r^{1-\f{3}p}\cL_{kN}^{-1}
\cA_k^{(2)}\bigr\|_{L^p_t(L^{\f{3p}{p+1}})}
\bigl\|\bigl(r^{\f{p-3}2}|\wt\nabla\wt u_k|
|\wt u_k|^{\f p2-1},r^{\f{p-5}2}|\wt u_k|^{\f p2}\bigr)\bigr\|_{L^2_t(L^2)}
\bigl\|r^{1-\f{3}p}\wt u_k\bigr\|_{L^p_t(L^{3p})}^{\f{p}2-1}.
\end{align*}
It follows from \eqref{S2eq3} that
$$\bigl\|r^{1-\f{3}p}\wt u_k\bigr\|_{L^p_t(L^{3p})}=\bigl\|r^{\f{p-3}2}|\wt u_k|^{\f p2}\bigr\|_{L^2_t(L^6)}^{\f2p}
\lesssim \bigl\|r^{\f{p-3}2}|\wt\nabla\wt u_k|
|\wt u_k|^{\f p2-1}\bigr\|_{L^2_t(L^2)}^{\f2p}
+\bigl\|r^{\f{p-5}2}|\wt u_k|^{\f p2}\bigr\|_{L^2_t(L^2)}^{\f2p},$$
from which, and the estimate \eqref{3.24a}, we infer
\beq\label{3.24b}
\begin{split}
\Bigl|\int_0^t\int_{\R^3}& \wt\nabla\cL_{kN}^{-1}\cA_k^{(2)}
|\wt u_k|^{p-2}\wt u_k r^{p-3}\,dx\,dt'\Bigr|\\
\lesssim&\|u_0\|_{L^\infty_t(L^3)}\Bigl(\bigl\|r^{\f{p-3}2}|\wt\nabla\wt u_k|
|\wt u_k|^{\f p2-1}\bigr\|_{L^2_t(L^2)}^{\f2p}
+\bigl\|r^{\f{p-5}2}|\wt u_k|^{\f p2}\bigr\|_{L^2_t(L^2)}^{\f2p}\\
&\qquad\qquad\qquad+\bigl\|r^{\f{p-3}2}|\wt u_k|^{\f p2}\bigr\|_{L^\infty_t(L^2)}^{\f1p-\f2{p^2}}
\bigl\|r^{\f{p-5}2}|\wt u_k|^{\f p2}\bigr\|_{L^2_t(L^2)}^{\f1p}
\bigl\|r^{-\f32}\wt u_k\bigr\|_{L^2_t(L^2)}^{\f1p}\Bigr)\\
&\times\Bigl(\bigl\|r^{\f{p-3}2}|\wt\nabla\wt u_k|
|\wt u_k|^{\f p2-1}\bigr\|_{L^2_t(L^2)}^{2-\f2p}
+\bigl\|r^{\f{p-5}2}|\wt u_k|^{\f p2}\bigr\|_{L^2_t(L^2)}^{2-\f2p}\Bigr).
\end{split}\eeq

By virtue of \eqref{3.23} and \eqref{3.24b},
we can deduce
\begin{align*}
\sum_{j=1}^2&\Psi^{(j)}\eqdef N^{-2\al_p}
\sum_{j=1}^2\Bigl|\int_0^t\int_{\R^3} \wt\nabla\cL_{N}^{-1}\cA_1^{(j)}
|\wt u_1|^{p-2}\wt u_1 r^{p-3}\,dx\,dt'\Bigr|\\
&\qquad\qquad\qquad+\sum_{k=2}^\infty(kN)^{2\beta_p}
\sum_{j=1}^2\Bigl|\int_0^t\int_{\R^3} \wt\nabla\cL_{kN}^{-1}\cA_k^{(j)}
|\wt u_k|^{p-2}\wt u_k r^{p-3}\,dx\,dt'\Bigr|\\
\lesssim&\|u_0\|_{L^\infty_t(L^3)}\Bigl\{
N^{-\f2p\al_p}\bigl\|r^{\f{p-3}2}
|\wt\nabla\wt u_1|
|\wt u_1|^{\f p2-1}\bigr\|_{L^2_t(L^2)}^{\f2p}
+N^{-\f2p(\al_p-1)}\bigl\|r^{\f{p-5}2}|\upsilon_1|^{\f p2}\bigr\|_{L^2_t(L^2)}^{\f2p}\\
&+\bigl(N^{-\al_p}\bigl\|r^{\f{p-3}2}|\wt u_1|^{\f p2}\bigr\|_{L^\infty_t(L^2)}\bigr)^{\f1p-\f2{p^2}}
\bigl(N^{-\al_p+1}\bigl\|r^{\f{p-5}2}|\wt u_1|^{\f p2}\bigr\|_{L^2_t(L^2)}\bigr)^{\f1p}
\bigl\|r^{-\f32}\wt u_1\bigr\|_{L^2_t(L^2)}^{\f1p}\Bigr\}\\
&\times\Bigl(N^{-(2-\f2p)\al_p}\bigl\|r^{\f{p-3}2}|\wt\nabla\wt u_1|
|\wt u_1|^{\f p2-1}\bigr\|_{L^2_t(L^2)}^{2-\f2p}
+N^{-(2-\f2p)(\al_p-1)}\bigl\|r^{\f{p-5}2}|\wt u_1|^{\f p2}\bigr\|_{L^2_t(L^2)}^{2-\f2p}\Bigr)\\
&+\|u_0\|_{L^\infty_t(L^3)}\sum_{k=2}^\infty
\Bigl\{N^{\f2p\beta_p}\bigl\|r^{\f{p-3}2}
|\wt\nabla\wt u_k|
|\wt u_k|^{\f p2-1}\bigr\|_{L^2_t(L^2)}^{\f2p}
+N^{\f2p(\beta_p+1)}\bigl\|r^{\f{p-5}2}|\upsilon_k|^{\f p2}\bigr\|_{L^2_t(L^2)}^{\f2p}\\
&+{N^{\f2{p^2}\beta_p-\f1p}}
\bigl(N^{\beta_p}\bigl\|r^{\f{p-3}2}|\wt u_k|^{\f p2}\bigr\|_{L^\infty_t(L^2)}\bigr)^{\f1p-\f2{p^2}}
\bigl(N^{\beta_p+1}\bigl\|r^{\f{p-5}2}|\wt u_k|^{\f p2}\bigr\|_{L^2_t(L^2)}\bigr)^{\f1p}
\bigl\|r^{-\f32}\wt u_k\bigr\|_{L^2_t(L^2)}^{\f1p}\Bigr\}\\
&\times\Bigl(N^{(2-\f2p)\beta_p}\bigl\|r^{\f{p-3}2}|\wt\nabla\wt u_k|
|\wt u_k|^{\f p2-1}\bigr\|_{L^2_t(L^2)}^{2-\f2p}
+N^{(2-\f2p)(\beta_p+1)}\bigl\|r^{\f{p-5}2}|\wt u_k|^{\f p2}\bigr\|_{L^2_t(L^2)}^{2-\f2p}\Bigr),
\end{align*}
Observing that
$$N^{\f2{p^2}\beta_p-\f1p}<N^{\f2{p^2}\f{p-1}4-\f1p}<1 \andf  \left(\f1p-\f2{p^2}\right)+\f1p+\f1p
+\left(2-\f2p\right)>2,\quad\forall\ p\in]5,6[,$$
we conclude that
\begin{equation}\begin{split}\label{estiPsi12}
\Psi^{(1)}+\Psi^{(2)}&\lesssim \|u_0\|_{L^\infty_t(L^3)}\Bigl(E_p
+E_p^{1-\f{1}{p^2}}\bigl(\sum_{k=1}^\infty
\bigl\|r^{-\f32}\wt u_k\bigr\|_{L^2_t(L^2)}^2\bigr)^{\f1{2p}}\Bigr)\\
&\lesssim \|u_0\|_{L^\infty_t(L^3)}\Bigl(E_p
+\bigl(\sum_{k=1}^\infty
\bigl\|r^{-\f32}\wt u_k\bigr\|_{L^2_t(L^2)}^2\bigr)^{\f{p}2}\Bigr).
\end{split}\end{equation}

\no$\bullet$~{\bf The estimate of $\Psi^{(3)}.$}

We first observe from \eqref{3.21} that
 $\cL_{kN}^{-1}\cA_k^{(3)}$ is composed  of the following terms:
$$\cL_{kN}^{-1}\lozenge\f{\upsilon_k}{r^2}\quad\text{or}\quad
\cL_{kN}^{-1}\lozenge\f{kN\upsilon_k}{r^2},
\quad\text{where}\quad\lozenge\eqdef\bigl(\pa_r,\pa_z,\f1r\bigr)
\andf\upsilon_k\eqdef(\ur_k,\vt_k,\uz_k).$$
So that  for any small enough $\e>0$, we deduce from Proposition
 \ref{lemLm} that
\begin{align*}
\bigl\|r^{\f{p-3}2}\cL_{kN}^{-1}\cA_k^{(3)}\bigr\|_{L_t^{\f{10}{7-p}}(L^{\f{10}{7-p}})}
&\lesssim kN\bigl\|r^{\f{p-3}2}\cL_{kN}^{-1}\lozenge\f1r
\bigl(\f{\upsilon_k}r\bigr)\bigr\|_{L_t^{\f{10}{7-p}}(L^{\f{10}{7-p}})}\\
&\lesssim_\e(kN)^{\f1p+\f{p}{10}+\e-\f15}\bigl\|r^{\f{p-2}5-\f3p}
\upsilon_k\bigr\|_{L_t^{\f{10}{7-p}}(L^p)}\\
&\lesssim_\e(kN)^{\f1p+\f{p}{10}+\e-\f15}
\bigl\|r^{\f{p-3}2}|\upsilon_k|^{\f p2}\bigr\|_{L_t^\infty(L^2)}^{\f2p-\f{7-p}5}
\bigl\|r^{\f{p-5}2}|\upsilon_k|^{\f p2}\bigr\|_{L_t^2(L^2)}^{\f{7-p}5},
\end{align*}
from which, we infer
\begin{equation}\begin{split}\label{ineqA3}
\Bigl|&\int_0^t\int_{\R^3} \wt\nabla\cL_{kN}^{-1}\cA_k^{(3)}
|\wt u_k|^{p-2}\wt u_k r^{p-3}\,dx\,dt'\Bigr|\lesssim\bigl\|r^{\f{p-3}2}\cL_{kN}^{-1}\cA_k^{(3)}
\bigr\|_{L_t^{\f{10}{7-p}}(L^{\f{10}{7-p}})}\\
&\times\Bigl(\bigl\|r^{\f{p-3}2}|\wt u_k|^{\f p2-1}\wt\nabla\wt u_k\bigr\|_{L^2_t(L^2)}+\bigl\|r^{\f{p-5}2}|\wt u_k|^{\f p2}\bigr\|_{L^2_t(L^2)}\Bigr)
\bigl\|r^{\f{p-5}2}|\wt u_k|^{\f p2}\bigr\|_{L^2_t(L^2)}^{\f35}
\bigl\|r^{-\f32}\wt u_k\bigr\|_{L^2_t(L^2)}^{\f p5-1}\\
&\lesssim_\e(kN)^{\f1p+\f{3p}{10}+\e-\f{11}5}\bigl\|r^{\f{p-3}2}|\upsilon_k|^{\f p2}\bigr\|
_{L_t^\infty(L^2)}^{\f2p-\f{7-p}5} \Bigl(\bigl\|r^{\f{p-3}2}|\wt u_k|^{\f p2-1}\wt\nabla\wt u_k\bigr\|_{L^2_t(L^2)}\\
&\qquad+\bigl\|r^{\f{p-5}2}|\wt u_k|^{\f p2}\bigr\|_{L^2_t(L^2)}\Bigr)
\Bigl(kN \bigl\|r^{\f{p-5}2}|\upsilon_k|^{\f p2}\bigr\|_{L^2_t(L^2)}\Bigr)^{2-\f p5}
\bigl\|r^{-\f32}\wt u_k\bigr\|_{L^2_t(L^2)}^{\f p5-1}.
\end{split}\end{equation}
Observing that
$$\left(\f2p-\f{7-p}5\right)+1+\left(2-\f p5\right)
+\left(\f p5-1\right)>2,\quad\forall\ p\in]5,6[,$$
then by virtue of \eqref{3.22} and \eqref{ineqA3}, we get, by a similar derivation of \eqref{estiPsi12}, that
$$\Psi^{(3)}\lesssim_\e\sup_{k\in\N^+}
(kN)^{\f1p+\f{3p}{10}+\e-\f{11}5+2\beta_p(\f15-\f1p)}
\cdot E_p^{\f45+\f1p}\Bigl(\sum_{k=1}^\infty
\bigl\|r^{-\f32}\wt u_k\bigr\|_{L^2_t(L^2)}^2\Bigr)^{\f p{10}-\f12}.$$
While thanks to
\eqref{condialbeta}, one has
$$\f1p+\f{3p}{10}-\f{11}5+2\beta_p\left(\f15-\f1p\right)
<\f1p+\f{3p}{10}-\f{11}5+\f{p-1}{2}\left(\f15-\f1p\right)
<-0.15,$$
so that we may take $\e=0.05$ and obtain
\begin{equation}\begin{split}\label{estiPsi3}
\Psi^{(3)}&\leq C N^{-\f1{10}}\cdot
E_p^{\f45+\f1p}\Bigl(\sum_{k=1}^\infty
\bigl\|r^{-\f32}\wt u_k\bigr\|_{L^2_t(L^2)}^2\Bigr)^{\f p{10}-\f12}\\
&\leq \f18E_p
+CN^{-\f{p}{2(p-5)}}\Bigl(\sum_{k=1}^\infty
\bigl\|r^{-\f32}\wt u_k\bigr\|_{L^2_t(L^2)}^2\Bigr)^{\f p2}.
\end{split}\end{equation}

\no$\bullet$~{\bf The estimate of $\Psi^{(4)}.$}

For simplicity, we shall only present the estimate of the parts in $\Psi^{(4)}$ concerning $\pa_r F^r_k$ and $\f{kN}r F^\th_k.$
The estimate of the remaining parts in $\Psi^{(4)}$ follows along the same line.

By using the fourth equation of \eqref{eqtuk}, we write
\beq\label{S2eq6}
\begin{split}
\pa_r F^r_k=&\cA_k^{(4,1)}+\cA_k^{(4,2)} \andf
\f{kN}r F^\th_k=\cA_k^{(4,3)}+\cA_k^{(4,4)} \with\\
\cA_k^{(4,1)}\eqdefa &-\f12\sum_{(k_1,k_2)\in\Omega_k}
\pa_r\wt\nabla\cdot(\wt u_{k_1}\ur_{k_2})\\
\cA_k^{(4,2)}\eqdefa &-\f12\sum_{(k_1,k_2)\in\Omega_k}
\pa_r\Bigl(\f{\ur_{k_1}\ur_{k_2}}r+k_1N\f{\vt_{k_1}\ur_{k_2}}r\Bigr)\\
&+\f12\Bigl(\sum_{|k_1-k_2|=k}
-\sum_{k_1+k_2=k}\Bigr)\pa_r
\Bigl(\f{\vt_{k_1}\vt_{k_2}}r+k_2N\f{\vt_{k_1}\ur_{k_2}}r\Bigr),\\
\cA_k^{(4,3)}\eqdefa&\f{kN}{2r}\Bigl(\sum_{k_2-k_1=k}
-\sum_{(k_1,k_2)\in\Omega_k}\Bigr)
\Bigl(\wt\nabla\cdot(\wt u_{k_2}\vt_{k_1})
+2\f{\ur_{k_2}\vt_{k_1}}r\Bigr),\\
\cA_k^{(4,4)}=&\f{kN}{r}\Bigl(\sum_{k_2-k_1=k}
-\sum_{(k_1,k_2)\in\Omega_k}\Bigr)
\Bigl(k_2N\f{\vt_{k_1}\vt_{k_2}}r\Bigr),
\end{split} \eeq where the indices $k_1, k_2$ lie in $\N^+$, and $\Omega_k$ is defined by \eqref{defOmk}.
Correspondingly, for $j=1,2,3,4,$ we denote
\beq\label{S2eq6d}
\Psi_k^{(4,j)}\eqdef\Bigl|\int_0^t\int_{\R^3} \wt\nabla\cL_{kN}^{-1}\cA_k^{(4,j)}
|\wt u_k|^{p-2}\wt u_k r^{p-3}\,dx\,dt'\Bigr|.
\eeq

We first get, by using integration by parts and Proposition \ref{lemLmweightedRiesz} that
\begin{align*}
\Psi_k^{(4,1)}&\lesssim\bigl\|r^{1-\f3p}\cL_{kN}^{-1}
\cA_k^{(4,1)}\bigr\|_{L^{\f p2}_t(L^{\f{3p}{p-1}})}
\Phi_k(t)\\
&\lesssim\Bigl(\sum_{(k_1,k_2)\in\Omega_k}
\bigl\|r^{1-\f3p}\wt u_{k_1}\ur_{k_2}\bigr\|_{L^{\f p2}_t(L^{\f{3p}{p-1}})}
+\sum_{(k_1,k_2)\in\Omega_k}
\bigl\|\cL_{kN}^{-1}\bigl(\f{1}r\wt\nabla,\f1{r^2}\bigr)
\bigl(r^{1-\f3p}\wt u_{k_1}\ur_{k_2}\bigr)\bigr\|_{L^{\f p2}_t(L^{\f{3p}{p-1}})}\\
&\quad+\bigl\|\cL_{kN}^{-1}\cA_k^{(4,1)}
\bigr\|_{L^{\f p2}_t(L^{\f{3p}{2p-4}})}\Bigr)
\Phi_k(t)\eqdef \Psi_k^{(4,1,1)}+\Psi_k^{(4,1,2)}+\Psi_k^{(4,1,3)},
\end{align*}
where
\begin{align*}
\Phi_k(t)&=\bigl\|r^{\f{p-3}2(2-\f2p)}
|\wt\nabla\wt u_k||\wt u_k|^{p-2}
\bigr\|_{L^{\f{p}{p-2}}_t(L^{\f{3p}{2p+1}})}
+\bigl\|r^{\f{p-3}2(2-\f2p)-1}|\wt u_k|^{p-1}
\bigr\|_{L^{\f{p}{p-2}}_t(L^{\f{3p}{2p+1}})}\\
&\lesssim\Bigl(\bigl\|r^{\f{p-3}2}|\wt\nabla \wt u_k|
|\wt u_k|^{\f p2-1}\bigr\|_{L^2_t(L^2)}
+\bigl\|r^{\f{p-5}2}
|\wt u_k|^{\f p2}\bigr\|_{L^2_t(L^2)}\Bigr)
\bigl\|r^{\f{p-3}2}|\wt u_{k}|^{\f p2}\bigr\|_{L^\infty_t(L^2)}^{\f2p}
\bigl\|r^{\f{p-3}2}|\wt u_{k}|^{\f p2}\bigr\|_{L^2_t(L^6)}^{1-\f4p}.
\end{align*}
It follows from a similar derivation of  \eqref{3.15} that
\begin{equation}\label{3.29}
N^{-2\al_p}\Psi_1^{(4,1,1)}
+\sum_{k\geq2}(kN)^{2\beta_p} \Psi_k^{(4,1,1)}
\lesssim E_p^{1+\f1p}.
\end{equation}
While we get, by applying  Proposition \ref{lemLmRiesz}
and then Proposition \ref{lemLm}, that
\begin{align*}
&\bigl\|\cL_{kN}^{-1}\cA_k^{(4,1)}\bigr\|_{L^{\f p2}_t(L^{\f{3p}{2p-4}})}
+\sum_{(k_1,k_2)\in\Omega_k}\bigl\|\cL_{kN}^{-1}\bigl(\f{1}r\wt\nabla,\f1{r^2}\bigr)
\bigl(r^{1-\f3p}\wt u_{k_1}\ur_{k_2}\bigr)\bigr\|_{L^{\f p2}_t(L^{\f{3p}{p-1}})}\\
&\lesssim\sum_{(k_1,k_2)\in\Omega_k}
\Bigl(\bigl\|\wt u_{k_1}\ur_{k_2}\bigr\|_{L^{\f p2}_t(L^{\f{3p}{2p-4}})}
+\bigl\|\cL_{kN}^{-1}\f1r\wt\nabla\bigl(\wt u_{k_1}\ur_{k_2}\bigr)\bigr\|
_{L^{\f p2}_t(L^{\f{3p}{2p-4}})}\\
&\qquad\qquad\qquad+\bigl\|\cL_{kN}^{-1}\bigl(\f{1}r\wt\nabla,\f1{r^2}\bigr)
\bigl(r^{1-\f3p}\wt u_{k_1}\ur_{k_2}\bigr)\bigr\|_{L^{\f p2}_t(L^{\f{3p}{p-1}})}\Bigr)\\
&\lesssim\sum_{(k_1,k_2)\in\Omega_k}
\Bigl(\bigl\|\wt u_{k_1}\ur_{k_2}\bigr\|_{L^{\f p2}_t(L^{\f{3p}{2p-4}})}
+(kN)^{-0.1}\bigl\|r^{2-\f{16}p+\f{12}{p^2}}
\wt u_{k_1}\ur_{k_2}\bigr\|_{L^{\f p2}_t(L^{\f{p^2}{4(p-1)}})}\Bigr)\\
&\lesssim\sum_{(k_1,k_2)\in\Omega_k}
\Bigl(\bigl\|r^{1-\f3p}\wt u_{k_1}\ur_{k_2}\bigr\|_{L^{\f p2}_t(L^{\f{3p}{p-1}})}
+\bigl\|r^{2-\f{16}p+\f{12}{p^2}}
\wt u_{k_1}\ur_{k_2}\bigr\|_{L^{\f p2}_t(L^{\f{p^2}{4(p-1)}})}\Bigr).
\end{align*}
Hence we obtain
\begin{equation}\label{3.30}\begin{split}
\Psi_k^{(4,1,2)}+&\Psi_k^{(4,1,3)}\lesssim \Psi_k^{(4,1,1)}+\Psi_k^{(4,1,4)}\with\\
\Psi_k^{(4,1,4)}\eqdefa &\sum_{(k_1,k_2)\in\Omega_k}
\bigl\|r^{2-\f{16}p+\f{12}{p^2}}
\wt u_{k_1}\ur_{k_2}\bigr\|_{L^{\f p2}_t(L^{\f{p^2}{4(p-1)}})}\Phi_k(t).
\end{split}\end{equation}
By applying H\"older's inequality, one has
\begin{align*}
&\sum_{(k_1,k_2)\in\Omega_k}\bigl\|r^{2-\f{16}p+\f{12}{p^2}}
\wt u_{k_1}\ur_{k_2}\bigr\|_{L^{\f p2}_t(L^{\f{p^2}{4(p-1)}})}
\leq2\sum_{\substack{(k_1,k_2)\in\Omega_k\\k_1\leq k_2}}
\bigl\|r^{\f{p-3}p}\wt u_{k_1}\bigr\|_{L^\infty_t(L^p)}^{2-\f4p}
\bigl\|r^{-\f32}\wt u_{k_2}\bigr\|_{L^2_t(L^2)}^{\f4p}\\
&\ \lesssim \Bigl(N^{-2\al_p}\bigl\|r^{\f{p-3}2}|\wt u_1|^{\f p2}\bigr\|_{L^\infty_t(L^2)}^2
+\sum_{k_1\geq2}(k_1N)^{2\beta_p}
\bigl\|r^{\f{p-3}2}|\wt u_{k_1}|^{\f p2}\bigr\|_{L^\infty_t(L^2)}^2\Bigr)^{\f2p-\f4{p^2}}\\
&\qquad\times\Bigl(\sum_{k_2\in\N^+}(k_2N)^2
\bigl\|r^{-\f32}\wt u_{k_2}\bigr\|_{L^2_t(L^2)}^2\Bigr)^{\f2p}
\Bigl\{N^{\al_p(\f4p-\f8{p^2})}(kN)^{-\f4p}\\
&\qquad\qquad\qquad\qquad\qquad\qquad+\Bigl(\sum_{(k_1,k_2)\in\Omega_k,\,2\leq k_1\leq k_2}
(k_1N)^{-\f{4\beta_p}{p-2}}
(k_2N)^{-\f{4p}{(p-2)^2}}\Bigr)^{\f{(p-2)^2}{p^2}}\Bigr\}.
\end{align*}
Noticing that for $(k_1,k_2)\in\Omega_k$ with $k_1\leq k_2$,
there holds $k_2\geq\f k2$, thus
\begin{align*}
\Bigl(\sum_{\substack{(k_1,k_2)\in\Omega_k,\\
2\leq k_1\leq k_2}}
&(k_1N)^{-\f{4\beta_p}{p-2}}
(k_2N)^{-\f{4p}{(p-2)^2}}\Bigr)^{\f{(p-2)^2}{p^2}}\lesssim(kN)^{-\f4p}
\Bigl(\sum_{k_1\in\N^+}
(k_1N)^{-\f{4\beta_p}{p-2}}\Bigr)^{\f{(p-2)^2}{p^2}}
\lesssim(kN)^{-\f4p},
\end{align*}
where in the last step, we used the fact that $\f{4\beta_p}{p-2}>1$. As a result, it comes out
\begin{equation}\label{3.32}
\sum_{(k_1,k_2)\in\Omega_k}\bigl\|r^{2-\f{16}p+\f{12}{p^2}}
\wt u_{k_1}\ur_{k_2}\bigr\|_{L^{\f p2}_t(L^{\f{p^2}{4(p-1)}})}
\lesssim N^{\al_p(\f4p-\f8{p^2})}(kN)^{-\f4p}
E_p^{\f2p-\f4{p^2}}D^{\f2p},
\end{equation} for $D(t)$ being given by \eqref{aprioriEp}.

By substituting \eqref{3.32} into \eqref{3.30}, we achieve
\begin{align*}
\Psi_k^{(4,1,4)}&\lesssim N^{\al_p(\f4p-\f8{p^2})}(kN)^{-\f4p}
E_p^{\f2p-\f4{p^2}}D^{\f2p}\\
\times& \Bigl(\bigl\|r^{\f{p-3}2}|\wt\nabla \wt u_k|
|\wt u_k|^{\f p2-1}\bigr\|_{L^2_t(L^2)}
+\bigl\|r^{\f{p-5}2}
|\wt u_k|^{\f p2}\bigr\|_{L^2_t(L^2)}\Bigr)
\bigl\|r^{\f{p-3}2}|\wt u_{k}|^{\f p2}\bigr\|_{L^\infty_t(L^2)}^{\f2p}
\bigl\|r^{\f{p-3}2}|\wt u_{k}|^{\f p2}\bigr\|_{L^2_t(L^6)}^{1-\f4p},
\end{align*}
from which, we deduce that
\begin{equation}\begin{split}\label{3.33}
N^{-2\al_p}&\Psi_1^{(4,1,4)}
+\sum_{k\geq2}(kN)^{2\beta_p} \Psi_k^{(4,1,4)}\\
\lesssim& N^{\al_p(\f4p-\f8{p^2})}E_p^{1+\f1p-\f4{p^2}}D^{\f2p}
\Bigl\{N^{-\f2p\al_p-\f4p}
+\Bigl(\sum_{k\geq2}(kN)^{\left(\f2p\beta_p-\f4p\right)p}\Bigr)^{\f1p}\Bigr\}\\
\lesssim& N^{-\f4p+(\f4p-\f8{p^2})\al_p+\f2p\beta_p}E_p^{1+\f1p-\f4{p^2}}D^{\f2p},
\end{split}\end{equation}
where we used the fact that $\beta_p<\f{p-1}4<\f54$ so that
$\left(\f2p\beta_p-\f4p\right)p<-1$ in the last step.

By summarizing the estimates \eqref{3.29},~\eqref{3.30}
and \eqref{3.33}, and then using Young's inequality
and the condition \eqref{condialbeta}, we arrive at
\begin{equation}\begin{split}\label{3.34}
N^{-2\al_p}\Psi_1^{(4,1)}
+\sum_{k\geq2}(kN)^{2\beta_p} \Psi_k^{(4,1)}
&\lesssim E_p^{1+\f1p}
+N^{-\f4p+(\f4p-\f8{p^2})\al_p+\f2p\beta_p}E_p^{1+\f1p-\f4{p^2}}D^{\f2p}\\
&\lesssim E_p^{1+\f1p}
+N^{\bigl(-\f4p+(\f4p-\f8{p^2})\al_p+\f2p\beta_p\bigr)\f{p(p+1)}4}D^{\f{p+1}2}\\
&\lesssim E_p^{1+\f1p}
+N^{-2}D^{\f{p+1}2}.
\end{split}\end{equation}

The estimate of $N^{-2\al_p}\Psi_1^{(4,j)}
+\sum_{k\geq2}(kN)^{2\beta_p} \Psi_k^{(4,j)}$ for $j=2,3,4$ involves cumbersome calculations. In the Appendix \ref{sectB}, we shall present the
proof of the following estimates:
\begin{equation}\begin{split}\label{3.39}
N^{-2\al_p}\Psi_1^{(4,j)}
+\sum_{k\geq2}(kN)^{2\beta_p} \Psi_k^{(4,j)}
\lesssim E_p^{1+\f1p},\quad\text{for}\quad
j=2,3,4.
\end{split}\end{equation}

By summarizing the estimates \eqref{3.34} and \eqref{3.39},
and recalling the fact the  parts
in $\Psi^{(4)}$ concerning $\f1r F^r_k$ (resp. $\pa_z F^z_k$) can be handled exactly along the same line to  the part relating to
$\f{kN}r F^\th_k$ (resp. $\pa_r F^r_k$), we conclude that
\begin{equation}\begin{split}\label{3.45}
\Psi^{(4)}&\lesssim\sum_{j=1}^4
N^{-2\al_p}\Psi_1^{(4,j)}
+\sum_{k\geq2}(kN)^{2\beta_p} \Psi_k^{(4,j)}\\
&\lesssim E_p^{1+\f1p}+N^{-2}D^{\f{p+1}2}.
\end{split}\end{equation}

In view of \eqref{3.22},
 by summing up  the estimates
\eqref{estiPsi12},~\eqref{estiPsi3} and \eqref{3.45},
we obtain
 \eqref{3.46}.
This completes the proof of Lemma \ref{S2lem2}.
\end{proof}

\section{The proof of Proposition \ref{propH1}}\label{secH1}

The goal of this section is to present the proof of Proposition \ref{propH1}, which will be based on Proposition \ref{propEp},
 Plancherel's identity and Hausdorff-Young's inequality.

\begin{proof}[Proof of Proposition \ref{propH1}]
We first remark that the assumption $(a^r,b^\th,z^z)\in\cM,$ which is defined by \eqref{defcM}, ensures
that the initial data $u_{\rm in}$ given by \eqref{initialNodd} belongs to  $H^1(\R^3)$.
Then it follows from  the classical well-posedness theory of Navier-Stokes equations
(see \cite{CC88} for instance)   that the system \eqref{NS} with initial data $u_{\rm in}$  given by \eqref{initialNodd}
has a unique solution $u$ on $[0,T^\ast[,$ where $T^\ast$ is the lifespan of the solution. Furthermore,
the solution can be written in term of  \eqref{solexpan}. Let $u_0$ be given by \eqref{S2notion},
for some $\e>0$
to be determined later on, we denote
\begin{equation}\label{4.1}
T^\star\eqdef\sup\bigl\{T<T^\ast:\ \|\nabla u\|_{L^\infty_T(L^2)}^2
+\|\na^2 u\|_{L_T^2(L^2)}^2\leq 2\|\nabla u_{\rm in}\|_{L^2}^2,\  \|u_0\|_{L^\infty_T(L^3)}\leq 2\e\bigr\}.
\end{equation}
Since $u_{0,{\rm in}}=0$, $T^\star$ is well-defined and must be positive.
Then the proof of Proposition \ref{propH1} reduces to
 show that
$T^\star=\infty$ for $N$ being sufficiently large. Without loss of generality,
we may assume that $T^\star<\infty$ in what follows.

In view of  \eqref{solexpan}, one has
\begin{align*}
\wt\nabla^j u=&\Bigl\{\wt\nabla^j\ur_{0}
+\sum_{k=1}^\infty\wt\nabla^j\ur_{k}\cos kN\th\Bigr\}\vv e_r\\
&+\Bigl\{\sum_{k=1}^\infty\wt\nabla^j\vt_{k}
\sin kN\th\Bigr\}\vv e_\th
+\Bigl\{\wt\nabla^j\uz_{0}
+\sum_{k=1}^\infty\wt\nabla^j\uz_{k}\cos kN\th\Bigr\}\vv e_z,
~\text{ for $j=1$ or $2$},
\end{align*}
from which and Plancherel's identity, we infer
\beq\label{S4eq1}
\|\wt\nabla^j u_0\|_{L^2}^2
+\sum_{k=1}^\infty\Bigl(\|\wt\nabla^j\ur_k\|_{L^2}^2
+\|\wt\nabla^j\vt_k\|_{L^2}^j+\|\wt\nabla^j\uz_k\|_{L^2}^2\Bigr)
\leq \|\wt\nabla^j u\|_{L^2}^2.
\eeq

While in view of \eqref{coordinate}, one has
\begin{align*}
\pa_r&=\cos\th\,\pa_1+\sin\th\,\pa_2,\quad\pa_z=\pa_3,\quad
\pa_r\pa_z=\cos\th\,\pa_1\pa_3+\sin\th\,\pa_2\pa_3,
\quad\pa_z^2=\pa_3^2,\\
&\text{and}\quad
\pa_r^2=(\cos\th\,\pa_1+\sin\th\,\pa_2)^2
=\cos^2\th\,\pa_1^2+2\cos\th\sin\th\,\pa_1\pa_2+\sin^2\th\,\pa_2^2.
\end{align*}
Hence we deduce \eqref{S4eq1} that
\begin{align*}
\|\wt\nabla^j u_0\|_{L^2}^2
+\sum_{k=1}^\infty\Bigl(\|\wt\nabla^j\ur_k\|_{L^2}^2
+\|\wt\nabla^j\vt_k\|_{L^2}^j+\|\wt\nabla^j\uz_k\|_{L^2}^2\Bigr)
\leq\|\nabla^j u\|_{L^2}^2.
\end{align*}
Then  we deduce from \eqref{initialNodd} and \eqref{4.1}  that for any $t\leq T^\star$, there holds
\begin{equation}\begin{split}\label{4.4}
\|\wt\nabla u_0\|_{L^\infty_t(L^2)}^2
&+\|\wt\nabla^2 u_0\|_{L^2_t(L^2)}^2\\
&+\sum_{k=1}^\infty\Bigl(\|\wt\nabla\upsilon_k\|_{L^\infty_t(L^2)}^2
+\|\wt\nabla^2\upsilon_k\|_{L^2_t(L^2)}^2\Bigr)\leq C\|\nabla u_{\rm in}\|_{L^2}^2
\leq C_{\rm in} N^2.
\end{split}\end{equation}
Here and in all that follows, we always designate
 $(\ur_k,\vt_k,\uz_k)$ to be $\upsilon_k$,
and $C_{\rm in}$ to be some
positive constant which  depends only on $\|(a^r,b^\th,a^z)\|_{\cM}$.

On the other hand, we observe from \eqref{coordinate} that
$$\pa_1^2=\cos^2\th\,\pa_r^2-\f{\sin2\th}r\pa_r\pa_\th
+\f{\sin2\th}{r^2}\pa_\th+\f{\sin^2\th}r\pa_r
+\f{\sin^2\th}{r^2}\pa_\th^2,$$
so that it follows  from
\eqref{solexpan} that
\begin{align*}
\pa_1^2(u^3\vv e_3)=\Bigl\{&\f12\Bigl(\pa_r^2\uz_0+\f{\pa_r\uz_0}{r}\Bigr)
+\f12\Bigl(\pa_r^2\uz_0-\f{\pa_r\uz_0}{r}\Bigr)\cos2\th\\
&+\f14\sum_{k=1}^\infty\Bigl(\pa_r^2\uz_k-(1{-2}kN)\f{\pa_r\uz_k}{r}
+(k^2N^2-2kN)\f{\uz_k}{r^2}\Bigr)\cos (kN-2)\th\\
&+\f12\sum_{k=1}^\infty\Bigl(\pa_r^2\uz_k{+}\f{\pa_r\uz_k}{r}
-k^2N^2\f{\uz_k}{r^2}\Bigr)\cos kN\th\\
&+\f14\sum_{k=1}^\infty\Bigl(\pa_r^2\uz_k-(1+2kN)\f{\pa_r\uz_k}{r}
+(k^2N^2+2kN)\f{\uz_k}{r^2}\Bigr)\cos (kN+2)\th\Bigr\}\vv e_z.
\end{align*}
Noticing that for $N>4,$ one has
$$2<N-2<N<N+2<\cdots<kN+2<(k+1)N-2<\cdots.$$
Then we get, by using Plancherel's identity, that
\beq\label{S4eq1a}
\begin{split}
\sum_{k=1}^\infty\Bigl(\Bigl\|\pa_r^2\uz_k&-(1{-2}kN)\f{\pa_r\uz_k}{r}
+(k^2N^2-2kN)\f{\uz_k}{r^2}\Bigr\|_{L^2}^2
+\Bigl\|\pa_r^2\uz_k{+}\f{\pa_r\uz_k}{r}
-k^2N^2\f{\uz_k}{r^2}\Bigr\|_{L^2}^2\\
&\qquad\quad+\Bigl\|\pa_r^2\uz_k-(1+2kN)\f{\pa_r\uz_k}{r}
+(k^2N^2+2kN)\f{\uz_k}{r^2}\Bigr\|_{L^2}^2\Bigr)\lesssim\|\nabla^2 u\|_{L^2}^2.
\end{split}\eeq
Observing that
\begin{align*}
k^2N^2&\f{\uz_k}{r^2}-\f{\pa_r\uz_k}{r}
=\f14\Bigl(\pa_r^2\uz_k-(1-2kN)\f{\pa_r\uz_k}{r}
+(k^2N^2-2kN)\f{\uz_k}{r^2}\Bigr)\\
&+\f14\Bigl(\pa_r^2\uz_k-(1+2kN)\f{\pa_r\uz_k}{r}
+(k^2N^2+2kN)\f{\uz_k}{r^2}\Bigr)
-\f12\Bigl(\pa_r^2\uz_k+\f{\pa_r\uz_k}{r}
-k^2N^2\f{\uz_k}{r^2}\Bigr),
\end{align*}
and
\begin{align*}
kN\f{\pa_r\uz_k}{r}-kN\f{\uz_k}{r^2}
=&\f14\Bigl(\pa_r^2\uz_k-(1-2kN)\f{\pa_r\uz_k}{r}
+(k^2N^2-2kN)\f{\uz_k}{r^2}\Bigr)\\
&-\f14\Bigl(\pa_r^2\uz_k-(1+2kN)\f{\pa_r\uz_k}{r}
+(k^2N^2+2kN)\f{\uz_k}{r^2}\Bigr),
\end{align*}
then for $kN$ sufficiently large, we deduce from \eqref{S4eq1a} that
\beq\label{S4eq2}
\sum_{k=1}^\infty\Bigl(k^4N^4\Bigl\|\f{\uz_k}{r^2}\Bigr\|_{L^2}^2
+k^2N^2\Bigl\|\f{\pa_r\uz_k}{r}\Bigr\|_{L^2}^2\Bigr)\lesssim\|\nabla^2 u\|_{L^2}^2.
\eeq

Exactly along the same line, we get, by applying Plancherel's identity to
$\pa_1\pa_3 (u^3\vv e_3)$ and
$$(\pa_1^2,\pa_1\pa_3)(u^1\vv e_1+u^2\vv e_2)
=(\pa_1^2,\pa_1\pa_3)\Bigl(\ur_{0}\,\vv e_r
+\sum_{k=1}^\infty \ur_k\cos kN\th\,\vv e_r
+\sum_{k=1}^\infty\vt_k\sin kN\th\,\vv e_\th\Bigr)$$
that for $t\leq T^\star,$
\begin{align*}
\sum_{k=1}^\infty\Bigl(k^4N^4\Bigl\|\f{(\ur_k,\vt_k)}{r^2}\Bigr\|_{L^2}^2
+k^2N^2\Bigl\|\f{(\wt\nabla\ur_k,\wt\nabla\vt_k)}r\Bigr\|_{L^2}^2
+k^2N^2\Bigl\|\f{\pa_z\uz_k}r\Bigr\|_{L^2}^2\Bigr)
\lesssim\|\nabla^2 u\|_{L^2}^2,
\end{align*}
from which, \eqref{initialNodd}, \eqref{4.1} and \eqref{S4eq2},
 we deduce that for any $t\leq T^\star$,
\begin{equation}\label{4.5}
\sum_{k=1}^\infty\Bigl(k^4N^4\Bigl\|\f{\upsilon_{k}}{r^2}\Bigr\|_{L^2_t(L^2)}^2
+k^2N^2\Bigl\|\f{\wt\nabla\upsilon_k}r\Bigr\|_{L^2_t(L^2)}^2\Bigr)
\lesssim\|\na^2u\|_{L^2_t(L^2)}^2
\leq C_{\rm in} N^2.
\end{equation}

On the other hand, noticing that
\begin{align*}
\bigl(-\sin\th\,\pa_1+\cos\th\,\pa_2\bigr) u
&=\f{\pa_\th}r u=-\Bigl\{\sum_{k=1}^\infty
\Bigl(\f{\vt_{k}}r+kN\f{\ur_{k}}r\Bigr)\sin kN\th\Bigr\}\vv e_r\\
&+\Bigl\{\f{\ur_{0}}r+\sum_{k=1}^\infty \Bigl(\f{\ur_{k}}r+kN\f{\vt_{k}}r\Bigr)
\cos kN\th\Bigr\}\vv e_\th
-\Bigl\{\sum_{k=1}^\infty kN\f{\uz_{k}}r\sin kN\th\Bigr\}\vv e_z,
\end{align*}
from which and Plancherel's identity, we infer
\begin{align*}
\sum_{k=1}^\infty\Bigl(\Bigl\|\f{\vt_{k}}r+kN\f{\ur_{k}}r\Bigr\|_{L^2}^2
+\Bigl\|\f{\ur_{k}}r+kN\f{\vt_{k}}r\Bigr\|_{L^2}^2
+k^2N^2\Bigl\|\f{\uz_{k}}r\Bigr\|_{L^2}^2\Bigr)
\lesssim\bigl\|\f{\pa_\th u}r \bigr\|_{L^2}
\lesssim\|\nabla u\|_{L^2}^2.
\end{align*}
For $N$ large enough,
we deduce that for any $t\leq T^\star,$
\begin{equation}\begin{split}\label{4.6}
\sum_{k=1}^\infty k^2N^2\Bigl(\Bigl\|\f{\ur_{k}}r\Bigr\|_{L^\infty_t(L^2)}^2
+\Bigl\|\f{\vt_{k}}r\Bigr\|_{L^\infty_t(L^2)}^2
+\Bigl\|\f{\uz_{k}}r\Bigr\|_{L^\infty_t(L^2)}^2\Bigr)
\lesssim\|\nabla u\|_{L^\infty_t(L^2)}^2\leq C_{\rm in}N^2.
\end{split}\end{equation}

Similarly, we deduce from the $L^2$ energy inequality of \eqref{NS}:
$$\|u\|_{L^\infty_t(L^2)}^2+2\|\nabla u\|_{L^2(L^2)}^2\leq\|u_{\rm in}\|_{L^2}^2,$$
and Plancherel's identity  that
\begin{equation}\label{4.7}
\sum_{k=1}^\infty\Bigl(\|\upsilon_k\|_{L^\infty_t(L^2)}^2
+\|\wt\nabla\upsilon_k\|_{L^2_t(L^2)}^2
+k^2N^2\Bigl\|\f{\upsilon_{k}}{r}\Bigr\|_{L^2_t(L^2)}^2\Bigr)
\leq C_{\rm in}.
\end{equation}

By interpolating between \eqref{4.5} and \eqref{4.7}, we achieve for $t\leq T^\star$
\begin{equation}\label{4.8}
\begin{split}
&\sum_{k=1}^\infty k^3N^{2}\bigl\|r^{-\f32}\upsilon_{k}\bigr\|_{L^2_t(L^2)}^2
\leq \Bigl(\sum_{k=1}^\infty k^2N^{2} \Bigl\|\f{\upsilon_{k}}{r}\Bigr\|_{L^2_t(L^2)}^2\Bigr)^{\f12}
\Bigl(\sum_{k=1}^\infty k^4N^{2} \Bigl\|\f{\upsilon_{k}}{r^2}\Bigr\|_{L^2_t(L^2)}^2\Bigr)^{\f12}
\leq C_{\rm in},\\
&\sum_{k=1}^\infty k\bigl\|r^{-\f12}\wt\nabla\upsilon_{k}\bigr\|_{L^2_t(L^2)}^2
\leq \Bigl(\sum_{k=1}^\infty \|\wt\nabla{\upsilon_{k}}\|_{L^2_t(L^2)}^2\Bigr)^{\f12}
\Bigl(\sum_{k=1}^\infty k^2 \Bigl\|\f{\wt\nabla\upsilon_{k}}{r}\Bigr\|_{L^2_t(L^2)}^2\Bigr)^{\f12}
\leq C_{\rm in}.
\end{split}
\end{equation}

In view of \eqref{4.1},
by taking $\e$ to be so small
 that
$\f12-C\|u_0\|_{L^\infty_{T^\star}(L^3)}\geq\f13,$ we deduce from  \eqref{aprioriEp} and \eqref{4.8} that for $t\leq T^\star,$ there holds
\begin{equation}\label{4.2}
E_{p,\al_p,\beta_p}(t)
\leq 3E_{p,\al_p,\beta_p}(0)
+CE_{p,\al_p,\beta_p}^{1+\f1p}(t)+C_{\rm in}N^{-2}.
\end{equation}
Whereas by virtue of \eqref{initialNodd}, \eqref{defEpr} and \eqref{defEpth},
we find that
\begin{align*}
E_{p,\al_p,\beta_p}(0)=&N^{-2\al_p}\bigl\|r^{1-\f3p}a^r\bigr\|_{L^p}^p
+N^{-\f p2-2\al_p}\bigl\|r^{1-\f3p}b^\th\bigr\|_{L^p}^p
+N^{-2\al_p}\bigl\|r^{1-\f3p}a^z\bigr\|_{L^p}^p\\
\leq & C_{\rm in}N^{-2\al_p}\quad\mbox{for any}\ p\in ]5,6[.
\end{align*}
Then we deduce from \eqref{4.2} and a standard continuity argument that
\begin{equation}\label{4.3}
E_{p,\al_p,\beta_p}(t)
\leq 4E_{p,\al_p,\beta_p }(0)+C_{\rm in}N^{-2}
\leq C_{\rm in}N^{-2\al_p},\quad\forall\ t\leq T^\star.
\end{equation}

In particular, by taking $p=\f{100}{19},~\al_p=\f{1}{30}$ and
$\beta_p=\f{17}{16}$, which satisfy the condition \eqref{condialbeta}, we get, by interpolating between \eqref{4.8}
and \eqref{4.3} that
\begin{align*}
\|\wt u_1\|_{L^5_t(L^5)}^5
\leq&N^{-2+\f{57}{31}\al_p}\Bigl(N^2\bigl\|r^{-\f32}\wt u_1\bigr\|
_{L^2_t(L^2)}^2\Bigr)^{\f{5}{62}}\Bigl(N^{2-2\al_p}
\bigl\|r^{\f{5}{38}}|\wt u_1|^{\f{50}{19}}\bigr\|
_{L^2_t(L^2)}^2\Bigr)^{\f{57}{62}}\\
\leq&C_{\rm in}N^{-2+\f{57}{31}\al_p} E_{\f{100}{19},\f{1}{30},\f{17}{16}}^{\f{57}{62}}(0)
\leq C_{\rm in}N^{-2},
\end{align*}
and
\begin{align*}
\bigl\||\wt\nabla\wt u_1||\wt u_1|^{\f12}\bigr\|_{L^2_t(L^2)}^2
\leq&N^{\f{19}{31}\al_p}\Bigl(\bigl\|r^{-\f12}\wt\nabla\wt u_1\bigr\|
_{L^2_t(L^2)}^2\Bigr)^{\f{43}{62}}\Bigl(N^{-2\al_p}
\bigl\|r^{\f{43}{38}}|\wt\nabla\wt u_1||\wt{u}_1|^{\f{31}{19}}\bigr\|
_{L^2_t(L^2)}^2\Bigr)^{\f{19}{62}}\\
\leq&C_{\rm in}N^{\f{19}{31}\al_p} E_{\f{100}{19},\f{1}{30},\f{17}{16}}^{\f{19}{62}}(0)
\leq C_{\rm in},
\end{align*}
and
\begin{align*}
\sum_{k=2}^\infty k^4\|\wt u_{k}\|_{L^5_t(L^5)}^5
\leq&N^{-3.95}\sum_{k=2}^\infty\Bigl(k^3N^2\bigl\|r^{-\f32}\wt u_{k}\bigr\|
_{L^2_t(L^2)}^2\Bigr)^{\f{5}{62}}
\Bigl((kN)^{2+\f{17}8}\bigl\|r^{\f{5}{38}}|\wt u_{k}|^{\f{50}{19}}\bigr\|
_{L^2_t(L^2)}^2\Bigr)
^{\f{57}{62}}\\
\leq&C_{\rm in}N^{-3.95} E_{\f{100}{19},\f{1}{30},\f{17}{16}}^{\f{57}{62}}(0)
\leq C_{\rm in}N^{-4},\end{align*}
and
\begin{align*}
\sum_{k=2}^\infty k^{\f43}\bigl\||\wt\nabla\wt u_{k}||\wt u_{k}|^{\f12}\bigr\|_{L^2_t(L^2)}^2
\leq& N^{-0.65}\sum_{k=2}^\infty\Bigl(k\bigl\|r^{-\f12}\wt\nabla\wt u_{k}\bigr\|
_{L^2_t(L^2)}^2\Bigr)^{\f{43}{62}}\\
&\qquad\qquad\times\Bigl((kN)^{\f{17}8}\bigl\|r^{\f{43}{38}}
|\wt\nabla\wt u_k||\wt u_k|^{\f{31}{19}}\bigr\|_{L^2_t(L^2)}^2\Bigr)^{\f{19}{62}}\\
\leq&C_{\rm in}N^{-0.65} E_{\f{100}{19},\f{1}{30},\f{17}{16}}^{\f{19}{62}}(0)
\leq C_{\rm in}N^{-\f23}.
\end{align*}
Similar estimates hold for $\vt_{k}$.

As a consequence, we obtain
\begin{equation}\begin{split}\label{4.9}
&\|\wt u_1\|_{L^5_t(L^5)}^5
\leq C_{\rm in}N^{-2},\quad\|\vt_1\|_{L^5_t(L^5)}^5
\leq C_{\rm in}N^{-2-\f{75}{31}},\\
\sum_{k=2}^\infty (k&N)^4\|\wt u_{k}\|_{L^5_t(L^5)}^5
\leq C_{\rm in},\quad\sum_{k=2}^\infty (kN)^{4+\f{75}{31}}
\|\vt_{k}\|_{L^5_t(L^5)}^5\leq C_{\rm in},
\end{split}\end{equation}
and
\begin{equation}\label{4.10}
\bigl\||\wt\nabla\wt u_1||\wt u_1|^{\f12}\bigr\|_{L^2_t(L^2)}^2
\leq C_{\rm in},\quad\sum_{k=2}^\infty k^{\f43}
\bigl\||\wt\nabla\wt u_{k}||\wt u_{k}|^{\f12}\bigr\|_{L^2_t(L^2)}^2
\leq C_{\rm in}N^{-\f23}.
\end{equation}

The rest of the proof of  Proposition \ref{propH1} relies on the following two lemmas, which we admit for the time being.

\begin{lem}\label{propuL2}
{\sl  There exists large enough integer $N_0$ so that for $N\geq N_0$ and for any $t\leq T^\star$,  there holds
\beq \label{S4eq4} \|u_0\|_{L^\infty_t(L^3)}^3
+\bigl\||\nabla u_0||u_0|^{\f12}\bigr\|_{L^2_t(L^2)}^2
+\bigl\|\nabla|u_0|^{\f32}\bigr\|_{L^2_t(L^2)}^2
\leq 4C_{\rm in}N^{-\f35}.\eeq
}\end{lem}

\begin{lem}\label{propuH1}
{\sl There exists large enough integer $N_0$ so that for $N\geq N_0$ and for any $t\leq T^\star$, there holds
\beq \label{S4eq4w}  \|\nabla u\|_{L^\infty_t(L^2)}^2
+\|\D u\|_{L_t^2(L^2)}^2\leq \f32\|\nabla u_{\rm in}\|_{L^2}^2 \andf \|u\|_{L^5_t(L^5)}\leq C_{\rm in}N^{-\f15}.\eeq
}\end{lem}

Now we are in a position to complete
the proof of Proposition \ref{propH1}.
For  $N$ so large  that $4C_{\rm in}N^{-\f35}\leq \e^3$, we deduce from
Lemma  \ref{propuL2} that
$\|u_0\|_{L^\infty_{T^\star}(L^3)}
\leq\e,$
which together with \eqref{S4eq4w}
contradicts with the definition of $T^\star$ given by \eqref{4.1}  unless $T^\star=T^\ast=\infty$.
Moreover, \eqref{S4eq4w} ensures \eqref{estipropH1}.
This
completes the proof of Proposition \ref{propH1}.
\end{proof}

Let us now present the proof of Lemmas \ref{propuL2} and  \ref{propuH1}.

\begin{proof}[{\bf Proof of Lemma \ref{propuL2}}]
As we mentioned in Section \ref{secidea},
$P_0$ seems to be more difficult to be handled than $P_k$ for $k\in\N^+$,
hence we appeal to the system \eqref{eqtuoEuclidean}  in the Euclidean coordinates.

By taking $L^2$ inner product of \eqref{eqtuoEuclidean}
with $|u_0|u_0$, and using the divergence-free condition of $u_0$
and the fact that the initial data of  $u_{0}$ vanishes, we obtain
\beq\label{S4eq3}
\begin{split}
\f13\|u_0\|_{L^\infty_t(L^3)}^3
+&\bigl\||\nabla u_0||u_0|^{\f12}\bigr\|_{L^2_t(L^2)}^2
+\f49\bigl\|\nabla|u_0|^{\f32}\bigr\|_{L^2_t(L^2)}^2
\leq \int_0^t\int_{\R^3} |\nabla P_0|\cdot |u_0|^2\,dxdt'\\
&+\sum_{k=1}^\infty\int_0^t\int_{\R^3}\Bigl(|\wt u_k\cdot\wt\nabla\wt u_{k}|
+\f{|\vt_{k}|^2}r+kN\f{|\vt_k\wt u_k|}r\Bigr) |u_0|^2\,dxdt'
\eqdef \cB_1+\cB_2.
\end{split}\eeq
It follows from  \eqref{4.8}-\eqref{4.10} that for any $t\leq T^\star$
\begin{equation}\begin{split}\label{4.11}
\cB_2&\leq\sum_{k=1}^\infty
\Bigl(\bigl\||\nabla\wt u_k||\wt u_k|^{\f12}\bigr\|_{L^2_t(L^2)}
\|\wt u_k\|_{L^5_t(L^5)}^{\f12}
+\bigl\|r^{-\f32}\vt_k\bigr\|_{L^2_t(L^2)}^{\f23}
\|\vt_k\|_{L^5_t(L^5)}^{\f43}\\
&\qquad+kN\bigl\|r^{-\f32}\vt_k\bigr\|_{L^2_t(L^2)}^{\f23}
\|\wt u_k\|_{L^5_t(L^5)}^{\f13}\|\vt_k\|_{L^5_t(L^5)}\Bigr)
\|u_0\|_{L^\infty_t(L^3)}^{\f45}\|u_0\|_{L^3_t(L^9)}^{\f65}\\
&\leq C_{\rm in}N^{-\f15}\|u_0\|_{L^\infty_t(L^3)}^{\f45}
\bigl\|\nabla|u_0|^{\f32}\bigr\|_{L^2_t(L^2)}^{\f45}\\
&\leq \f14\Bigl(\f13\|u_0\|_{L^\infty_t(L^3)}^3
+\f49\bigl\|\nabla|u_0|^{\f32}\bigr\|_{L^2_t(L^2)}^2\Bigr)
+C_{\rm in}N^{-\f35}.
\end{split}\end{equation}

While by taking space divergence to the system \eqref{eqtuoEuclidean}, we decompose
$\nabla P_0$ as $\nabla P_0^{(1)}+\nabla P_0^{(2)}$, where
$$\nabla P_0^{(1)}\eqdefa\nabla(-\D)^{-1}\dive(u_0\cdot\nabla u_0),$$
and
\begin{align*}
&\nabla P_0^{(2)}\eqdefa-\f12\nabla(-\D)^{-1}\Bigl\{\pa_1\Bigl(\cos\th\sum_{k=1}^\infty
\bigl(\wt u_k\cdot\wt\nabla\ur_{k}-\f{|\vt_{k}|^2}r
-kN\f{\vt_k\ur_k}r\bigr)\Bigr)\\
&\qquad+\pa_2\Bigl(\sin\th\sum_{k=1}^\infty
\bigl(\wt u_k\cdot\wt\nabla\ur_{k}-\f{|\vt_{k}|^2}r
-kN\f{\vt_k\ur_k}r\bigr)\Bigr)
+\pa_3\sum_{k=1}^\infty\Bigl(\wt u_k\cdot\wt\nabla\uz_{k}
-kN\f{\vt_k\uz_k}r\Bigr)\Bigr\}.
\end{align*}
Observing that
\begin{align*}
\int_0^t\int_{\R^3} |\nabla P_0^{(2)}| |u_0|^2\,dx\,dt'
\leq\|\nabla P_0^{(2)}\|_{L^{\f53}_t(L^{\f53})}
\|u_0\|_{L^\infty_t(L^3)}^{\f45}\|u_0\|_{L^3_t(L^9)}^{\f65},
\end{align*}
from which and the $L^p$ boundedness of the operator $\nabla(-\D)^{-1}\nabla$, we deduce that
the term $\int_0^t\int_{\R^3} |\nabla P_0^{(2)}| |u_0|^2\,dx\,dt'$ shares
 the same estimate as $\cB_2$ in \eqref{4.11}.

Similarly, we deduce that
\begin{align*}
\int_0^t\int_{\R^3} |\nabla P_0^{(1)}| |u_0|^2\,dx\,dt'
&\leq\|\nabla P_0^{(1)}\|_{L^{2}_t(L^{\f32})}
\|u_0\|_{L^\infty_t(L^3)}^{\f12}\|u_0\|_{L^3_t(L^9)}^{\f32}\\
&\lesssim\|u_0\cdot\nabla u_0\|_{L^{2}_t(L^{\f32})}
\|u_0\|_{L^\infty_t(L^3)}^{\f12}\bigl\|\nabla|u_0|^{\f32}\bigr\|_{L^2_t(L^2)}\\
&\lesssim\bigl\||\nabla u_0||u_0|^{\f12}\bigr\|_{L^2_t(L^2)}
\|u_0\|_{L^\infty_t(L^3)}\bigl\|\nabla|u_0|^{\f32}\bigr\|_{L^2_t(L^2)}.
\end{align*}

By substituting the above estimates into \eqref{S4eq3}, we find
\begin{align*}
\bigl(\f12-C\|u_0\|_{L^\infty_t(L^3)}\bigr)
\Bigl(\f13\|u_0\|_{L^\infty_t(L^3)}^3
+\bigl\||\nabla u_0||u_0|^{\f12}\bigr\|_{L^2_t(L^2)}^2
+\f49\bigl\|\nabla|u_0|^{\f32}\bigr\|_{L^2_t(L^2)}^2\Bigr)
\leq C_{\rm in}N^{-\f35}.
\end{align*}
Then by virtue of \eqref{4.1} as long as $\e$ is so small that $C\e\leq \f14,$
we achieve \eqref{S4eq4}.
\end{proof}

\begin{proof}[{\bf Proof of Lemma \ref{propuH1}}]
We first get, by using Hausdorff-Young's inequality, that
$$\|u-u_0\|_{L^5([0,2\pi],d\th)}
\leq C\bigl\|(\upsilon_{k})_{k\in\N^+}\bigr\|_{\ell^{\f54}(\N^+)}.$$
By taking $L^5_t(L^5(rdrdz))$ norm of the above inequality
and then using Minkowski's inequality and \eqref{4.9}, we find
\begin{equation}\begin{split}\label{4.12}
\|u-u_0\|_{L^5_t(L^5)}
&\leq C\bigl\|\bigl(\|\upsilon_{k}\|_{L^5_t(L^5)}\bigr)
_{k\in\N^+}\bigr\|_{\ell^{\f54}(\N^+)}\\
&\leq C\bigl\|\bigl(k^{\f45}\|\upsilon_k\|_{L^5_t(L^5)}\bigr)
_{k\in\N^+}\bigr\|_{\ell^5(\N^+)}
\bigl\|\bigl(k^{-\f45}\bigr)_{k\in\N^+}\bigr\|_{\ell^{\f53}(\N^+)}\\
&\leq C_{\rm in}N^{-\f25}.
\end{split}\end{equation}
While it follows from Lemma \ref{propuL2}  that
$$\|u_0\|_{L^5_t(L^5)}\leq\|u_0\|_{L^\infty_t(L^3)}^{\f25}
\|u_0\|_{L^3_t(L^9)}^{\f35}
\leq C\|u_0\|_{L^\infty_t(L^3)}^{\f25}
\bigl\|\nabla|u_0|^{\f32}\bigr\|_{L^2_t(L^2)}^{\f25}
\leq C_{\rm in}N^{-\f15},$$
which together  with \eqref{4.12} ensures that
\begin{equation}\label{4.13}
\|u\|_{L^5_t(L^5)}\leq C_{\rm in}N^{-\f15},\quad\forall\ 0<t\leq T^\star.
\end{equation}

On the other hand, we get,
by taking $L^2$ inner product of
\eqref{NS} with $-\D u$, that
\begin{align*}
\f12\f{d}{dt}&\|\nabla u\|_{L^2}^2+\|\D u\|_{L^2}^2
=-\int_{\R^3}(u\cdot\nabla u)\cdot\D u\,dx\\
&\leq\|u\|_{L^5}\|\nabla u\|_{L^{\f{10}3}}\|\D u\|_{L^2}
\leq C\|u\|_{L^5}\|\nabla u\|_{L^2}^{\f25}\|\D u\|_{L^2}^{\f85}
\leq\f12\|\D u\|_{L^2}^2
+C\|u\|_{L^5}^5\|\nabla u\|_{L^2}^2,
\end{align*}
which implies
\begin{align*}
\f{d}{dt}\|\nabla u\|_{L^2}^2+\|\D u\|_{L^2}^2\leq 2C\|u\|_{L^5}^5\|\nabla u\|_{L^2}^2.
\end{align*}
By applying Gronwall's inequality and using the estimate \eqref{4.13}, we achieve
\begin{align*}
\|\nabla u\|_{L^\infty_t(L^2)}^2
+\|\D u\|_{L_t^2(L^2)}^2&\leq \|\nabla u_{\rm in}\|_{L^2}^2
\exp\bigl(C\|u\|_{L^5_tL^5}^5\bigr)\\
&\leq \|\nabla u_{\rm in}\|_{L^2}^2
\exp\bigl(C_{\rm in}N^{-1}\bigr)\leq\f32\|\nabla u_{\rm in}\|_{L^2}^2,
\end{align*}
provided that $N$ is sufficiently large.
This together with \eqref{4.13} completes the proof of Lemma \ref{propuH1}.
\end{proof}

\section{The proof of Theorem \ref{thm2}}\label{secthm2}

The purpose of this section is to present the proof of Theorem \ref{thm2} through a perturbation argument.

\begin{proof}[Proof of Theorem \ref{thm2}]
For some small positive constant $\e$ to be chosen later on,
the assumption \eqref{condsumk} ensures the existence of some integer $n_0$ such that
\begin{equation}\label{6.1}
\sum_{k=n_0+1}^\infty\|\Upsilon_k\|_{L^3}^{\f32}\leq \e^{\f32}.
\end{equation}
And we decompose the initial data in \eqref{initialsum} into three parts:
$u_{\rm in}=u_{\rm in}^{(1)}
+u_{\rm in}^{(2)}+u_{\rm in}^{(3)}$ with
\beq \label{S6eq1}
\begin{split}
&u_{\rm in}^{(1)}\eqdefa\bar u^r_{\rm in}\vv e_r+\bar u^z_{\rm in}\vv e_z,
\quad u_{\rm in}^{(2)}\eqdefa\sum_{k=1}^{n_0} w_{k,\rm in},
\quad u_{\rm in}^{(3)}\eqdefa\sum_{k=n_0+1}^\infty w_{k,\rm in} \andf\\
&w_{k,\rm in}\eqdefa \bigl(a^r_k\vv e_r+N^{-1}_ka^\th_k\vv e_\th+a^z_k \vv e_z\bigr)\cos N_k\th
+\bigl(b^r_k\vv e_r+N^{-1}_kb^\th_k\vv e_\th+b^z_k \vv e_z\bigr)\sin N_k\th.
\end{split} \eeq

We first deduce from Theorem 1 of \cite{LMNP}  that the system
\eqref{NS} with initial data $u^{(1)}_{\rm in}$
admits a unique global solution $\bar u$ which satisfies
\begin{equation}\label{6.2}
\|\bar u\|_{L^\infty_t(H^1)}^2+\|\nabla\bar u\|_{L^2_t(H^1)}^2
\leq C(\|u^{(1)}_{\rm in}\|_{H^2}),\quad\forall\ t>0.
\end{equation}
Here and in all that follows, we always use $C(\star)$ to denote some positive constant depending only on $\star$.

While for $N_0$ large enough and $N_k\geq N_0,$ it follows Theorem \ref{thm1} that the system
\eqref{NS} with initial data $w_{k,\rm in}$, for $k=1,\cdots,n_0,$
has a unique global
solution $w_k,$ which satisfies
\begin{equation}\label{6.3}
\begin{split}
&\|w_k\|_{L^5(\R^+;L^5)}\leq C(M_k)N_k^{-\f15}\leq 1
\with M_k\eqdefa \bigl\|(a^r,a^\th,a^z,b^r,b^\th,b^z)\bigr\|_{\cM}.
\end{split}
\eeq

Let $w$  be the unique local solution of \eqref{NS} with initial data $u_{\rm in}^{(1)}+u_{\rm in}^{(2)}$ on $[0,T^\ast_1[,$
where $T_1^\ast$ is the lifespan of $w.$ Then it is easy to verify that
$v\eqdefa w-\bar u-\sum_{k=1}^{n_0} w_k$ satisfies
\begin{equation}\label{6.4}
\left\{
\begin{aligned}
&\p_t v-\Delta v+\PP\dive\bigl(v\otimes v
+v\otimes \bar u+\bar u\otimes v\bigr)
+\sum_{k=1}^{n_0}\PP\dive
\bigl(v\otimes w_k+w_k\otimes v\bigr)\\
&\qquad+\sum_{k=1}^{n_0}\PP\dive
\bigl(\bar u\otimes w_k+w_k\otimes \bar u\bigr)
+\sum_{1\leq i<j\leq n_0}\PP\dive
\bigl(w_i\otimes w_j+w_j\otimes w_i\bigr)= 0,\\
&\dive v=0,\\
& v|_{t=0} =0,
\end{aligned}
\right.
\end{equation}
where $\PP\eqdefa{\rm I}+\na(-\D)^{-1}\dive$ denotes the Leray projection operator into divergence-free vector fields space.

In what follows, we just present the {\it a priori} estimates.
By taking $L^2$ inner product of the $v$ equation of \eqref{6.4} with $|v|v$
 and using integration by parts, we find
\begin{equation}\begin{split}\label{6.5}
\f13&\f{d}{dt}\|v\|_{L^3}^3+\f49\bigl\|\nabla|v|^{\f32}\bigr\|_{L^2}^2+\bigl\||v|^{\f12}\nabla v\bigr\|_{L^2}^2\\
&\leq C\bigl(\f49\bigl\|\nabla|v|^{\f32}\bigr\|_{L^2}
+\bigl\||v|^{\f12}\nabla v\bigr\|_{L^2}\bigr)
\Bigl\{\bigl\||v|^{\f32}\bigr\|_{L^6}^{\f35}
\bigl\||v|^{\f32}\bigr\|_{L^2}^{\f25}
\Bigl(\|\bar u\|_{L^5}+\sum_{k=1}^{n_0}\|w_k\|_{L^5}\Bigr)\\
&\qquad+\bigl\||v|^{\f32}\bigr\|_{L^6}\|v\|_{L^3}
+\bigl\||v|^{\f12}\bigr\|_{L^{10}}
\Bigl(\sum_{k=1}^{n_0}\|\bar u\|_{L^5}\|w_k\|_{L^5}
+\sum_{1\leq i<j\leq n_0}
\|w_i\|_{L^5}\|w_j\|_{L^5}\Bigr)\Bigr\}\\
&\leq\bigl(\f14+C_1\|v\|_{L^3}\bigr)
\bigl(\f49\bigl\|\nabla|v|^{\f32}\bigr\|_{L^2}^2
+\bigl\||v|^{\f12}\nabla v\bigr\|_{L^2}^2\bigr)
+C\|v\|_{L^3}^{3}\Bigl(\|\bar u\|_{L^5}^5+n_0^4\sum_{k=1}^{n_0}\|w_k\|_{L^5}^5\Bigr)\\
&\quad
+C\Bigl(\sum_{k=1}^{n_0}\|\bar u\|_{L^5}\|w_k\|_{L^5}
+\sum_{1\leq i<j\leq n_0}
\|w_i\|_{L^5}\|w_j\|_{L^5}\Bigr)^{\f52},
\end{split}\end{equation}
where we used the fact that
\begin{align*}
\bigl\||v|^{\f12}\bigr\|_{L^{10}}
=\bigl\||v|^{\f32}\bigr\|_{L^{\f{10}3}}^{\f13}
\leq \bigl\||v|^{\f32}\bigr\|_{L^{2}}^{\f2{15}} \bigl\||v|^{\f32}\bigr\|_{L^{6}}^{\f15}
\leq C\|v\|_{L^3}^{\f15}\bigl\|\na |v|^{\f32}\bigr\|_{L^{2}}^{\f15}.
\end{align*}

Notice that $v|_{t=0} =0$, we  denote
\beq \label{S6eq2} T_1^\star\eqdef \sup\Bigl\{\ t<T^\ast_1: \ \|v\|_{L^\infty_t(L^3)}\leq \f1{4C_1}
\ \ \Bigr\}.\eeq
Then for any $0\leq t\leq T^\star_1$, we deduce from \eqref{6.5} that
\begin{align*}
\f{d}{dt}\|v\|_{L^3}^3
+\bigl\|\nabla|v|^{\f32}\bigr\|_{L^2}^2+\bigl\||v|^{\f12}\nabla v\bigr\|_{L^2}^2
\leq& C\|v\|_{L^3}^{3}\Bigl(\|\bar u\|_{L^5}^5+n_0^4\sum_{k=1}^{n_0}\|w_k\|_{L^5}^5\Bigr)\\
&+C\Bigl(\|\bar u\|_{L^5}^5
+n_0^4\sum_{k=1}^{n_0}\|w_k\|_{L^5}^5\Bigr)^{\f12}
\Bigl(n_0^4\sum_{k=1}^{n_0}\|w_k\|_{L^5}^5\Bigr)^{\f12}.
\end{align*}
By applying Gronwall's inequality and using the estimates \eqref{6.2} and \eqref{6.3}, we obtain
\begin{align*}
\|v&\|_{L^\infty_t(L^3)}^3
+\bigl\|\nabla|v|^{\f32}\bigr\|_{L^2_t(L^2)}^2+\bigl\||v|^{\f12}\nabla v\bigr\|_{L^2_t(L^2)}^2\\
&\leq C \Bigl(n_0^4\sum_{k=1}^{n_0}\|w_k\|_{L^5_t(L^5)}^5\Bigr)^{\f12}
\exp\Bigl(C\|\bar u\|_{L^5_t(L^5)}^5
+Cn_0^4\sum_{k=1}^{n_0}\|w_k\|_{L^5_t(L^5)}^5\Bigr)\\
&\leq C_2n_0^{\f52}N_0^{-\f12}
\exp\Bigl(C_2n_0^5N_0^{-1}\Bigr),
\end{align*}
where $C_2$ is some positive constant depending only on $\|u^{(1)}_{\rm in}\|_{H^2}$ and $M_1,\cdots,M_{n_0}$ defined by \eqref{6.3}.
It is easy to observe that if $N_0$ is sufficiently large, we find
$$\|v\|_{L^\infty_t(L^3)}^3
+\bigl\|\nabla|v|^{\f32}\bigr\|_{L^2_t(L^2)}^2+\bigl\||v|^{\f12}\nabla v\bigr\|_{L^2_t(L^2)}^2
\leq(8C_1)^{-3},\quad 0\leq t\leq T_1^\star,$$
which contradicts with the definition of $T^\star_1$ given by \eqref{S6eq2}. This in turn shows that $T^\star_1=T^\ast_1=\infty,$ and
for any $t>0,$
there holds
\begin{equation}\label{6.6}
\|v\|_{L^\infty_t(L^3)}^3
+\bigl\|\nabla|v|^{\f32}\bigr\|_{L^2_t(L^2)}^2+\bigl\||v|^{\f12}\nabla v\bigr\|_{L^2_t(L^2)}^2
\leq(4C_1)^{-3}.
\end{equation}

Thanks to \eqref{6.2},~\eqref{6.3} and \eqref{6.6}, we deduce that $w=v+\bar u+\sum_{k=1}^{n_0} w_k$ satisfies
\begin{equation}\label{6.7}
\|w\|_{L^5(\R^+\times\R^3)}\leq C\bigl(\|u^{(1)}_{\rm in}\|_{H^2}\bigr),
\end{equation}
provided that $N_0$ is sufficiently large.

Let
$\delta\eqdefa u-w.$ Then $\delta$
 verifies
\begin{equation}\label{S6eq3}
\left\{
\begin{aligned}
&\p_t \delta-\Delta \delta
+\PP\dive\bigl(\delta\otimes \delta
+\delta\otimes w+w\otimes \delta\bigr)=0,\\
&\dive \delta=0,\\
& \delta|_{t=0}=u_{\rm in}^{(3)}=\sum_{k=n_0+1}^\infty w_{k,\rm in},
\end{aligned}
\right.
\end{equation}
It is easy to observe that the system \eqref{S6eq3} has a unique solution on $[0,T_2^\ast[$ with
$T_2^\ast$ being the maximal time of existence. Therefore to complete the proof of Theorem \ref{thm2},
it remains to prove that $T_2^\ast=\infty.$

Indeed, for any $t<T_2^\ast,$ we get, by a similar derivation of \eqref{6.5}, that
\beq\label{S6eq4}
\begin{split}
\f13&\f{d}{dt}\|\delta\|_{L^3}^3
+\f49\bigl\|\nabla|\delta|^{\f32}\bigr\|_{L^2}^2+\bigl\||\delta|^{\f12}\nabla \delta\bigr\|_{L^2}^2\\
&\leq C\bigl(\f49\bigl\|\nabla|\delta|^{\f32}\bigr\|_{L^2}+\bigl\||\delta|^{\f12}\nabla \delta\bigr\|_{L^2}\bigr)
\Bigl(\bigl\||\delta|^{\f32}\bigr\|_{L^6}
\|\delta\|_{L^3}
+\bigl\||\delta|^{\f32}\bigr\|_{L^6}^{\f35}
\bigl\||\delta|^{\f32}\bigr\|_{L^2}^{\f25}\|w\|_{L^5}\Bigr)\\
&\leq\bigl(\f14+C_1\|\delta\|_{L^3}\bigr)\bigl(\f49
\bigl\|\nabla|\delta|^{\f32}\bigr\|_{L^2}^2+\bigl\||\delta|^{\f12}\nabla \delta\bigr\|_{L^2}^2\bigr)
+C\|\delta\|_{L^3}^{3}\|w\|_{L^5}^5.
\end{split}\eeq
It follows from \eqref{6.1}, \eqref{S6eq1} and Hausdorff-Young's inequality guarantee that
$$\bigl\|u_{\rm in}^{(3)}\bigr\|_{L^3}
=\Bigl\|\sum_{k=n_0+1}^\infty w_{k,\rm in}\Bigr\|_{L^3}
\leq C\Bigl(\sum_{k=n_0+1}^\infty\|\Upsilon_k\|_{L^3}^{\f32}
\Bigr)^{\f23}<C\e.$$
Hence thanks to \eqref{6.7} and \eqref{S6eq4}, we get,  by a similar derivation of \eqref{6.6}, that
\begin{equation}\label{6.9}
\begin{split}
\|\delta\|_{L^\infty_t(L^3)}^3
+\f12\bigl\|\nabla|\delta|^{\f32}\bigr\|_{L^2_t(L^2)}^2
\leq &\bigl\|u_{\rm in}^{(3)}\bigr\|_{L^3}^3\exp\Bigl(C\|w\|_{L^5_t(L^5)}^5\Bigr)\\
\leq & C\bigl(\|u_{\rm in}^{(1)}\|_{H^2}\bigr)\e
\leq \bigl(4C_1\bigr)^{-3},\quad\forall\ t<T_2^\ast,
\end{split}
\end{equation} provided that $\e$ is chosen to be so small that $\e\leq \bigl(4C_1\bigr)^{-3}C^{-1}\bigl(\|u_{\rm in}^{(1)}\|_{H^2}\bigr).$
This in turn shows that $T^\ast_2=\infty.$

Due to $u=\delta+w,$
we deduce from  \eqref{6.6} and \eqref{6.9} that
\begin{equation}\label{6.10}
\|u\|_{L^5_t(L^5)}\leq C\bigl(\|u^{(1)}_{\rm in}\|_{H^2}\bigr),
\quad\forall\ t\geq0.
\end{equation}
Then the classical Ladyzhenskaya-Prodi-Serrin's criteria (see Theorem 1.2 of \cite{ISS03} for instance) ensures that the strong solution $u$ of \eqref{NS} exists globally in time. This finishes the proof of Theorem \ref{thm2}.
\end{proof}

\appendix

\section{The proof  of Lemmas \ref{propa1} and  \ref{propa2}}\label{sectA}

\begin{proof}[{\bf Proof of Lemma \ref{propa1}}]
 Let us first focus on the case when $s$ is away from zero. Indeed for any $s>4$ and any $\th\in [0,\pi]$,
we get, by using Taylor's expansion, that
there exists some $\tau\in(0,1)$ so that
\begin{equation}\begin{split}\label{a4}
\Bigl(1+\f{2(1-\cos\th)}s\Bigr)^{-\f12}
=&\sum_{j=0}^{m-1}\binom{-\f12}{j}
\Bigl(\f{2(1-\cos\th)}s\Bigr)^j
+\binom{-\f12}{m}\Bigl(\f{2(1-\cos\th)}s\Bigr)^m\\
&+\binom{-\f12}{m}\Bigl(\Bigl(1+\tau\f{2(1-\cos\th)}s\Bigr)^{-(\f12+m)}-1\Bigr)
\Bigl(\f{2(1-\cos\th)}s\Bigr)^m,
\end{split}\end{equation}
here the first term on the right hand side of \eqref{a4} vanishes when $m=0$.

Since $\cos^j\th$ can be written as a linear combination
of $1,~\cos\th,~\cdots,~\cos j\th$, and
$$\int_0^\pi\cos m\th\cdot\cos j\th\,d\th=0,\quad\forall\ j=0,1,\cdots,m-1,$$
we  deduce that
$$\f{1}{\sqrt s}\int_0^\pi\cos m\th
\sum_{j=0}^{m-1}\binom{-\f12}{j}
\Bigl(\f{2(1-\cos\th)}s\Bigr)^j\,d\th=0,$$
and
\begin{align*}
\f{1}{\sqrt s}\int_0^\pi\cos m\th\Bigl(\f{2(1-\cos\th)}s\Bigr)^m\,d\th
&=(-2)^ms^{-m-\f12}\int_0^\pi\cos m\th\cos^m\th\,d\th\\
&=(-1)^m2s^{-m-\f12}\int_0^\pi\cos m\th
\cos m\th\,d\th\\
&=(-1)^m\pi s^{-m-\f12}.
\end{align*}
On the other hand, noticing that
$$\Bigl(1+\tau\f{2(1-\cos\th)}s\Bigr)^{-(\f12+m)}-1
\in(-1,0),$$
from which, we infer
\begin{align*}
&\f{1}{\sqrt s}\int_0^\pi\Bigl|\Bigl(1+\tau\f{2(1-\cos\th)}s\Bigr)^{-(\f12+m)}-1\Bigr|
\times\Bigl|\f{2(1-\cos\th)}s\Bigr|^m\,d\th
\leq 4^m\pi
 s^{-m-\f12}.
\end{align*}
By summarizing
 the above three estimates
and using the fact that $\bigl|\binom{-\f12}{m}\bigr|
\thicksim\f{1}{\sqrt{1+m}}$, we achieve
\begin{equation}\label{a5}
\bigl|s^{m+\f12}F_m(s)\bigr|\leq4^m\pi,
\quad\forall\ s>4.
\end{equation}

In the case when $s\rightarrow0^+$, we write
\begin{align*}
F_m(s)=&\int_0^\pi\f{\cos m\th-\cos\f{\th}2+\f s2+(1-\cos\th)}
{\sqrt{s+2(1-\cos\th)}}\,d\th\\
&+\int_0^\pi\f{\cos\f{\th}2}{\sqrt{s+2(1-\cos\th)}}\,d\th
-\f12\int_0^\pi \sqrt{s+2(1-\cos\th)}\,d\th\eqdef\cC_1+\cC_2+\cC_3.
\end{align*}
It is easy to calculate that
$$\cC_2=\int_0^\pi\f{\cos\f{\th}2}{\sqrt{s+4\sin^2\f{\th}2}}\,d\th
=\int_0^1\f{dx}{\sqrt{\f s4+ x^2}}=\ln\Bigl(\sqrt{\f 4s}+\sqrt{1+\f4s}\Bigr),$$
and
$$\cC_3=-\f12\int_0^\pi \sqrt{2(1-\cos\th)}\,d\th+\cO(s)=-2+\cO(s).$$
While by using Taylor's expansion once again, we write
\begin{align*}
\cC_1&=\int_0^\pi\f{\cos m\th-\cos\f{\th}2+(1-\cos\th)}
{\sqrt{s+2(1-\cos\th)}}\,d\th+\cO(\sqrt s)\\
&=2\int_0^{\f\pi2}\f{1-\cos\varphi-\cos 2\varphi+\cos 2m\varphi}
{\sqrt{s+4\sin^2\varphi}}\,d\varphi+\cO(\sqrt s)\eqdef g_s(m)+\cO(\sqrt s).
\end{align*}
It is easy to verify that $g_s(0)$ and $g_s(1)$ are some finite constants.
As for $m\geq1$, we first get, by using integration by parts, that
\begin{align*}
g_s'(m)=&-4\int_0^{\f\pi2}\f{\varphi\sin 2m\varphi}
{\sqrt{s+4\sin^2\varphi}}\,d\varphi\\
=&\f2m\Bigl(\f{(-1)^m\pi}{2\sqrt{s+4}}-\lim_{\varphi\to 0_+}\f{\varphi}{\sqrt{s+4\sin^2\varphi}}\Bigr)-\f2m\int_0^{\f\pi2}
\cos 2m\varphi\cdot\bigl(\f\varphi{\sqrt{s+4\sin^2\varphi}}\bigr)'\,d\varphi.\end{align*}
This together with the fact that
$$0\leq \lim_{\varphi\to 0_+}\f{\varphi}{\sqrt{s+4\sin^2\varphi}}\leq \lim_{\varphi\to 0_+}\f{\varphi}{2\sin\varphi}=\f12,$$
and
$\bigl(\f\varphi{\sqrt{s+4\sin^2\varphi}}\bigr)'$
is bounded for $\varphi\in[0,\f\pi2]$, implies that $m|g_s'(m)|$
is uniformly bounded for any $m\geq1$. Therefore, for any $m\geq1$, we find
$$|g_s(m)|\leq |g_s(1)|+\int_1^m|g_s'(x)|\,dx\leq C+\int_1^m\f1x\,dx
\leq C\ln(2+m),$$
which ensures that
$$|\cC_1|\leq C\ln(2+m)+\cO(\sqrt s).$$
By summarizing the estimates for $\cC_1,~\cC_2$ and $\cC_3$,
we deduce that there exists some $\ve>0$ so that
\begin{equation}\label{a6}
\bigl|s^\lambda F_m(s)\bigr|<C\lambda^{-1}\ln(2+m),\quad\forall\ \lambda>0,~ s<\ve.
\end{equation}

Finally when $\ve\leq s\leq 4$, it is easy to see that
\begin{equation}\label{a7}
\bigl|s^\lambda F_m(s)\bigr|\leq s^\lambda \int_0^\pi\f1{\sqrt{s}}\,d\th
=\pi s^{\lambda-\f12}
\leq\left\{\begin{array}{l}
\displaystyle \ve^{\lambda-\f12}\pi,\quad\text{if}\quad0<\lambda<\f12,\\
\displaystyle 4^{\lambda-\f12}\pi ,\quad\text{if}\quad\lambda\geq\f12.
\end{array}\right.
\end{equation}

The estimates \eqref{a5}-\eqref{a7} ensures the first
estimate for $F_m(s)$ in \eqref{a3}.
The other two estimates for $F_m'(s)$ and $F_m''(s)$ can be derived along the same line, we omit the details here.
\end{proof}

\begin{proof}[{\bf Proof of Lemma \ref{propa2}}]
(i) We first get, by using
integration by parts, that
\begin{equation}\label{a10}
s^\al F_m(s)=\f1m\f{s^\al\sin m\th}{\sqrt{s+2(1-\cos\th)}}\Bigl|_0^\pi
+\f{s^\al}m\int_0^\pi\sin m\th\sin\th\cdot\bigl(s+2(1-\cos\th)\bigr)^{-\f32}\,d\th.
\end{equation}
Observing that for $
0<\al\leq \f12$ and $s>0$, one has
$$\f{|2s^\al\sin\f\th2|}{\sqrt{s+4\sin^2{\f\th2}}}
\leq\f{|2\sin\f\th2|}{\bigl(s+4\sin^2{\f\th2}\bigr)^{\f12-\al}}
\leq|2\sin\f\th2|^{2\al},$$
while for $\al>\f12,$ one has
$$\f{|2s^\al\sin\f\th2|}{\sqrt{s+4\sin^2{\f\th2}}}
\leq 2s^{\al-\f12}|\sin\f\th2|,$$
which implies
$$\lim_{\th\rightarrow0_+}\f1m\f{s^\al\sin m\th}{\sqrt{s+2(1-\cos\th)}}
=\lim_{\th\rightarrow0_+}\f{2s^\al\sin\f\th2}{\sqrt{s+4\sin^2{\f\th2}}}=0.$$
By inserting the above estimate into \eqref{a10}, we obtain
\begin{equation}\label{a11}
s^\al F_m(s)=\f{s^\al}m G_m(s),
\end{equation}
where
\begin{align*}
G_m(s)&\eqdefa\int_0^\pi\sin m\th\sin\th
\cdot\bigl(s+2(1-\cos\th)\bigr)^{-\f32}\,d\th\\
&=\f12\int_0^\pi\Bigl(\cos (m-1)\th-\cos(m+1)\th\Bigr)
\cdot\bigl(s+2(1-\cos\th)\bigr)^{-\f32}\,d\th.
\end{align*}

For $s>4,$ we get, by using
 Taylor's expansion, that
\begin{align*}
\Bigl|\Bigl(1+\f{2(1-\cos\th)}s\Bigr)^{-\f32}
-\Bigl(1-\f{3(1-\cos\th)}s\Bigr)\Bigr|\leq \f{C}{s^2},
\quad\forall\ s>4,
\end{align*}
which implies
\begin{align*}
\Bigl|G_m(s)-s^{-\f32}\int_0^\pi\sin m\th\sin\th
\Bigl(1-\f{3(1-\cos\th)}s\Bigr)\,d\th\Bigr|\leq C{s^{-\f72}}.
\end{align*}
This together with the fact that
\begin{align*}
\int_0^\pi\cos (m-1)\th
\Bigl(1-\f{3(1-\cos\th)}s\Bigr)\,d\th=\int_0^\pi\cos (m+1)\th
\Bigl(1-\f{3(1-\cos\th)}s\Bigr)\,d\th=0
\end{align*}
leads to
\begin{equation}\label{a12}
\bigl|s^{\f72}G_m(s)\bigr|\leq C,
\quad\forall\ s>4,~ m\geq3.
\end{equation}

Next, let us consider the case when $s$ is small. Observing that
$$\Bigl|s\cdot\sin m\th\sin\th
\bigl(s+2(1-\cos\th)\bigr)^{-\f32}\Bigr|
=\Bigl|2\sin m\th\cos{\f\th2}\Bigl(s\cdot\sin{\f\th2}
\bigl(s+4\sin^2{\f\th2}\bigr)^{-\f32}\Bigr)\Bigr|\leq2,$$
we find
\begin{equation}\label{a13}
\bigl|sG_m(s)\bigr|\leq2\pi,
\quad\forall\ 0<s\leq4,~ m\geq3.
\end{equation}

Combining the estimates \eqref{a11}-\eqref{a13}
gives rise to the first inequality in \eqref{a8}.
And the second estimate in \eqref{a8}
for $F_m'(s)$ and $F_m''(s)$ can be derived similarly.
\smallskip

\noindent(ii) As in the previous step, here we only present the proof for the first
inequality in \eqref{a9}. Once again, we first get, by  using Taylor's expansion,  that  for any $s>4$
\begin{align*}
\Bigl|\Bigl(1+\f{2(1-\cos\th)}s\Bigr)^{-\f12}
-\Bigl(1-\f{1-\cos\th}s\Bigr)\Bigr|\leq \f{C}{s^2},
\end{align*}
which implies
\begin{align*}
\Bigl|F_m(s)-\f{1}{\sqrt s}\int_0^\pi\cos m\th
\Bigl(1-\f{1-\cos\th}s\Bigr)\,d\th\Bigr|\leq C{s^{-\f52}}.
\end{align*}
This together with the fact that
\begin{align*}
\int_0^\pi\cos m\th
\Bigl(1-\f{1-\cos\th}s\Bigr)\,d\th=0,\quad\forall\ m\geq3,
\end{align*}
leads to
$$\bigl|s^{\f52}F_m(s)\bigr|\leq C,
\quad\forall\ s>4,~ m\geq3,$$
which together with \eqref{a6} and \eqref{a7}
yields  the first
inequality in \eqref{a9}.
\end{proof}

\section{The proof of \eqref{3.39} with $j=2,3,4$}\label{sectB}

\begin{proof}[{\bf Proof of \eqref{3.39} with $j=2,3$}] In view of \eqref{S2eq6d},
we first get, by  using
integration by parts and Proposition \ref{lemLm}, that
\begin{align*}
&\Psi_k^{(4,2)}\lesssim\bigl\|r^{2-\f7p}\cL_{kN}^{-1}
\cA_k^{(4,2)}\bigr\|_{L^{\f p2}_t(L^p)}
\bigl\|\wt\nabla \bigl(r^{\f{p-3}2}|\wt u_k|^{\f p2}\bigr)\bigr\|_{L^2_t(L^2)}
\bigl\|r^{-\f72+\f7p+\f p2}|\wt u_k|^{\f{p}2-1}\bigr\|_{L^{\f{2p}{p-4}}_t(L^{\f{2p}{p-2}})}\\
&\lesssim_\e(kN)^{-\f12+\f1p+\e}\sum_{(k_1,k_2)\in\Omega_k}
\Bigl(\bigl\|r^{\f{2(p-5)}p}\upsilon_{k_1}\otimes\upsilon_{k_2}
\bigr\|_{L^{\f p2}_t(L^{\f p2})}
+(k_1+k_2)N\bigl\|r^{\f{2(p-5)}p}\vt_{k_1}\ur_{k_2}\bigr\|_{L^{\f p2}_t(L^{\f p2})}\Bigr)\\
&\qquad\times\bigl\|\wt\nabla \bigl(r^{\f{p-3}2}
|\wt u_k|^{\f p2}\bigr)\bigr\|_{L^2_t(L^2)}
\bigl\|r^{\f{p-3}2}|\wt u_{k}|^{\f p2}\bigr\|_{L^\infty_t(L^2)}^{\f2p}
\bigl\|r^{\f{p-5}2}|\wt u_{k}|^{\f p2}\bigr\|_{L^2_t(L^2)}^{1-\f4p}
\eqdef \Psi_k^{(4,2,1)}+\Psi_k^{(4,2,2)}.
\end{align*}
Similarly, one has
\begin{equation}\begin{split}\label{3.35}
\Psi_k^{(4,3)}
&\lesssim_\e(kN)^{\f12+\f1p+\e}\sum_{(k_1,k_2)\in\Omega_k}
\bigl\|r^{\f{2(p-5)}p }\upsilon_{k_1}\vt_{k_2}
\bigr\|_{L^{\f p2}_t(L^{\f p2})}\\
&\qquad\times\bigl\|\wt\nabla \bigl(r^{\f{p-3}2}
|\wt u_k|^{\f p2}\bigr)\bigr\|_{L^2_t(L^2)}
\bigl\|r^{\f{p-3}2}|\wt u_{k}|^{\f p2}\bigr\|_{L^\infty_t(L^2)}^{\f2p}
\bigl\|r^{\f{p-5}2}|\wt u_{k}|^{\f p2}\bigr\|_{L^2_t(L^2)}^{1-\f4p}.
\end{split}\end{equation}
Noticing that
\begin{align*}
\bigl\|r^{\f{2(p-5)}p}\upsilon_{k_1}\otimes\upsilon_{k_2}
\bigr\|_{L^{\f p2}_t(L^{\f p2})}
\leq &\bigl\|r^{\f{p-5}2}|\upsilon_{k_1}|^{\f p2}\bigr\|_{L^2_t(L^2)}^{\f2p}
\bigl\|r^{\f{p-5}2}|\upsilon_{k_2}|^{\f p2}\bigr\|_{L^2_t(L^2)}^{\f2p},\\
\bigl\|r^{\f{2(p-5)}p}\upsilon_{k_1}\vt_{k_2}
\bigr\|_{L^{\f p2}_t(L^{\f p2})}
\leq&\bigl\|r^{\f{p-5}2}|\upsilon_{k_1}|^{\f p2}\bigr\|_{L^2_t(L^2)}^{\f2p}
\bigl\|r^{\f{p-5}2}|\vt_{k_2}|^{\f p2}\bigr\|_{L^2_t(L^2)}^{\f2p},\end{align*}
and there is $(kN)^{2}$ in the front of
$\bigl\|r^{\f{p-5}2}|\wt u_{k}|^{\f p2}\bigr\|_{L^2_t(L^2)}^2$
(resp. $\bigl\|r^{\f{p-5}2}|\vt_{k}|^{\f p2}\bigr\|_{L^2_t(L^2)}^2$)
in the definition of $E_p^{r,z}$  given by \eqref{defEpr} (resp. $E_p^{\th}$ by \eqref{defEpth}),
and there is an additional $(kN)^{\f p2}$ in front of the terms
in $E_p^\th$ than the corresponding terms in $E_p^{r,z},$ so that along the same line to the estimate of
$\int_0^t I_k^{(4,1)}\,dt'$ (see the proof of Lemma \ref{S2lem1}),
$\Psi_k^{(4,2,1)}$ and $\Psi_k^{(4,3)}$  satisfy
\begin{equation}\begin{split}\label{3.36}
N^{-2\al_p}\Psi_1^{(4,2,1)}
+\sum_{k\geq2}(kN)^{2\beta_p} \Psi_k^{(4,2,1)}
&\lesssim_\e N^{-\f12+\f1p+\e}N^{-\f7p+1}N^{-1+\f4p} E_p^{1+\f1p}\lesssim E_p^{1+\f1p},
\end{split}\end{equation}
and
\begin{equation}\begin{split}\label{3.37}
N^{-2\al_p}\Psi_1^{(4,3)}
+\sum_{k\geq2}(kN)^{2\beta_p} \Psi_k^{(4,3)}
&\lesssim_\e N^{\f12+\f1p+\e}N^{-\f7p+1}N^{-1+\f4p}N^{-\f12} E_p^{1+\f1p}\lesssim E_p^{1+\f1p}.
\end{split}\end{equation}

The estimate of $\Psi_k^{(4,2,2)}$ is more complicated.
We shall split it into two parts according to the values of $k_1+k_2$.
The first part, which we denote as $\Psi_k^{(4,2,2,1)},$ contains terms in $\Psi_k^{(4,2,2)}$  with $k_1+k_2\leq3k$,
can be handled as follows
\begin{align*}\Psi_k^{(4,2,2,1)}\lesssim
&(kN)^{\f12+\f1p+\e}\sum_{\substack{(k_1,k_2)\in\Omega_k\\k_2\leq2k}}
\bigl\|r^{2(1-\f5p)}\vt_{k_1}\ur_{k_2}\bigr\|_{L^{\f p2}_t(L^{\f p2})}\\
&\times\bigl\|\wt\nabla \bigl(r^{\f{p-3}2}
|\wt u_k|^{\f p2}\bigr)\bigr\|_{L^2_t(L^2)}
\bigl\|r^{\f{p-3}2}|\wt u_{k}|^{\f p2}\bigr\|_{L^\infty_t(L^2)}^{\f2p}
\bigl\|r^{\f{p-5}2}|\wt u_{k}|^{\f p2}\bigr\|_{L^2_t(L^2)}^{1-\f4p},
\end{align*}
which can be controlled by the right-hand side of \eqref{3.35}.
Hence $\Psi_k^{(4,2,2,1)}$ shares the same estimate for $\Psi_k^{(4,3)}$
in \eqref{3.37}. Precisely, there holds
\begin{equation}\begin{split}\label{3.38}
N^{-2\al_p}\Psi_1^{(4,2,2,1)}
+\sum_{k\geq2}(kN)^{2\beta_p} \Psi_k^{(4,2,2,1)}
\lesssim E_p^{1+\f1p}.
\end{split}\end{equation}

For $(k_1,k_2)\in\Omega_k$ with $k_1+k_2>3k$,
which satisfies also $|k_1-k_2|=k$, one has
\begin{equation}\label{3.38a}
\min\{k_1,k_2\}\geq k+1,\quad
\f12<\f{k_1}{k_2}<2,\quad 1<\f{k_1+k_2}{k_2}<3.
\end{equation}
Hence
the remaining part $\Psi_k^{(4,2,2,2)}$ in $\Psi_k^{(4,2,2)}$
can be bounded by
\beq\label{3.38b}
\begin{split}
\Psi_k^{(4,2,2,2)}\lesssim&(kN)^{-\f12+\f1p+\e}
\sum_{\substack{k_1+k_2>3k\\|k_1-k_2|=k}}
k_2N\bigl\|r^{2(1-\f5p)}\vt_{k_1}\ur_{k_2}\bigr\|_{L^{\f p2}_t(L^{\f p2})}\\
&\times\bigl\|\wt\nabla\bigl(r^{\f{p-3}2}
|\wt u_k|^{\f p2}\bigr)\bigr\|_{L^2_t(L^2)}
\bigl\|r^{\f{p-3}2}|\wt u_{k}|^{\f p2}\bigr\|_{L^\infty_t(L^2)}^{\f2p}
\bigl\|r^{\f{p-5}2}|\wt u_{k}|^{\f p2}\bigr\|_{L^2_t(L^2)}^{1-\f4p}.
\end{split} \eeq
By using \eqref{3.38a} again, the summation part in \eqref{3.38b} can be bounded by
\begin{align*}
&\sum_{\substack{k_1+k_2>3k\\|k_1-k_2|=k}}
k_2N\bigl\|r^{2(1-\f5p)}\vt_{k_1}\ur_{k_2}\bigr\|_{L^{\f p2}_t(L^{\f p2})}
\lesssim\Bigl(\sum_{k_1=k+1}^\infty(k_1N)^{\f p2+2+2\beta_p}
\bigl\|r^{1-\f5p}\vt_{k_1}\bigr\|_{L^p_t(L^p)}^p\Bigr)^{\f1p}\\
&\qquad\qquad\qquad\times\Bigl(\sum_{k_2=k+1}^\infty\bigl(k_2N\bigr)^{2+2\beta_p}
\bigl\|r^{1-\f5p}\ur_{k_2}\bigr\|_{L^p_t(L^p)}^p\Bigr)^{\f1p}
\Bigl(\sum_{k_2=k+1}^\infty\bigl(k_2N\bigr)^{\bigl(\f12-\f4p(1+\beta_p)\bigr)\f{p}{p-2}}\Bigr)^{1-\f2p},
\end{align*}
and the assumptions: $\beta_p>1$ and $p<6,$ ensures that
$$\Bigl(\sum_{k_2=k+1}^\infty\bigl(k_2N\bigr)
^{\bigl(\f12-\f4p(1+\beta_p)\bigr)\f{p}{p-2}}\Bigr)^{1-\f2p}
\lesssim(kN)^{\f32-\f6p-\f4p\beta_p}\lesssim(kN)^{\f32-\f{10}p}.$$
As a result, it comes out
\begin{align*}
\sum_{k_2=k+1}^\infty
k_2N\bigl\|r^{2(1-\f5p)}\vt_{k_1}\ur_{k_2}\bigr\|_{L^{\f p2}_t(L^{\f p2})}
\lesssim (kN)^{\f32-\f{10}p}E_p^{\f2p}.
\end{align*}
By inserting the above estimate into \eqref{3.38b}, we obtain
\begin{align*}
\Psi_k^{(4,2,2,2)}\lesssim(kN)^{1-\f9p+\e}E_p^{\f2p}
\bigl\|\wt\nabla \bigl(r^{\f{p-3}2}
|\wt u_k|^{\f p2}\bigr)\bigr\|_{L^2_t(L^2)}
\bigl\|r^{\f{p-3}2}|\wt u_{k}|^{\f p2}\bigr\|_{L^\infty_t(L^2)}^{\f2p}
\bigl\|r^{\f{p-5}2}|\wt u_{k}|^{\f p2}\bigr\|_{L^2_t(L^2)}^{1-\f4p},
\end{align*}
from which, we infer
\begin{align*}
N^{-2\al_p}\Psi_1^{(4,2,2,2)}
+\sum_{k\geq2}(kN)^{2\beta_p} \Psi_k^{(4,2,2,2)}
&\lesssim\Bigl(N^{-\f5p-\f2p\al_p+\e}+\Bigl(
\sum_{k\geq2}(kN)^{(-\f5p+\f2p\beta_p+\e)p}\Bigr)^{\f1p}\Bigr)E_p^{1+\f1p}\\
&\lesssim E_p^{1+\f1p}.
\end{align*}
Together with \eqref{3.36}-\eqref{3.38},
we conclude the proof of \eqref{3.39} with $j=2,3$.
\end{proof}

\begin{proof}[{\bf Proof of \eqref{3.39} with $j=4$}] In view of \eqref{S2eq6d},
we  get, by using H\"older's inequality, that
\begin{equation}\begin{split}\label{3.40}
\Psi_k^{(4,4)}
&\lesssim\bigl\|r^{3-\f7p}\wt\nabla\cL_{kN}^{-1}\cA_k^{(4,4)}\bigr\|_{L_t^{\f p2}(L^p)}
\bigl\|r^{\f{p-3}2}|\wt u_k|^{\f p2}\bigr\|_{L^\infty_t(L^2)}^{\f2p}
\bigl\|r^{\f{p-5}2}|\wt u_k|^{\f p2}\bigr\|_{L^2_t(L^2)}^{2-\f4p}.
\end{split}\end{equation}
It follows from \eqref{S2eq6} and Proposition \ref{lemLmp} that for sufficiently small $\e>0$,
\beq\label{3.40a}
\begin{split}
\bigl\|&r^{3-\f7p}\wt\nabla\cL_{kN}^{-1}\cA_k^{(4,4)}\bigr\|_{L_t^{\f p2}(L^p)}
\lesssim kN\sum_{(k_1,k_2)\in\Omega_k}k_2N\bigl\|r^{3-\f7p}\wt\nabla\cL_{kN}^{-1}
\bigl(\f1r\times\f{\vt_{k_1}\vt_{k_2}}r\bigr)\bigr\|_{L_t^{\f p2}(L^p)}\\
&\lesssim_\e(kN)^{\f12+\f1p+\e}\sum_{(k_1,k_2)\in\Omega_k}k_2N
\bigl\|r^{\f{p-5}p}\vt_{k_1}\bigr\|_{L^p_t(L^p)}
\bigl\|r^{\f{p-5}p}\vt_{k_2}\bigr\|_{L^p_t(L^p)}\\
&\lesssim_\e(kN)^{\f12+\f1p+\e}\sum_{\substack{(k_1,k_2)\in\Omega_k\\k_2\geq k_1}}
k_2N\bigl\|r^{\f{p-5}p}\vt_{k_1}\bigr\|_{L^p_t(L^p)}
\bigl\|r^{\f{p-5}p}\vt_{k_2}\bigr\|_{L^p_t(L^p)},
\end{split}\eeq
where we  used symmetry of the indices $k_1,k_2$ in $\Omega_k$ in the last step.

When $k=1$ in \eqref{3.40a}, we have
\begin{equation}\begin{split}\label{3.41}
\bigl\|&r^{3-\f7p}\wt\nabla\cL_{N}^{-1}\cA_1^{(4,4)}\bigr\|_{L_t^{\f p2}(L^p)}
\lesssim_\e N^{\f12+\f1p+\e}\sum_{j=2}^\infty jN
\bigl\|r^{\f{p-5}p}\vt_{j-1}\bigr\|_{L^p_t(L^p)}
\bigl\|r^{\f{p-5}p}\vt_{j}\bigr\|_{L^p_t(L^p)}\\
&\lesssim_\e N^{\f12+\f1p+\e}
\Bigl(\sum_{j=2}^\infty(jN)^{\f p2+2+2\beta_p}
\bigl\|r^{\f{p-5}p}\vt_{j}\bigr\|_{L^p_t(L^p)}^p\Bigr)^{\f1p}\\
&\quad\times\Bigl(N^{\f p2+2-2\al_p}
\bigl\|r^{\f{p-5}p}\vt_1\bigr\|_{L^p_t(L^p)}^p
+\sum_{j=3}^\infty \bigl((j-1)N\bigr)^{\f p2+2+2\beta_p}
\bigl\|r^{\f{p-5}p}\vt_{j-1}\bigr\|_{L^p_t(L^p)}^p\Bigr)^{\f1p}\\
&\quad\times\Bigl(N^{-\f4p-\f2p\beta_p+\f2p\al_p}
+\Bigl(\sum_{j=2}^\infty (jN)^{-\f{4(1+\beta_p)}{p-2}}\Bigr)^{1-\f2p}\Bigr)\\
&\lesssim_\e N^{\f12-\f3p-\f2p\beta_p+\f2p\al_p+\e}E_p^{\f2p}.
\end{split}\end{equation}
Similarly, when $k=2$ in \eqref{3.40a}, we have
\begin{equation}\begin{split}\label{3.42}
\bigl\|r^{3-\f7p}\wt\nabla\cL_{N}^{-1}\cA_2^{(4,4)}\bigr\|_{L_t^{\f p2}(L^p)}
&\lesssim_\e N^{\f12+\f1p+\e}\Bigl(N\bigl\|r^{\f{p-5}p}\vt_1\bigr\|_{L^p_t(L^p)}
\bigl\|r^{\f{p-5}p}\vt_1\bigr\|_{L^p_t(L^p)}\\
&\qquad\qquad\qquad+\sum_{j=3}^\infty jN
\bigl\|r^{\f{p-5}p}\vt_{j-2}\bigr\|_{L^p_t(L^p)}
\bigl\|r^{\f{p-5}p}\vt_{j}\bigr\|_{L^p_t(L^p)}\Bigr)\\
&\lesssim_\e N^{\f12-\f3p+\f4p\al_p+\e}E_p^{\f2p}.
\end{split}\end{equation}
And for the case when $k\geq3$ in \eqref{3.40a}, there holds
\begin{align*}
\bigl\|r^{3-\f7p}\wt\nabla\cL_{N}^{-1}\cA_k^{(4,4)}\bigr\|_{L_t^{\f p2}(L^p)}
\lesssim_\e& N^{\f12+\f1p+\e}\Bigl((k\pm1)N\bigl\|r^{\f{p-5}p}\vt_1\bigr\|_{L^p_t(L^p)}
\bigl\|r^{\f{p-5}p}\vt_{k\pm1}\bigr\|_{L^p_t(L^p)}\\
&+\sum_{j=[\f{k+1}2]}^{k-2} jN
\bigl\|r^{\f{p-5}p}\vt_{k-j}\bigr\|_{L^p_t(L^p)}
\bigl\|r^{\f{p-5}p}\vt_{j}\bigr\|_{L^p_t(L^p)}\\
&+\sum_{j=k+2}^{\infty} jN
\bigl\|r^{\f{p-5}p}\vt_{j-k}\bigr\|_{L^p_t(L^p)}
\bigl\|r^{\f{p-5}p}\vt_{j}\bigr\|_{L^p_t(L^p)}\Bigr)\eqdefa \sum_{\ell=1}^3\cF_k^\ell.
\end{align*}
$\cF_k^2$ actually vanishes when $k=3$,
and when $k\geq4$, we have
\begin{align*}
\cF_k^2
\lesssim &\Bigl(\sum_{j=[\f{k+1}2]}^{k-2} (jN)^{\f p2+2+2\beta_p}
\bigl\|r^{\f{p-5}p}\vt_j\bigr\|_{L^p_t(L^p)}^p\Bigr)^{\f1p}
\Bigl(\sum_{j=[\f{k+1}2]}^{k-2} \bigl((k-j)N\bigr)^{\f p2+2+2\beta_p}
\bigl\|r^{\f{p-5}p}\vt_{k-j}\bigr\|_{L^p_t(L^p)}^p\Bigr)^{\f1p}\\
&\times\Bigl(\sum_{j=[\f{k+1}2]}^{k-2}(jN)^{\left(\f12-\f2p-\f2p\beta_p\right)\f{p}{p-2}}
\bigl((k-j)N\bigr)^{-\left(\f12+\f2p+\f2p\beta_p\right)\f{p}{p-2}}\Bigr)^{1-\f2p}\\
\lesssim&(kN)^{\f12-\f2p-\f2p\beta_p}N^{-\f12-\f2p-\f2p\beta_p}E_p^{\f2p},
\end{align*}
where in the last step, we  used the facts that
$$-\left(\f12+\f2p+\f2p\beta_p\right)\f{p}{p-2}
<-\left(\f12+\f4p\right)\f{p}{p-2}<-\f74,$$
and for any
$[\f{k+1}2]\leq j\leq k-2$, there holds $jN\thicksim kN$, so that
\begin{align*}
&\Bigl(\sum_{j=[\f{k+1}2]}^{k-2}(jN)^{\left(\f12-\f2p-\f2p\beta_p\right)\f{p}{p-2}}
\bigl((k-j)N\bigr)^{-\left(\f12+\f2p+\f2p\beta_p\right)\f{p}{p-2}}\Bigr)^{1-\f2p}\\
&\quad\thicksim(kN)^{\f12-\f2p-\f2p\beta_p}
\Bigl(\sum_{j=[\f{k+1}2]}^{k-2}
\bigl((k-j)N\bigr)^{(-\f12-\f2p-\f2p\beta_p)\cdot\f{p}{p-2}}\Bigr)^{1-\f2p}
\lesssim(kN)^{\f12-\f2p-\f2p\beta_p}N^{-\f12-\f2p-\f2p\beta_p}.
\end{align*}
While we observe that
\begin{align*}
\cF_k^3
\lesssim &\Bigl(\sum_{j=k+2}^\infty (jN)^{\f p2+2+2\beta_p}
\bigl\|r^{\f{p-5}p}\vt_j\bigr\|_{L^p_t(L^p)}^p\Bigr)^{\f1p}\\
&\times\Bigl(\sum_{j=k+2}^\infty \bigl((j-k)N\bigr)^{\f p2+2+2\beta_p}
\bigl\|r^{\f{p-5}p}\vt_{j-k}\bigr\|_{L^p_t(L^p)}^p\Bigr)^{\f1p}\\
&\times\Bigl(\sum_{j=k+2}^\infty(jN)^{\left(\f12-\f2p-\f2p\beta_p\right)\f{p}{p-2}}
\bigl((j-k)N\bigr)^{-\left(\f12+\f2p+\f2p\beta_p\right)\f{p}{p-2}}\Bigr)^{1-\f2p}\\
\lesssim&(kN)^{\f12-\f2p-\f2p\beta_p}N^{-\f12-\f2p-\f2p\beta_p}E_p^{\f2p}.
\end{align*}
As a result, we deduce  that for $k\geq3$
\begin{equation}\begin{split}\label{3.43}
\bigl\|r^{3-\f7p}\wt\nabla\cL_{N}^{-1}\cA_k^{(4,4)}\bigr\|_{L_t^{\f p2}(L^p)}
&\lesssim_\e N^{-\f1p+\f2p\al_p+\e}(k\pm1)N
\bigl\|r^{\f{p-5}p}\vt_{k\pm1}\bigr\|_{L^p_t(L^p)}E_p^{\f1p}\\
&\qquad+(kN)^{\f12-\f2p-\f2p\beta_p}N^{-\f1p-\f2p\beta_p+\e}E_p^{\f2p}.
\end{split}\end{equation}

By substituting the estimates \eqref{3.41}-\eqref{3.43} into \eqref{3.40}, we  obtain
\begin{align*}
N^{-2\al_p}\Psi_1^{(4,4)}+N^{2\beta_p}\Psi_2^{(4,4)}
\lesssim_\e& N^{\f12-\f3p+\f4p\al_p+\e}E_p^{\f2p}
 N^{-2+\f4p+\f2p\beta_p}\\
&\times\Bigl(\bigl(N^{-\al_p}\bigl\|r^{\f{p-3}2}|\wt u_1|^{\f p2}
\bigr\|_{L^\infty_t(L^2)}\bigr)^{\f2p}
\bigl(N^{1-\al_p}\bigl\|r^{\f{p-5}2}|\wt u_1|^{\f p2}
\bigr\|_{L^2_t(L^2)}\bigr)^{2-\f4p}\\
&+\bigl(N^{\beta_p}\bigl\|r^{\f{p-3}2}|\wt u_2|^{\f p2}
\bigr\|_{L^\infty_t(L^2)}\bigr)^{\f2p}
\bigl(N^{1+\beta_p}\bigl\|r^{\f{p-5}2}|\wt u_2|^{\f p2}
\bigr\|_{L^2_t(L^2)}\bigr)^{2-\f4p}\Bigr)\\
\lesssim&E_p^{1+\f2p},
\end{align*}
and
\begin{align*}
\sum_{k\geq3}(kN)^{2\beta_p} \Psi_k^{(4,4)}
\lesssim_\e&\sum_{k\geq3}(kN)^{-\f32+\f2p}
\Bigl\{N^{-\f1p-\f2p\beta_p+\e}E_p^{\f2p}\\
&+N^{-\f1p+\f2p\al_p+\e}
\Bigl(\bigl((k\pm1)N\bigr)^{\f p4+1+\beta_p}\bigl\|r^{\f{p-5}2}|\vt_{k\pm1}|^{\f p2}
\bigr\|_{L^p_t(L^p)}\Bigr)^{\f2p}E_p^{\f1p}\Bigr\}\\
&\times\Bigl((kN)^{\beta_p}
\bigl\|r^{\f{p-3}2}|\wt u_k|^{\f p2}\bigr\|_{L^\infty_t(L^2)}\Bigr)^{\f2p}
\Bigl((kN)^{1+\beta_p}\bigl\|r^{\f{p-5}2}|\wt u_k|^{\f p2}\bigr\|_{L^2_t(L^2)}
\Bigr)^{2-\f4p}\\
\lesssim_\e&N^{-\f1p+\f2p\al_p+\e}\Bigl\{
\Bigl(\sum_{k\geq3}(kN)^{(-\f32+\f2p)p}\Bigr)^{\f1p}
+\sup_{k\geq3}(kN)^{-\f32+\f2p}\Bigr\}E_p^{1+\f1p}\\
\lesssim&E_p^{1+\f1p}.
\end{align*}
This leads to \eqref{3.39} with $j=4$. \end{proof}

\medskip

\section*{Acknowledgments}
Y. Liu is supported by NSF of China under grant 12101053.
Ping Zhang is supported by National Key R$\&$D Program of China under grant
  2021YFA1000800 and K. C. Wong Education Foundation.
 He is also partially supported by National Natural Science Foundation of China under Grants  12288201 and 12031006.
\end{document}